\documentclass[a4paper,10pt]{amsart}

\title[Fixed Points of Parking Functions]{Fixed Points of Parking Functions}


\author[J.~McCammond]{Jon McCammond}
\address[J.~McCammond]{University of California, Santa Barbara}
\email{jon.mccammond@math.ucsb.edu}

\author[H.~Thomas]{Hugh Thomas}
\address[H.~Thomas]{LaCIM, Universit\'e du Qu\'ebec \`a Montr\'eal}
\email[H.~Thomas]{hugh.ross.thomas@gmail.com}

\author[N.~Williams]{Nathan Williams}
\address[N.~Williams]{University of Texas at Dallas}
\email{nathan.f.williams@gmail.com}

\usepackage{amsmath, amsfonts, amssymb, amsthm, mathtools, thmtools,enumerate,bbm,graphics, graphicx, xcolor,frcursive,xparse,array,colortbl,multicol,comment}
\usepackage[all]{xy}
\usepackage{hyperref}
\usepackage{diagbox}
\usepackage{url, hypcap}
\hypersetup{colorlinks=true, citecolor=darkblue, linkcolor=darkblue}

\usepackage{tikz}
\usetikzlibrary{calc,through,backgrounds,shapes,matrix,fit}
\usepackage{verbatim}   
\newif\ifshowtikz
\showtikztrue

\newcommand{\margincolor}{red}      
\definecolor{darkgreen}{rgb}{0,0.7,0}

\newcounter{margincounter}
\newcommand{\marginnum}{
\ifnum\value{margincounter}<10
\textcolor{\margincolor}{\begin{picture}(0,0)\put(2.2,2.4){\circle{9}}\end{picture}\footnotesize\arabic{margincounter}}
\else\ifnum\value{margincounter}<100
\textcolor{\margincolor}{\begin{picture}(0,0)\put(4.256,2.5){\circle{11}}\end{picture}\footnotesize\arabic{margincounter}}
\else
\textcolor{\margincolor}{\begin{picture}(0,0)\put(6.8,2.5){\circle{14}}\end{picture}\footnotesize\arabic{margincounter}}
\fi\fi
}

\let\oldtikzpicture\tikzpicture
\let\oldendtikzpicture\endtikzpicture

\renewenvironment{tikzpicture}{%
    \ifshowtikz\expandafter\oldtikzpicture%
    \else\comment%
    \fi
}{%
    \ifshowtikz\oldendtikzpicture%
    \else\endcomment%
    \fi
}

\definecolor{darkblue}{rgb}{0.0,0,0.7}
\newcommand{\darkblue}{\color{darkblue}}
\definecolor{darkred}{rgb}{0.7,0,0}
 
\definecolor{lightgrey}{rgb}{0.7,0.7,0.7}

\theoremstyle{plain}
\newtheorem{theorem}{Theorem}[section]
\newtheorem{proposition}[theorem]{Proposition}
\newtheorem{corollary}[theorem]{Corollary}
\newtheorem{lemma}[theorem]{Lemma}

\theoremstyle{definition}
\newtheorem{definition}[theorem]{Definition}
\newtheorem{example}[theorem]{Example}

\newtheorem{remark}[theorem]{Remark}
\usepackage[nameinlink]{cleveref}
\usepackage[all]{xy}
\usepackage[T1]{fontenc}
\crefformat{footnote}{#2\footnotemark[#1]#3}
\crefformat{conjecture}{Conjecture~#2#1#3}

\newcommand{\defn}[1]{\emph{\darkblue #1}}
\newcommand{\w}{{\sf w}}

\newcommand{\pp}{{\sf p}}
\newcommand{\qq}{{\sf q}}

\newcommand{\swg}{{\sf sweep}}

\definecolor{keywords}{RGB}{255,0,90}
\definecolor{comments}{RGB}{0,0,113}
\definecolor{red}{RGB}{160,0,0}
\definecolor{green}{RGB}{0,150,0}
\newcommand{\Dyck}{{\mathcal{DF}}}
\newcommand{\Fix}{{\mathrm{Fix}}}

\newcommand{\DyckW}{{\mathcal{DW}}}
\newcommand{\Bal}{{\mathcal{BF}}}
\newcommand{\Park}{{\mathcal{PW}}}
\newcommand{\PF}{{\mathcal{PT}}}
\newcommand{\BF}{{\mathcal{BT}}}
\newcommand{\FF}{{\mathcal{F}}}
\newcommand{\EF}{{\widetilde{\mathcal{F}}}}

\newcommand{\idl}{{\mathfrak{i}}}
\newcommand{\del}{{\mathfrak{d}}}
\newcommand{\bal}{{\mathfrak{b}}}
\newcommand{\prk}{{\mathfrak{p}}}

\newcommand{\So}{{\mathcal{S}}}
\newcommand{\Sym}{{\mathfrak{S}}}
\newcommand{\ASym}{{\widetilde{\mathfrak{S}}}}
\newcommand{\aw}{{w}}
\newcommand{\Astar}{A}
\newcommand{\Bstar}{B}
\newcommand{\A}{{\mathcal{A}}}
\newcommand{\ZZ}{{\mathbb{Z}}}
\newcommand{\CC}{{\mathbb{C}}}
\newcommand{\RR}{{\mathbb{R}}}

\newcommand{\xx}{{\mathbf{x}}}
\newcommand{\yy}{{\mathbf{y}}}
\newcommand{\zz}{{\mathbf{z}}}

\newcommand{\maj}{{\mathsf{maj}}}
\newcommand{\inv}{{\mathsf{inv}}}
\newcommand{\area}{{\mathsf{area}}}
\newcommand{\dinv}{{\mathsf{dinv}}}
\newcommand{\bounce}{{\mathsf{bounce}}}

\renewcommand{\mod}{\operatorname{mod}}

\usepackage{todonotes}

\begin{document}

\begin{abstract}
We define an action of words in $[m]^n$ on $\RR^m$ to give a new characterization of rational parking functions---they are exactly those words whose action has a fixed point.  We use this viewpoint to give a simple definition of Gorsky, Mazin, and Vazirani's zeta map on rational parking functions when $m$ and $n$ are coprime~\cite{gorsky2016affine}, and prove that this zeta map is invertible.  A specialization recovers Loehr and Warrington's sweep map on rational Dyck paths~\cite{armstrong2015sweep,thomas2015sweeping,gorsky2017rational}.
\end{abstract}

\maketitle

\section{Introduction}
\subsection{Parking Words}
\emph{Let $m$ and $n$ be positive integers, not necessarily coprime.}  Classical parking words have a well-known interpretation in the language of parking cars.  There are $n$ parking places and $n$ cars, each indexed  from $0$ to $n{-}1$.  As in~\cite[Section 6]{konheim1966occupancy}, car $i$ has a preference for parking place $\pp_i$, and cars attempt to park as follows: for $0 \leq i \leq n-1$, car $i$ takes the unoccupied parking place with the lowest number larger than or equal to $\pp_i$, should such a parking place exist.
The \defn{classical parking words} $\Park_n$ are defined as those words for which each car is able to park.

The 16 parking words in $\Park_3$ are given on the left side of~\Cref{fig:parkwords}.  Garsia introduced a combinatorial interpretation of $\Park_n$ as certain super-diagonal labelled paths in an $n {\times} n$ square, which has served as the basis of many subsequent investigations.  Replacing the square by an $m \times n$ rectangle gives the \defn{(m,n)-parking words} $\Park_{m}^n$---those words \begin{align}\begin{split}\pp = \pp_0 \cdots \pp_{n-1} \in [m]^n &:= \{0,1,\ldots,m{-}1\}^n  \text{ such that } \\  \Big| \big\{ j : \pp_j < i \big\}\Big| &\geq \frac{in}{m} \text{ for } 1 \leq i \leq m.\label{eq:parking_def}\end{split}\end{align}  The classical parking words are recovered as $\Park_n = \Park_{n{+}1}^{n}$.  The 25 parking words in $\Park_5^3$ are illustrated on the right side of~\Cref{fig:parkwords}. 

\begin{figure}[htb]
\begin{tabular}{ccccc}
000 & 001 & 002 & 011 & 012\\
    & 010 & 020 & 101 & 021\\
    & 100 & 200 & 110 & 102\\
    &     &     &     & 120\\
    &     &     &     & 201\\
    &     &     &     & 210
\end{tabular}\hfill
\begin{tabular}{ccccccc}
000 & 001 &    002 &    003 &    011 &    012 &    013\\
    & 010 &    020 &    030 &    101 &    021 &    031\\
    & 100 &    200 &    300 &    110 &    201 &    301\\
    &     &        &        &        &    210 &    310\\
    &     &        &        &        &    120 &    130\\
    &     &        &        &        &    102 &    103
\end{tabular}
\caption{Left: the $16$ $(4,3)$-parking words in $\Park_3$ (these are also the $(3,3)$-parking words).  Right: the $25$ $(5,3)$-parking words in $\Park_5^3$.  Each column is an orbit under $\mathfrak{S}_3$.}
\label{fig:parkwords}
\end{figure}

\subsection{A New Characterization of Parking Words}
Our main result is a new characterization of $(m,n)$-parking words as \emph{piecewise-linear functions} from $\mathbb{R}^m$ to $\mathbb{R}^m$.  This characterization is new even for classical parking words.  Define \begin{align*} V_0^m := \Big\{ \xx = (x_0,\ldots,x_{m-1}) \in \mathbb{R}^m : \sum_{i=0}^{m-1} x_i = 0 \text{ and } x_0\leq x_1 \leq \cdots \leq x_{m-1} \Big\}. \end{align*}  A letter $i \in [m]$ acts on $\mathbf{x} \in V_0^m$ by adding $m$ to $x_i$, subtracting the tuple $\mathbbm{1}_m:=(1,1,\ldots,1)$, and then resorting.  A word $\w \in [m]^n$ acts on $\mathbf{x} \in V_0^m$  by acting by its letters from left to right.  The following theorem distinguishes parking words in $[m]^n$ by their action on $V_0^m$.

\begin{theorem}
\label{thm:park_char}
The action of $\w \in [m]^n$ on $V_0^m$ has a fixed point if and only if $\w$ is an $(m,n)$-parking word.  More precisely,
the action of $\w \in [m]^n$ on $V_0^m$:
\begin{itemize}
    \item has a {\bf unique} fixed point iff $\w \in \Park_m^n$ and $\gcd(m,n)=1$;
    \item has {\bf infinitely many} fixed points iff $\w \in \Park_m^n$ and $\gcd(m,n)> 1$; and
    \item has {\bf no} fixed points iff $\w \in [m]^n \setminus \Park_m^n$.
\end{itemize}
\end{theorem}

The motivation for~\Cref{thm:park_char} comes from generalizations of the space of diagonal coinvariants and the zeta map on parking functions, as we now explain.

\subsection{Coinvariants and the Symmetric Group}  The Hilbert series for the space of coinvariants is the generating function for two important statistics on the $n!$ permutations in $\Sym_n$:
\begin{equation} \mathrm{Hilb}\left(\CC[\xx_n] / \langle \CC[\xx_n]_+^{\Sym_n} \rangle; q\right) = \sum_{w \in \Sym_n} q^{\inv(w)} = \sum_{w \in \Sym_n} q^{\maj(w)},\label{eq:coinv}\end{equation} where $\CC[\xx_n]$ is shorthand for a polynomial ring in $n$ variables and $\langle \CC[\xx_n]_+^{\Sym_n} \rangle$ is the ideal of $\CC[\xx_n]$  generated by symmetric polynomials with no constant term.

Artin gave a basis for this space using the code of a permutation to reflect the first generating function of~\Cref{eq:coinv}~\cite{artin1944galois}, while Garsia and Stanton found a basis using the descents of a permutation to explain the second~\cite{garsia1984group}. 

A statistic with the same distribution as $\inv$ or $\maj$ is eponymously named \defn{mahonian}~\cite{macmahon1913indices}, but Foata gave the first bijection sending one statistic to the other~\cite{foata1968netto}.  Exploiting the fact that this bijection preserves descents of the inverse permutation, Foata and Sch\"utzenberger later found an involution that \emph{interchanges} $\inv$ and $\maj$~\cite{foata1978major}.

\subsection{Diagonal Coinvariants}  The study of the space of diagonal coinvariants originated with Garsia and Haiman; its relationship to parking words was first suggested by Gessel~\cite{haiman1994conjectures,garsia1996remarkable}.  More precisely, Carlsson and Mellit's recent proof of the shuffle conjecture~\cite{haglund2005combinatorial,carlsson2015proof,haglund2017lecture} implies the long-suspected fact that the bigraded Hilbert series of the space of diagonal coinvariants is encoded as a positive sum over the $(n{+}1)^{n{-}1}$ parking words $\Park_n$~\cite{haiman2002vanishing,haglund2005conjectured}:\footnote{Carlsson and Mellit actually proved a stronger result, giving an explicit formula for the \emph{Frobenius series} for the space of diagonal coinvariants.}

\begin{equation} \mathrm{Hilb}\left(\CC[\xx_n,\mathbf{y}_n] / \langle \CC[\xx_n,\mathbf{y}_n]_+^{\Sym_n} \rangle; q,t\right) = \sum_{\pp \in \Park_n} q^{\dinv(\pp)} t^{\area(\pp)},\label{eq:diagonal_coinvariants}\end{equation} where $q$ records the degree of the variables $\xx_n$, $t$ the degree of $\mathbf{y}_n$, and $\area$ and $\dinv$ are certain statistics on parking functions.  Recently, Carlsson and Oblomkov artfully merged the Artin and Garsia-Stanton bases to give an explicit basis of the space of diagonal coinvariants~\cite{carlsson2018affine}, explaining the generating function in~\Cref{eq:diagonal_coinvariants}.

It is known from~\Cref{eq:diagonal_coinvariants} that $\area$ and $\dinv$ are symmetric, i.e.,
   \begin{equation}\sum_{\pp \in \Park_n} q^{\area(\pp)} t^{\dinv(\pp)} = \sum_{\pp \in \Park_n} q^{\dinv(\pp)} t^{\area(\pp)}.\label{eq:qtsym}\end{equation}
However, it is a long-standing open problem to find an involution that interchanges $\area$ and $\dinv$---in the style of Foata and Sch\"utzenberger's involution for $\inv$ and $\maj$---thus combinatorially proving~\Cref{eq:qtsym}.
     This problem is still wide open, even for the alternating subspace~\cite{carleton,gillespie2016combinatorial}.  As a first step towards this elusive involution, the equidistribution of $\dinv$ and $\area$---obtained by setting $t=1$ in~\Cref{eq:qtsym}---was proven combinatorially by Loehr and Remmel~\cite{loehr2004conjectured,haglund2005conjectured}~\cite[Corollary 5.6.1]{haglund2008q}:
\begin{theorem}[{\cite{loehr2004conjectured}}]\label{thm:haglund_loehr}
For $n\geq 1$,
	\[\sum_{\pp \in \Park_n} q^{\area(\pp)} = \sum_{\pp \in \Park_n} q^{\dinv(\pp)}.\]\end{theorem}

This bijection on $\Park_n$ takes $\area$ to $\dinv$, combinatorially proving the symmetry of~\Cref{thm:haglund_loehr}.  It was first understood, generalized, and inverted for the alternating subspace, where it was called the \defn{zeta map}~\cite{krattenthaler2002enumeration,haglund2003conjectured,haglund2005conjectured,garsia2002proof,haglund2008q,armstrong2015sweep,thomas2015sweeping}.  It has been rediscovered many times.  We review the history of the zeta map in~\Cref{sec:zeta_history}.

\subsection{Rational Parking Words and the Affine Symmetric Group}
\emph{We now assume $m$ and $n$ are coprime}.  The classical parking words $\Park_n$, their statistics $\area$ and $\dinv$, and the shuffle conjecture have all been (at least combinatorially) generalized to the $(m,n)$-parking words $\Park_{m}^n$~\cite{bergeron2015compositional,armstrong2016rational, gorsky2016affine, gorsky2015refined, thiel2016anderson,gorsky2017rational}.

The Fuss $(nk{+}1,n)$ generalization of the story of diagonal coinvariants is due to Garsia and Haiman~\cite{haiman1998t,garsia1996remarkable}.  Writing $\mathcal{A}$ for the ideal generated by the alternating polynomials in $\CC[\xx_n,\mathbf{y}_n]$, Mellit proved the rational shuffle conjecture in ~\cite{mellit2016toric}, which implies that \[\text{Hilb}\left(\mathcal{A}^{k-1}/\mathcal{A}^{k-1} \CC[\xx_n,\mathbf{y}_n]_+^{\Sym_n}; q,t\right) = \sum_{\pp \in \Park_{kn{+}1}^n} q^{\area(\pp)} t^{\dinv(\pp)}. \]

The more general rational $(m,n)$ version comes from Hikita's study of the Borel-Moore homology of affine type $A$ Springer fibers, which has a natural basis indexed by the $m^{n-1}$ elements of the affine symmetric group $\widetilde{\Sym}_n$ lying inside an $m$-fold dilation of the fundamental alcove~\cite{shi1987sign,haiman1994conjectures,cellini2002ad,cherednik2003double,sommers2003b,hikita2014affine,gorsky2016affine,thiel2016anderson}.  Thus, while the space of coinvariants $\CC[\xx_n] / \langle \CC[\xx_n]_+^{\Sym_n} \rangle$ is related to the symmetric group $\Sym_n$, the diagonal coinvariants are related to the affine symmetric group $\ASym_n$.

There are many bijections from these affine elements to the parking words $\Park_{m}^n$.  Armstrong found natural interpretations of $\area$ and $\dinv$ in terms of affine permutations for the Fuss case~\cite{armstrong2013hyperplane}, and his work was extended to the rational case by Gorsky, Mazin, and Vazirani~\cite{gorsky2016affine,gorsky2017rational}.  Gorsky and Negut formulated the rational shuffle conjecture in~\cite{gorsky2015refined}---that Hikita's polynomial was given by an operator from an elliptic Hall algebra (see also~\cite{bergeron2015compositional}).  This operator formulation leads to a $q,t$-symmetric bivariate polynomial generalizing~\Cref{eq:qtsym}:\footnote{Something is lost in the rational case: one statistic remains the degree, but the second statistic now appears only using a filtration.}
\begin{equation}\sum_{\pp \in \Park_m^n} q^{\area(\pp)} t^{\dinv(\pp)} = \sum_{\pp \in \Park_m^n} q^{\dinv(\pp)} t^{\area(\pp)}. \label{eq:parkinghilb}\end{equation}

As a combinatorial proof of $q,t$-symmetry seems out of reach even in the classical $m=n{+}1$ case, the next best thing is the analogue of the equidistribution of~\Cref{thm:haglund_loehr}.  To this end, Gorsky, Mazin, and Vazirani defined a zeta map on $\Park_m^n$ (a map taking $\area$ to $\dinv$), and conjectured that it was a bijection by providing what they believed to be an inverse map.  A curious feature of their conjectural inverse is that it appears to \emph{converge} to the correct answer.

As a corollary to our~\Cref{thm:park_char}, we prove Gorsky, Mazin, and Vazirani's conjecture and obtain a rational generalization of~\Cref{thm:haglund_loehr}.

\begin{theorem}    \label{thm:rat_sym}  For $m$ and $n$ relatively prime,
	\[\sum_{\pp \in \Park_{m}^n} q^{\area(\pp)} = \sum_{\pp \in \Park_{m}^n} q^{\dinv(\pp)}.\]
\end{theorem}

\subsection{Outline of the Paper}

In~\Cref{sec:new_char} we define $(m,n)$-parking words, the action of words in $[m]^n$ on $\mathbb{R}^m$, and prove our characterizations in~\Cref{thm:park_char} using the Brouwer fixed point theorem.

To relate this characterization to parking functions, we introduce some notation.  Fixing $(m,n)$ relatively prime, we define $(m,n)$-filters as certain periodic filters of $\ZZ\times \ZZ$ in~\Cref{sec:filters}, and show that equivalence classes of these filters are naturally parameterized by rational $(m,n)$-Dyck paths and balanced $(m,n)$-filters.  We  define $(m,n)$-filter tuples in~\Cref{sec:tuples} as certain sequences of $(m,n)$-filters, and relate these sequences to labeled $(m,n)$-Dyck paths.

The notion of $(m,n)$-filters allows us to give a new, remarkably simple definition of the zeta map on $(m,n)$-parking words in~\Cref{sec:combinatorial_zeta}.  We summarize past work on zeta maps in~\Cref{sec:zeta_history}, define the zeta map in~\Cref{sec:zeta}, and relate our construction to Loehr and Warrington's sweep map on $(m,n)$-Dyck paths in~\Cref{sec:sweep}.

In~\Cref{sec:affine_symmetric_and_regions}, we finally turn to the affine symmetric group.  After basic definitions in~\Cref{sec:affine_symmetric}, we use balanced $(m,n)$-filters to give a bijection between $(m,n)$-filter tuples and affine permutations whose inverses lie in the Sommers region in~\Cref{sec:sommers}.  We use this bijection in~\Cref{sec:sommers_parking_words} to relate our constructions to the work of Gorsky, Mazin, and Vazirani, showing that our~\Cref{thm:park_char} resolves~\cite[Conjecture 1.4]{gorsky2016affine}.

\section{Words and Actions}
\label{sec:words_and_actions}

\subsection{Parking Words}
\emph{Let $m$ and $n$ be positive integers, not necessarily coprime.}  As in the introduction, we define the \defn{(m,n)-parking words} $\Park_{m}^n$ to be those words $\pp = \pp_0 \cdots \pp_{n-1} \in [m]^n$ such that \begin{equation} \Big| \big\{ j : \pp_j < i \big\}\Big| \geq \frac{in}{m} \text{ for } 1 \leq i \leq m.\label{eq:parking_def2}\end{equation}

By definition, any $(m,n)$-parking word is a permutation of the column lengths of a lattice path staying above the main diagonal in an $m \times n$ rectangle, as illustrated in~\Cref{fig:dyck_paths}.  (Here, by ``column lengths,'' we mean the distances between the top of the rectangle and the horizontal steps of the lattice path.) We write $\DyckW_m^n$ for the increasing $(m,n)$-parking words---the \defn{$(m,n)$-Dyck words}---which are in bijection with the set of such lattice paths.

\subsection{Hyperplanes}
\label{sec:hyperplane_and_functions}
Although we defer most of the connections between parking words and the affine symmetric group to \Cref{sec:affine_symmetric_and_regions}, we will require the hyperplane arrangement of the affine symmetric group $\widetilde{\mathfrak{S}}_{m}$ immediately. For $0 \leq i,j < m$ and $k\in \mathbb Z$, define the hyperplane \[\mathcal{H}^k_{i,j} = \{\mathbf{x} \in \RR^m : x_i-x_j=mk\}.\]
Observe that $\mathcal H^k_{i,j}=H^{-k}_{j-i}$. We define the height of $\mathcal H^k_{i,j}$ to be $j-i+mk$, where we assume that $k$ is positive or $k=0$ and $j>i$.
It follows that the \defn{affine simple hyperplanes} $\left\{\mathcal{H}^0_{i,i+1}\right\}_{0\leq i < m-1} \cup \left\{\mathcal{H}^1_{m-1,0}\right\}$ each have height one.  We call the set $\left\{\mathcal{H}^0_{i,i+1}\right\}_{0\leq i < m-1}$ the \defn{simple hyperplanes}.  Write \[\mathcal{H} =  \bigcup_{\substack{0 \leq i < j < m\\ k \in \mathbb{Z}}} \mathcal{H}^k_{i,j}\] for the affine hyperplane arrangement of type $\widetilde{\mathfrak{S}}_{m}$ and let \[\mathbb{R}^{m}_t = \left\{\xx \in \mathbb{R}^m : \sum_{i=0}^{m-1} x_i = t \right\} \cong \RR^{m-1}.\]   The closure of each connected region of $\RR^m_t \setminus \mathcal{H}$ is called an \defn{alcove}.  For $0 \leq i < m$, write $\mathbf{e}_i$ for the $i$th standard basis vector of $\RR^m$ and $\mathbbm{1}_m=\sum_{i=0}^{m-1} \mathbf{e}_i$.  The set of alcoves in $\RR^m_t$ is permuted under translations by $m\mathbf{e}_i-\mathbbm{1}_m$ and under reflections in any hyperplane $\mathcal{H}_{i,j}^k$.   There is an alcove-preserving bijection between $\RR_{t_1}^m$ and $\RR_{t_2}^m$, defined by the addition of the multiple $(t_2-t_1)\mathbbm{1}_m$; we call this \defn{rebalancing}.

We will need a metric on $\mathbb R^m_t$. This metric is a constant multiple of the usual Euclidian metric, but it will be convenient for us to describe it in a different way. To begin with, define:
\begin{align*} \mathcal N(\xx)&:={\sum_{0\leq i<j< m} (x_j-x_i)^2}\\
|\xx|&:=\xx \cdot \xx^\intercal=\sum_{i=0}^{m-1}x_i^2\end{align*}

Observe that $\mathcal N(\xx)=\mathcal N(\xx-t\mathbbm{1}_m)$ for any $t\in \mathbb R$. Thus, to understand the behaviour of $\mathcal N(\xx)$, it suffices to assume $\xx\in\mathbb R^m_0$.  Define a matrix $$N=(n_{ij})_{0\leq i,j<m} \textrm{ with } n_{ij}=\begin{cases} (m -1) & \text{if } i=j \\ -1 & \text{otherwise}\end{cases}$$ and write $\mathbbm{1}_{m \times m}$ for the $m\times m$ matrix containing all ones.  Then, for $\xx\in \mathbb R^m_0$, we can write \[\mathcal{N}(\xx) = \xx \cdot N \cdot \xx^\intercal= \xx \cdot \left(N +\mathbbm{1}_{m\times m}\right) \cdot \xx^\intercal = m (\xx \cdot \xx^\intercal) = m |\xx|.\]

Now, for $\xx,\yy\in\mathbb R^m_t$ set $d(\xx,\yy)=\mathcal N(\xx-\yy)^{1/2}$. Since $\xx-\yy\in \mathbb R^m_0$, we have that $\mathcal N(\xx-\yy)=m|\xx-\yy|$. It follows that $d(\xx,\zz)\leq d(\xx,\yy)+d(\yy,\zz)$, with equality only if $\yy$ is on the line segment between $\xx$ and $\zz$. (We refer to this statement, including the conditions for equality, as the ``strong triangle inequality.'')  
Since $d$ is a multiple of the usual Euclidean metric, the metric topology defined by $d$ is the same as the usual (metric) topology on
$\mathbb R^m_t$.


\begin{definition}
A fundamental domain for the natural action of $\mathfrak{S}_m$ on $\mathbb{R}^m$ is given by those points whose coordinates weakly increase.  Define the cone 
 \[V^m_t := \Big\{ \xx \in \mathbb{R}^m_t : x_0\leq x_1 \leq \cdots \leq x_{m-1} \Big\}.\]   We may rebalance an element of $V^m_{t_1}$ to an element of $V^m_{t_2}$ by adding the appropriate multiple of $\mathbbm{1}_m$.
\label{def:space_and_norm}
\end{definition}

\subsection{Actions of Words}
For each $i \in [m]$, we define piecewise linear transformations on $\RR^m_t\setminus \mathcal H$ and on $V^m_t$.

\begin{definition}\label{def:action}
A letter $i \in [m]$ acts on $\mathbf{x} \in \RR^m_t \setminus \mathcal{H}$ by adding $m$ to the $i$th smallest coordinate of $\mathbf{x}$ and subtracting the tuple $\mathbbm{1}_m$.  (This definition is unambiguous because we exclude the points of $\mathcal H$, which are exactly the points where there are some equal coordinates. The coordinates of a point $\xx\in\RR^m_t\setminus\mathcal H$ are all distinct, so it makes sense to speak of its $i$-th smallest coordinate.) The letter $i$ acts on $\mathbf{x} \in  V^m_t$ in the same way, but with a final resorting step at the end.  A word $\w \in [m]^n$ acts on $\mathbf{x} \in \RR^m_t \setminus \mathcal{H}$ or $\mathbf{x} \in V^m_t$  by acting by its letters from left to right.

More explicitly, writing $\mathbf{y}:=\mathrm{sort}(\mathbf{x})$ for the increasing rearrangement of a point $\mathbf{x} \in \RR^m_t$, define
 \begin{align}
 \label{eq:action_letter}
 \begin{split}
 i(\mathbf{x}) &:= \mathbf{x}+m\mathbf{e}_j-\mathbbm{1}_m \text{ for } \mathbf{x} \in \RR^m_t \setminus \mathcal{H} \text{ if } x_j={y}_i, \text{ and }\\
 i(\mathbf{x}) &:= \mathrm{sort}(\mathbf{x}+m\mathbf{e}_i-\mathbbm{1}_m) \text{ for } \mathbf{x} \in V^m_t.
 \end{split}
 \end{align}
\end{definition}

 An example of the action on $\RR^3_6 \setminus \mathcal{H}$ is given in~\Cref{fig:parking_on_perm_classes}.\footnote{To cleanly bridge from this section to the affine symmetric group in~\Cref{sec:affine_symmetric_and_regions}, we will want to normalize points so that $\sum_{i=0}^{m-1} x_i = \binom{m+1}{2}$.  We therefore use that normalization in~\Cref{fig:parking_on_perm_classes}.}

\begin{figure}[htb]
\includegraphics[width=.8\textwidth]{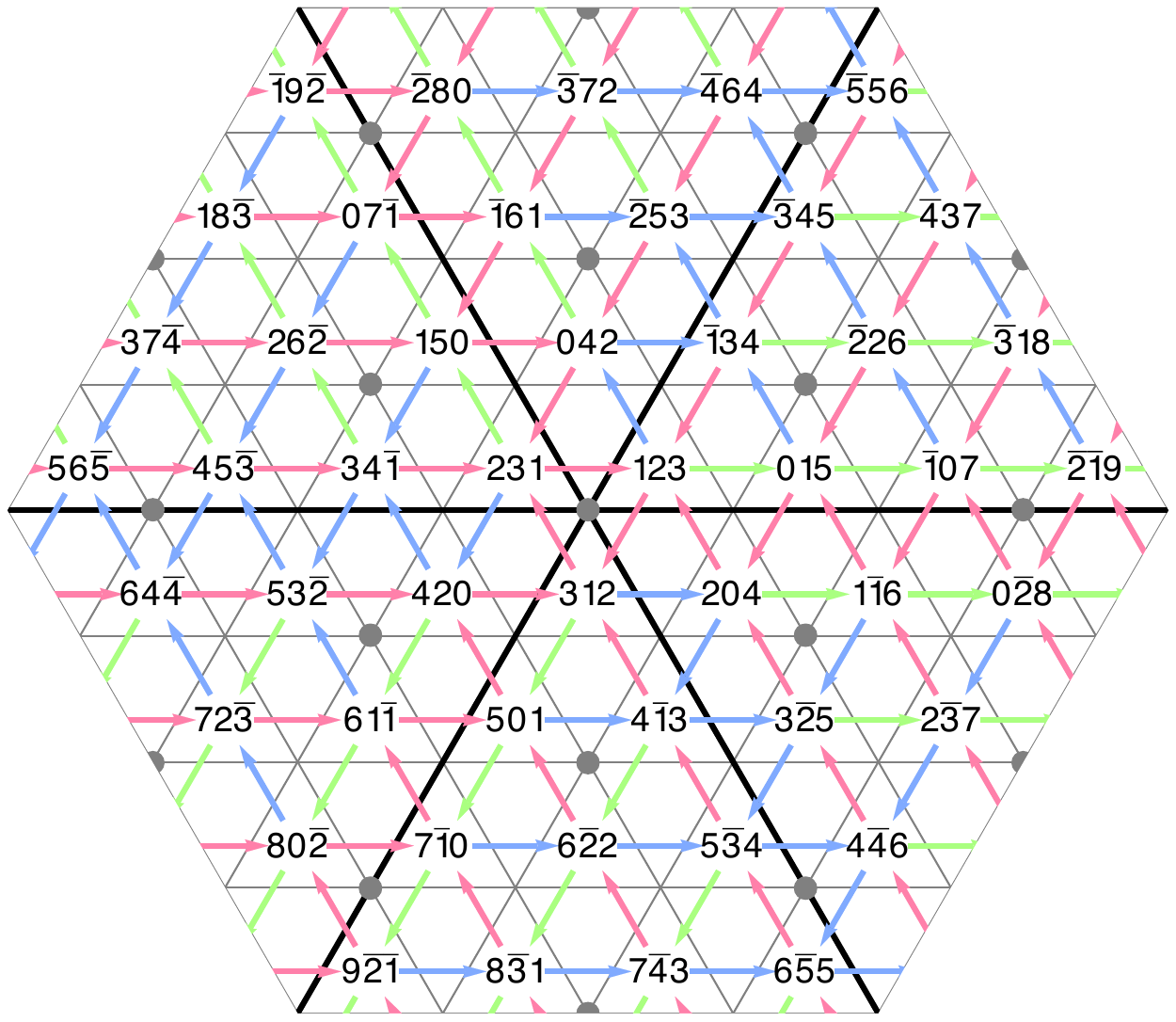}
\caption{An orbit of the action of the letters $\textcolor{red}{0}$, $\textcolor{blue}{1}$, and $\textcolor{green}{2}$ on $\RR^3_6 \setminus \mathcal{H}$. Acting by $\textcolor{red}{0}$ adds 3 to the smallest coordinate and subtracts $\mathbbm{1}_3$; acting by $\textcolor{green}{2}$ adds 3 to the largest coordinate and subtracts $\mathbbm{1}_3$; and acting by $\textcolor{blue}{1}$ adds 3 to the coordinate that is neither largest nor smallest and subtracts $\mathbbm{1}_3$.  For formatting, we have written $\overline{i}$ for $-i$ and suppressed commas and parentheses; thus, the string $\overline{2}53$ stands for the point with coordinates $(-2,5,3)$.}
\label{fig:parking_on_perm_classes}
\end{figure}

The action of a letter $i \in [m]$ on $V_t^m$ is the restriction to $V_t^m$ of a piecewise-linear function from $\RR^m_t \cong \RR^{m-1}$ to $V_t^m$ that sends alcoves to alcoves.
 By~\Cref{eq:action_letter}, the letter $i$ acts on $\xx \in \RR^m_t$ by the translation $\xx + m\mathbf{e}_i - \mathbbm{1}_m$, and the final resorting of the coordinates into increasing order may be interpreted geometrically by folding once along the simple hyperplane $\mathcal{H}^0_{i,i+1}$, and then folding again as needed along simple hyperplanes until all points lie in the cone $V_t^m$.
 
\begin{lemma}
\label{rem:action}
The action of a word $\w \in [m]^n$ on $V_t^m$ sends alcoves to alcoves and only decreases distances between points: $d(\mathbf{x},\mathbf{y})\geq d(\w(\mathbf{x}),\w(\mathbf{y}))$. In particular, $\w$ defines a continuous
map from $V_t^m$ to $V_t^m$.
\end{lemma}
\begin{proof}

  The lemma follows from the geometric description given above.  The action of $\w$ is the composition of a translation with a sequence of reflections; each of these operation sends alcoves to alcoves, so the same is true of the action of $\w$.

To see that the action of $\w$ reduces relative distances with respect to the metric $d$,
observe first that the initial 
  translation step does not change relative distances. The subsequent folding steps apply either the identity map or a reflection. Reflections and the identity map each individually preserve distances with respect to $d$, and so each 
  folding step can only reduce distances between points with respect to
  $d$, by the triangle inequality. 
  Since $\w$ does not increase distances with respect to $d$ (which is a constant multiple of the usual Euclidian distance), it is continuous.
%
  \end{proof}
  

\subsection{Affine Dimension}

  We say that a subset $X \subseteq V_t^m$ is \defn{of affine dimension $k$} if it is a convex set contained in an affine subspace of dimension $k$ and contains an  open ball in that affine subspace.  In particular, $X$ is of affine dimension $0$ if it consists of a single fixed point.   Note that affine dimension is not defined for an arbitrary subset of $V_t^m$; to say that a subset is of affine dimension $k$ is to make a strong statement about the kind of subset it is.  

For $\w \in [m]^n$, define \[\Fix(\w):=\big\{\xx \in V_t^m : \w(\xx)=\xx\big\}\] for the set of points fixed under $\w$.  Choose $\pp \in \Park_m^n$ and write $d:=\gcd(m,n)$.  We will prove~\Cref{thm:park_char} in~\Cref{sec:new_char} by showing that $\Fix(\pp)$ is of affine dimension $d-1$.  For now, we prove that $\Fix(\pp)$ is convex and is contained in an affine subspace of dimension $d-1$.  

\begin{lemma}
$\Fix(\pp)$ is convex.
\label{lem:convex}
\end{lemma}

\begin{proof}
Let $\pp \in \Park_m^n$ and suppose $\xx,\mathbf{y} \in \Fix(\pp)$. Since the application of $\pp$ only decreases distances, the fact that the strong triangle inequality implies that the line between them is fixed. 
\end{proof}

\begin{lemma}
Let $\gcd(m,n)=d$, $\pp \in \Park_m^n$, and suppose $\xx \in \Fix(\pp)$.
Then the multiset of coordinates $\{x_i\}_{i=0}^{m-1}$ can be partitioned into disjoint multisets, each of which is of size $m/d$ and of the form $\{a+kd+b_km\}_{k=0}^{m/d-1}$.
\label{lem:ad}
\end{lemma}

\begin{proof}
Up to the rebalancing by subtraction of a multiple of $\mathbbm{1}_m$, the action of each letter of $\pp$ increases one coordinate of $\xx$ by $m$, but the effect of the entire parking word is to send $x_i$ to $x_i + n$.  Since each individual entry changes by a multiple of $m$, it does not change modulo $m$. This means that the multiset of remainders of $x_i \mod m$ must be fixed under addition of $n$, so this multiset must also be fixed under addition of $\gcd(n,m)=d$.
\end{proof}

\begin{example}\label{ex:partitioning}
Fix $(m,n) = (6,9)$ with $d = \gcd(m,n) = 3$, and consider the $(m,n)$-parking word $\pp = 020101151$.  We verify that $\pp$ has a fixed point $\xx=(-3,1,2,4,6,11) \in V_{21}^6$ (up to rebalancing by subtraction of a multiple of $\mathbbm{1}_m$):
\begin{align*}
(\textcolor{red}{-3},\textcolor{blue}{1},2,\textcolor{blue}{4},\textcolor{red}{6},11) &\xmapsto{0} (\textcolor{blue}{1},2,\textcolor{red}{3},\textcolor{blue}{4},\textcolor{red}{6},11)  \xmapsto{2}(\textcolor{blue}{1},2,\textcolor{blue}{4},\textcolor{red}{6},\textcolor{red}{9},11) \xmapsto{0}(2,\textcolor{blue}{4},\textcolor{red}{6},\textcolor{blue}{7},\textcolor{red}{9},11) \xmapsto{1} \\ &\xmapsto{1} (2,\textcolor{red}{6},\textcolor{blue}{7},\textcolor{red}{9},\textcolor{blue}{10},11) \xmapsto{0} (\textcolor{red}{6},\textcolor{blue}{7},8,\textcolor{red}{9},\textcolor{blue}{10},11) \xmapsto{1} (\textcolor{red}{6},8,\textcolor{red}{9},\textcolor{blue}{10},11,\textcolor{blue}{13}) \xmapsto{1} \\ &\xmapsto{1} (\textcolor{red}{6},\textcolor{red}{9},\textcolor{blue}{10},11,\textcolor{blue}{13},14)  \xmapsto{5} (\textcolor{red}{6},\textcolor{red}{9},\textcolor{blue}{10},11,\textcolor{blue}{13},20) \xmapsto{1} (\textcolor{red}{6},\textcolor{blue}{10},11,\textcolor{blue}{13},\textcolor{red}{15},20).\end{align*}
Then $\xx$ is a fixed point of $\pp$, since rebalancing gives \[\pp(\xx)=\pp(-3,1,2,4,6,11) = (6,10,11,13,15,20)-9\cdot \mathbbm{1} = (-3,1,2,4,6,11) = \xx.\] 
The partition guaranteed by~\Cref{lem:ad} is $\textcolor{red}{\{-3,6\}},\textcolor{blue}{\{1,4\}},\{2,11\}$.

\end{example}

\begin{lemma}
\label{lem:contained}
Let $\gcd(m,n)=d\geq 1$ and $\pp\in \Park_m^n$.  Then $\Fix(\pp)$ is contained in an affine subspace of dimension $d-1$.
\end{lemma}

\begin{proof}
Let $\xx\in\Fix(\pp)$, and let $\yy\in\Fix(\pp)$ be another fixed point in a small ball around $\xx$.  By \Cref{lem:ad}, the coordinates of $\yy$ can be partitioned into $d$ disjoint multisets of size $m/d$, each of which consists of a set of residues mod $m$ which are fixed under addition of $d$.  Because $\yy$ is close to $\xx$, the partition we have found for the coordinates of $\yy$ also works for the coordinates of $\xx$. For each of the parts in the partition, there is therefore some offset such that adding this offset to the coordinates of $\xx$ in that part, yields the coordinates of $\yy$ in that part.  These offsets must add up to zero, since the sum of the entries of $\xx$ and $\yy$ are assumed to be the same. Therefore, $\yy$ lies in an affine subspace of dimension $d-1$ which also contains $\xx$. 

In principle, if we chose a different $\yy'\in\Fix(\pp)$ near $\xx$, we could obtain a different affine subspace (corresponding to a different way of partitioning the coordinates of $\xx$).  However, convexity would then imply that the line between $\yy$ and $\yy'$ is also in $\Fix(\pp)$, and this includes points which are not on any affine subspace of the above form, which is impossible.  Thus all the points in $\Fix(\pp)$ near $\xx$ lie in a single affine subspace.    By convexity, any point in $\Fix(\pp)$ lies in the same affine subspace.  
\end{proof}

\section{A New Characterization of Parking Words}
\label{sec:new_char}
In this section we prove~\Cref{thm:park_char}, distinguishing parking words in the set of all words in $[m]^n$ using the action of a word on $V_t^m$.  

\begin{definition}\label{def:touch} For $\pp \in \Park_m^n$ and $1 \leq i < m$, define $i$ to be a \defn{touch point} of $\pp$ if \[\Big|\big\{j : \pp_j<i\big\}\Big|=i\frac{n}{m}.\]\end{definition}

Note that we do not count $0$ or $m$ as touch points, so that when $n$ and $m$ are coprime, no parking word has a touch point.

We break the proof of~\Cref{thm:park_char} into five parts.  Let $\pp \in \Park_m^n$ and $d:=\gcd(m,n)$:
\begin{itemize}
    \item \Cref{lem:rp_fixed}: if $\pp$ has no touch points, then it has a fixed point.
    \item \Cref{lem:rp_unique}: if $d=1$, then $\pp$ has a unique fixed point.
    \item \Cref{lem:nrp_fixed0}: if $\pp$ has no touch points, then its fixed point space is bounded of affine dimension $d-1$.

    \item \Cref{lem:nrp_fixed}: if $\pp$ has a touch point, then it has infinitely many fixed points, on which $\mathcal N$ is arbitrarily large.
    \item \Cref{lem:no_fixed}: if $\w \in [m]^n \setminus \Park_m^n$, then it has no fixed points.
\end{itemize}

\subsection{Parking words without touch points}

We begin by recalling the Brouwer fixed point theorem.

\begin{theorem}[{Brouwer~\cite[Theorem 1.2]{todd1976fixedpoint}}]
Any continuous function from a closed topological ball to itself has a fixed point.
\end{theorem}

\begin{lemma}\label{lem:rp_fixed}
Let $\pp \in \Park_m^n$ have no touch point.  Then $\pp$ has a fixed point in $V_t^m$.
\end{lemma}

\begin{proof}  We argue for $t=0$, since the the statement for general $t$ follows by rebalancing.
We want to show that $\mathcal N(\pp(\xx))<\mathcal N(\xx)$ for $\mathcal N(\xx)>N$, provided $N$ is sufficiently large (that is, parking contracts).  By~\Cref{rem:action}, the Brouwer fixed point theorem can then be invoked on the $(m-1)$-ball \[\left\{\xx \in V^m_0 : \mathcal N(\xx)\leq N\right\}\] to guarantee a fixed point.

\medskip

We first consider the case that the $x_i$ are sufficiently separated that each ``resort'' step does nothing---that is, the actions of $\pp$ on $\xx$ as an element of $\RR^m_0 \setminus \mathcal{H}$ and as an element of $V_0^m$ coincide.  For $0\leq i <m$, let \[q_{i} = \Big| \big\{ j : \pp_j = i\big\} \Big|\] be the number of occurrences of $i$ in the $(m,n)$-parking word $\pp$.  Now, for $\xx \in V_0^m$,

\begin{multline} \mathcal N(\pp(\xx))-\mathcal N(\xx)\\
=\sum_{0 \leq i<j< m} \left[ (x_i+mq_i-x_j-mq_j)^2 - (x_i-x_j)^2 \right] \\
=\sum_{0 \leq i<j<m}m^2(q_i-q_j)^2 + \sum_{i=0}^{m-1} 2m^2 q_i x_i -2\left(\sum_{i=0}^{m-1} x_i\right)\left( \sum_{i=0}^{m-1} q_i\right).\label{eq:pdiff}\end{multline}

Of the three terms on the right-hand side of~\Cref{eq:pdiff}, the first
sum depends only on $\pp$.  The third vanishes because we have assumed that $\sum_i x_i =0$.  We want to show that the second sum on the right-hand side \[\sum_{i=0}^{m-1} 2m^2 q_i x_i\] is sufficiently negative to
dominate the first, provided $\mathcal N(\xx)$ is big enough.  We will begin by establishing that the second sum is negative, by showing that we can add a sequence of positive numbers to it
to make it zero.

More precisely, we will do the following. For $0\leq i \leq m-1$, initialize
the variable $y_i=x_i$. We will now carry out a process where we gradually change the value $\yy$, so that, at each step $\sum_i q_iy_i$ increases, and so that, at the end, $y_i=0$ for all $i$. Since
$\sum_i q_i \cdot 0=0$, this will show that the initial value, $\sum_i q_ix_i$, was
negative. 

Suppose that $y_0=y_1=\dots=y_{b-1}$ (with $b$ maximal) and
$y_{m-1}=y_{m-2}=\dots=y_{m-c}$ (with $c$ maximal).  (At the beginning of the process, when $y_i=x_i$ for all $i$, $b$ and $c$ will both be 1, but as we continue, this will change.) If we increase each of the $b$
minimal coordinates of $\yy$ by $c\alpha$ and decrease each of the $c$ maximal coordinates of $\yy$ by $b\alpha$, we
have not changed the average value of $\yy$.  On the other hand, the value of $\sum_i q_i y_i$
changes by

\begin{equation}\label{eq:pdiff2}\sum_{i=0}^{b-1} \alpha c q_i -  \sum_{i=m-c}^{m-1} \alpha b q_i.\end{equation}

By~\Cref{eq:parking_def2}---because $\pp$ is an $(m,n)$-parking word---the first term of~\Cref{eq:pdiff2} is greater than $\alpha cb\frac{n}{m}$, while the second term is less
than $\alpha cb\frac{n}{m}$.  Thus, changing the values of $\yy$ in this way
increases the sum $\sum_i q_i y_i$.

Choose $\alpha$ maximal so that,
in increasing the $b$ minimal coordinates and decreasing the $c$ maximal coordinates, none of the values changed pass any other values of $\yy$.  We call this a \defn{step}.  Since at least one of $b$ or $c$ increases, after a finite number of steps, we will have all entries of $\yy$ at zero and the value of the sum
$\sum_i q_i y_i$ will also be zero.  But since we increased the value of the sum at each step, its
initial value was negative.

In fact, we can bound the value of $\sum_i q_i x_i$ away from 0 by approximating the change of $\sum_i q_iy_i$ across all steps.  For $b,c<m$, we have \begin{align*}\left(\sum_{i=0}^{i=b-1} q_i\right) - \frac{bn}{m} &\geq \frac{1}{m} \hspace{2em} \text{ and} \\ \frac{cn}{m} - \left(\sum_{i=m-c}^{i=m-1}q_i\right) &\geq \frac{1}{m},\end{align*} since both left-hand sides are strictly positive (because of our assumption that $\pp$ has no touch points) and can be expressed as a rational number with denominator $m$.  Therefore,
\begin{equation}\sum_{i=0}^{b-1} \alpha c q_i -  \sum_{i=m-c}^{m-1} \alpha b q_i \geq \frac{\alpha c}{m} + \frac{\alpha b}{m}.\label{eq:approx}\end{equation}
To approximate $\sum_i q_i y_i$, we bound the two terms on the right-hand side of~\Cref{eq:approx} over the entire process which moves all the
$y_i$ to zero.

Since $\alpha c$ is the amount that each of the minimal $y_i$'s were moved during each step, the sum of $\alpha c/m$ over all steps is
$1/m$ times the total amount the minimal coordinates are increased over the whole process.  But this begins at $x_0$ and terminates at $0$, so the total amount they change by is $-x_0$ and the sum of the first term on the right-hand side of~\Cref{eq:approx} over the whole process is $-x_0/m$.  Similarly, the sum of the
second term on the right-hand side of~\Cref{eq:approx} over the whole process is $x_{m-1}/m$.  We obtain the bound
\[\sum_{i=0}^{m-1} 2m^2 q_i x_i \leq 2m \left(x_0-x_{m-1}\right),\]
 which we can make as negative as we like by
requiring $\mathcal N(\xx)$ to be sufficiently large.

\medskip

We now consider the case that the resorting is not necessarily trivial---that is, the action on $\xx$ as an element of $\RR^m_0 \setminus \mathcal{H}$ and as an element of $V_0^m$ do \emph{not} necessarily coincide.  Fix a sorted tuple $\xx$ and distinguish this tuple as living in $\RR^m_0 \setminus \mathcal{H}$ or $V_0^m$ by writing $\xx_U \in \RR^m_0 \setminus \mathcal{H}$ and $\xx_V \in V_0^m$.  Let \[\xx_U^{(j)} := \pp_0\pp_1\cdots\pp_{j-1}(\xx_U) \text{ and } \xx_V^{(j)} :=\pp_0\pp_1\cdots\pp_{j-1}(\xx_V).\]  We note that after applying a single letter $i$ to $\xx_U$ and $\xx_V$, the difference between any coordinate of $\xx_U$ and the same coordinate of $\xx_V$ is less than $m$.  By induction,
corresponding coordinates of $\xx_U^{(j)}$ and $\xx_V^{(j)}$ differ by at most $mj$.

On the other hand, for any tuple $\yy$, applying a single letter $i$ to $\yy_U$ or $\yy_V$, we compute
\begin{align*}\mathcal N(i(\yy_U))-\mathcal N(\yy_U) = \mathcal N(i(\yy_V))-\mathcal N(\yy_V) &= \sum_{\substack{0\leq j < m\\ j\ne i}} \Big[ (y_i+m - y_j)^2 - (y_i-y_j)^2\Big] \\
&= (m-1)m^2 +2m\sum_{\substack{0\leq j < m\\ j\ne i}} (y_i-y_j) \\ &= (m-1)m^2+2m^2y_i.\label{eq:diffdiff}\end{align*}

By telescoping, we can now bound the difference $\mathcal N(\pp(\xx_U))-\mathcal N(\pp(\xx_V))$:

\begin{align*}
\mathcal N(\pp(\xx_U))-\mathcal N(\pp(\xx_V)) &= \sum_{j=0}^{n-1} \left(\mathcal N(\xx_U^{(j+1)})-\mathcal N(\xx_U^{(j)})\right)-\left(\mathcal N(\xx_V^{(j+1)})-\mathcal N(\xx_V^{(j)})\right) \\ &\leq \sum_{j=0}^{n-1} 2m^3 j = n(n-1)m^3,
\end{align*}
using our analysis that corresponding coordinates in $\xx_U^{(j)}$ and $\xx_V^{(j)}$ differ by at most $mj$.
This quantity is still a constant in the fixed parameters $n$ and $m$, so we
can overcome it by requiring that $\mathcal N(\xx)$ be sufficiently large.

\medskip

We conclude that the second term of the right-hand side of~\Cref{eq:pdiff} dominates the first if $\mathcal N(\xx)>N$ for $N$ sufficiently large, so that $\mathcal N(\pp(\xx))<\mathcal N(\xx)$ for $\mathcal N(\xx)>N$.
\end{proof}

In the case $\gcd(m,n)=1$, \Cref{lem:contained}, together with our previous results, suffices to show that the set of fixed points is of affine dimension 0 (i.e., consists of a single point).

\begin{corollary}
Let  $\gcd(m,n)=1$ and $\pp \in \Park_m^n$.  Then $\Fix(\pp)$ is of affine dimension 0.  In particular, $|\Fix(\pp)|=1$.
\label{lem:rp_unique}
\end{corollary}

We now show that $\Fix(\pp)$ is of affine dimension $d-1$ for 
$d=\gcd(m,n)$ in the case that $\pp$ has no touch points.

\begin{lemma}
Let  $\gcd(m,n)=d\geq 1$ and $\pp \in \Park_m^n$ with no touch points.  Then $\Fix(\pp)$ is bounded of affine dimension $d-1$.
\label{lem:nrp_fixed0}
\end{lemma}

\begin{proof}
$\Fix(\pp)$ is bounded, since we showed in~\Cref{lem:rp_fixed} that $\mathcal N(\pp(\xx))<\mathcal N(\xx)$ for $\mathcal N(x)>N$ for some large $N$.

Let $F$ be a face of the affine arrangement $\mathcal{H}$ such that for any face $G$ having $F$ as a face, we have $G\cap\Fix(\pp)=F\cap\Fix(\pp)$.  Suppose, seeking a contradiction, that some $F$ of codimension $c\geq 1$ exists.  Consider the action of $\pp$ on a small sphere $S$ around a point $x$ of $\Fix(\pp)$ in the plane normal to $F$.  Since the sphere is not fixed by $\pp$, the action of $\pp$ on it is by some non-trivial foldings.  The image therefore misses some open ball $B$ in the sphere.  Restricting, $\pp$ now defines a map from $S\setminus B$ to
$S\setminus B$, and by Brouwer's fixed point theorem, it has a fixed point.  This contradicts our assumption on $F$.  Thus there must be a fixed point $\xx$ not lying on any hyperplane.

\Cref{lem:ad} divides the set of all coordinates of $\xx$ into $d$ subsets of size $m/d$, where the elements of each set are congruent modulo $d$.  Since $\xx$ lies on no hyperplane, no coordinate value modulo $m$  is repeated, so it is unambiguous how to apportion the coordinates into these sets.

Now consider the action of $\pp$, omitting rebalancing.  Each entry in the multiset of coordinates is changed by a multiple of $d$.  Thus the entries in each of the $d$ subsets are permuted among themselves by the action of $\pp$.  We may translate each family with respect to the others by some small amount without changing the relative order of the coordinates, so all such points are still fixed.  This gives us an
open ball around $\xx$ in the $(d-1)$-dimensional affine subspace constructed in \Cref{lem:contained} consisting entirely of fixed points.  $\Fix(\pp)$ is therefore of affine dimension $d-1$.
\end{proof}

\begin{example}
As in~\Cref{ex:partitioning}, fix $(m,n) = (6,9)$ with $d = \gcd(m,n) = 3$ and the $(m,n)$-parking word $\pp = 020101151$.  Note that $\pp$ has no touch points, and recall that $\pp$ has a fixed point $\xx=(-3,1,2,4,6,11) \in V_{21}^6$. Modulo $d$, this fixed point is of the form $(0,1,2,1,0,2)$.  Let \begin{align*}v_0 &= (-3,0,3,3,6,12),\\
v_1 &= (-2,1,1,4,7,10), \text{ and} \\
v_2 &= (-4,2,2,5,5,11).\end{align*}  Then one can check that $\mathrm{Fix}(\pp) \supseteq \mathrm{conv}(v_0,v_1,v_2).$
\end{example}

\subsection{Parking words with touch points}
When the parking word has a touch point, we now use~\Cref{lem:rp_fixed} to also produce infinitely many fixed points. The value of $\mathcal N$ on these fixed points may now be arbitrarily large.

\begin{lemma}
The action of $\w \in \Park_m^n$ on $V_t^m$ has infinitely many fixed points when $\w$ has at least one touch point.  The set $\Fix(\w)$ has affine dimension $d-1$, and contains fixed points on which $\mathcal N$ is arbitrarily large.
\label{lem:nrp_fixed}
\end{lemma}

\begin{proof}  As in~\Cref{lem:rp_fixed}, it suffices to argue for $V_t^m$ for $t=0$.  Suppose that $\gcd(m,n)=d\neq 1$ and that $\w$ has $k\geq 1$ touch points.  We will break $\pp$ into a number of smaller parking words based on its touch points, find the unique fixed points for each of those parking words, and then reassemble them in uncountably many ways to find fixed points for $\pp$.  To this end, list the $k$ touch points of $\pp$ as $m_1,\dots,m_k$ with \[m_0=0<m_1<m_2<\cdots<m_{k}<m=m_{k+1}.\]  For $0\leq j \leq k$, let $\pp^{(j)}$ be the (not-necessarily consecutive) subword of $\pp$ containing all letters $p$ of $\pp$ such that $m_{j}\leq p < m_{j+1}$.  Let $n_j$ be the length of $\pp^{(j)}$---necessarily a multiple of $n/d$---and note that $\pp$ is a shuffle of $\pp^{(0)},\pp^{(1)},\ldots,\pp^{(k-1)}$.

To define smaller parking words, we shift the individual letters of $\pp^{(j)}$ by the previous touch point to produce the $(m_j,n_j)$-parking word $\qq^{(j)}:= \pp^{(j)}-m_{j}$.

We can now use~\Cref{lem:rp_fixed} and the previous case to find $\xx^{(j)} \in V_0^{m_j}$ that are fixed points for the $\qq^{(j)}$.    In preparation to reassemble these individual fixed points $\xx^{(j)}$ into a fixed point for $\pp$, we scale them to define \[\xx^{(j)}_N := \frac{n}{n_j}\xx^{(j)}+N_j\] for some $N_j \in \RR$.  Finally, define $\xx_N \in V_0^m$ by the concatenation: \[\xx_N := \left(\xx^{(0)}_N,\xx^{(1)}_N,\ldots,\xx^{(k)}_N\right),\] and then rebalancing so that the sum is 0.

We now check that $\xx_N$ is really a fixed point of $\pp$, as long as the $N_j$ give sufficient space between the $\xx_N^{(j)}$.  Since $\pp$ is a shuffle of the $\pp^{(j)}$, as long as the individual coordinates of $\xx_N$ do not overlap during the application of the letters of $\pp$ (for example, we may take $N_j>mn+N_{j-1}$), we may discuss the action of $\pp$ on each component  $\xx_N^{(j)}$ separately.  On $\xx_N^{(j)}$, then, only the subword $\pp^{(j)}$ of $\pp$ will act; the only difference from its usual action on $\xx^{(j)}$ is that (as a piece of the larger parking word $\pp$) it adds $m$ rather than $m_j$---but we have compensated for this by the scaling factor $\frac{n}{n_j}$.
\end{proof}

\begin{example}
We illustrate the proof of~\Cref{lem:nrp_fixed}.  Let $(m,n)=(9,12)$ so that $d=3$, and let $\pp = 531030678631$.  Then there are 2 touch points of $\pp$: $m_1=3$ and $m_2=6$, so that \[\pp^{(0)} = 1001, \pp^{(1)}=5333, \text{ and } \pp^{(2)} = 6786\] and \[\qq^{(0)} = 1001, \qq^{(1)}=2000, \text{ and } \qq^{(2)} = 0120.\]
Fixed points for $\qq^{(j)}$ are
\[\xx^{(0)} = (-2,0,2),\, \xx^{(1)} = (-1,0,1), \text{ and } \xx^{(2)} = (-2,-1,3),\]
so that
\[\xx^{(0)}_N = (-6,0,6),\, \xx^{(1)}_N = (-3,0,3)+N_2, \text{ and } \xx^{(2)}_N = (-6,-3,9)+N_3,\]
and before rebalancing
\[\xx_N = \left(-6,0,6,-3+N_2,N_2,3+N_2,-6+N_3,-3+N_3,9+N_3\right).\]

When $N_2>21$ and $N_3>21+N_2$, we see that portion of $\pp$ corresponding to $\pp^{(j)}$ acts only on the $3j-2$, $3j-1$, and $3j$th coordinates of $\xx$.
\end{example}

\subsection{Non-parking words}

\begin{lemma}
For any $\xx \in V_t^m$ and any $\w \in [m]^n \setminus \Park_m^n$, $\lim_{i\to \infty} (\w^i(\xx))_{m-1} = \infty$.  In particular, $\w$ has no fixed points.
\label{lem:no_fixed}
\end{lemma}

\begin{proof}
It suffices to argue for $V_t^m$ for $t=0$.  We show that repeated application of $\w$ sends the last coordinate of any point to infinity.  Suppose that $\w \in [m]^n \setminus \Park_m^n$ is not a parking function because it has too many numbers that are at least $k$, and choose $k$ maximal.
Let \[\xx= \left( x_0,x_1,\ldots,x_{m-1}\right) \in V_0^m\] be a vector with sum $0$.  We claim that the result of applying $\w$ to $\xx$ has the effect of increasing the difference between the average value of $x_k,\dots,x_{m-1}$ and the average value of $x_0,\dots,x_{k-1}$ by a fixed quantity.  Thus, after enough applications of $\w$, the value of $x_{m-1}$ will be arbitrarily large.

In the course of applying $\w$ to $\xx$ there are two ways that the difference between the
average value of $x_0,\dots,x_{k-1}$ and the average value of $x_{k},\dots,x_{m-1}$ changes.  One is as a result of adding $m$ to an entry corresponding to an element of $\w$.  By the assumption on $\w$, these steps have the property that, on average, a more than proportionate number of these steps are applied to the entries $x_k,\dots,x_{m-1}$, which therefore increases the difference between the average values by a fixed positive amount.  The other way that the difference between the averages increases is in the resort step.  If an element in $x_1,\dots,x_{k-1}$ is increased far enough that it moves into the top $m-k$ elements, then it is resorted into one of these positions.  Whenever this happens, this also increases the difference between the average values.
\end{proof}

\subsection{Summary}
We obtain~\Cref{thm:park_char} as a corollary of~\Cref{lem:rp_fixed,lem:nrp_fixed0,lem:nrp_fixed,lem:no_fixed} and ~\Cref{lem:rp_unique}.  Examples for $m=3$ are given in~\Cref{fig:examplessection3}.

 The remainder of this paper is devoted to explaining the coprime case in more detail, explicitly identifying the isolated fixed points of parking words as the centers of alcoves of dominant affine permutations whose inverses lie in the Sommers region.  It would be desirable to explicitly identify the regions of fixed points in the non-relatively prime case.  We note that in the special $(m,mk)$ case when the fixed regions are full dimensional, Gorsky, Mazin, and Vazirani have recently identified the set of fixed regions of an $(m,mk)$-parking word with the dominant regions in the $k$-Shi arrangement of $\widetilde{\mathfrak{S}}_{m}$~\cite[Section 3.4]{gorsky2017rational} (compare with~\Cref{fig:examplessection3}).

\begin{figure}[htb]
\centering
\begin{tikzpicture}[scale=1.6]
\draw[thick] (0,0) -- (1.5,2.6);
\draw[thick] (0,0) -- (-1.5,2.6);
\draw[thick] (-.5,.866) -- (.5,.866);
\draw[thick] (-.5,.866) -- (.5,2.6);
\draw[thick] (.5,.866) -- (-.5,2.6);
\node(a) at (0,0.577) {$123$};
\node(b) at (0,1.155) {$024$};
\node(e) at (0,2.3094) {$\overline{2}26$};
\node(d) at (.5,1.443) {$015$};
\node(c) at (-.5,1.443) {$\overline{1}34$};
\node[draw,gray](a3) at (-3,-.866) {$\begin{array}{ccc} 000 &010 &100 \\ 110 &200 &210\end{array}$};
\node[draw,gray](a4) at (3,-.866) {$\begin{array}{cccc} 0000 &0010 &0100 &0110 \\ 0200& 0210& 1000& 1010 \\ 1100 &2000& 2010& 2100\end{array}$};
\node[draw,gray](b3) at (-3,.2887) {$\begin{array}{ccc} 001 &020 &101 \\ 110 &200 &210\end{array}$};
\node[draw,gray](b4) at (3,.2887) {$\begin{array}{ccc} 0001 &0020 &0101 \\ 1001 &1020 &1200\end{array}$};
\node[draw,gray](c3) at (-3,1.44) {$\begin{array}{ccc}011 &021 &101\\ 110 &201 &210\end{array}$};
\node[draw,gray](d3) at (-3,2.59) {$\begin{array}{ccc}002 &020 &102\\ 120 &200 &210\end{array}$};
\node[draw,gray](c4) at (3,2.59) {$\begin{array}{cccc}0011 &0021 &0201 &2001\end{array}$};
\node[draw,gray](d4) at (3,1.44) {$\begin{array}{cccc}0002 &0102 &0120 &1002\end{array}$};
\node[draw,gray](e3) at (-3,3.753) {$\begin{array}{ccc}012 &021 &102 \\120 &201 &210\end{array}$};
\node[draw,gray](e4) at (3,3.753) {$\begin{array}{c}0012\end{array}$};
\draw[gray,->] (a3) -- (a);
\draw[gray,->] (b3) -- (b);
\draw[gray,->] (a4) -- (a);
\draw[gray,->] (b4) -- (b);
\draw[gray,->] (c3) -- (c);
\draw[gray,->] (c4) -- (c);
\draw[gray,->] (d3) -- (d);
\draw[gray,->] (d4) -- (d);
\draw[gray,->] (e3) -- (e);
\draw[gray,->] (e4) -- (e);
\end{tikzpicture}
\caption{The dominant part of the $\widetilde{\mathfrak{S}}_3$ Shi arrangement.  Each region is labeled by a coordinate corresponding to the one-line notation of the affine permutation whose alcove is lowest in the region (see~\Cref{sec:affine_symmetric_and_regions} for more details).  The gray words on the left are the $(3,3)$-parking words that fix every point of the (closed) region to which they point; the gray words on the right are the $(4,3)$-parking words that fix precisely the coordinate to which they point.}
\label{fig:examplessection3}
\end{figure}

\section{Parking Filters}
\label{sec:parking_functions}

\emph{For the rest of the paper, we fix $m$ and $n$ relatively prime.}  In this section, we define the combinatorial objects---generally thought of as Dyck paths and labeled Dyck paths---that we will use to compute the zeta map defined in~\Cref{sec:combinatorial_zeta}.  These objects are all well-known; our main contribution is the simplicity of our definition of the zeta map on parking functions in~\Cref{def:zetamap}, and its relation with affine permutations in~\Cref{sec:affine_symmetric_and_regions}.

\subsection{Filters}\label{sec:filters}
Fix $m$ and $n$ relatively prime, and label the point $(i,j) \in \ZZ \times \ZZ$ by its \defn{level} \[\ell(i,j)=(i,j) \cdot (m,n) = im+jn.\]  If we draw the levels of points in the plane, rows correspond to residue classes modulo $m$, while columns correspond to residue classes modulo $n$.  Any fixed row and column intersect in a unique point, and the Chinese remainder theorem ensures that the levels are distinct modulo $mn$ in any contiguous $n \times m$ rectangle.  A portion of the levels of $\ZZ \times \ZZ$ for $(m,n)=(3,4)$ and $(3,5)$ is illustrated in~\Cref{fig:filters}.

\begin{definition}
An \defn{$(m,n)$-filter} $\idl$ is a subset of $\ZZ \times \ZZ$ with $\min_{(i,j) \in \idl} \ell(i,j) > -\infty$, such that whenever the point $(i,j)$ is in $\idl$, then 
the following points are also in $\idl$:
\begin{itemize}
\item $(i+m,j)$ and $(i,j+n)$, as well as
\item all $(i',j')$ for which $\ell(i',j')=\ell(i,j)$.
\end{itemize}
A \defn{corner} of $\idl$ is a point $(i,j)$ such that neither $(i-m,j)$ nor $(i,j-n)$ are in $\idl$.  We write $\FF_m^n$ for the set of all $(m,n)$-filters.
\label{def:filters}
\end{definition}

Interchanging the copies of $\ZZ$ in $\ZZ \times \ZZ$ gives a bijection between the set of $(m,n)$-filters and the set of $(n,m)$-filters; we call this the \defn{$(m{\leftrightarrow}n)$-bijection}.  An $(m,n)$-filter $\idl$ is specified in three natural ways:
\begin{itemize}
    \item $\ell(\idl) := \{ \ell(i,j) : (i,j) \in \idl\}$, the set of all its levels,
    \item $m(\idl) := \mathrm{sort}([ \min_{(i,j) \in \idl} \ell(i,j) : j \in \mathbb{Z} ])$, i.e., the sorted list formed by taking, for each row, the minimal level of a point in that row which is also in $\idl$, or
    \item $n(\idl) := \mathrm{sort}([ \min_{(i,j) \in \idl} \ell(i,j) : i \in \mathbb{Z} ])$, i.e., the sorted list formed by taking, for each column, the minimal level of a point in that column which is also in $\idl$.
\end{itemize}
Note that $m(\idl)$ consists of $m$ integers, one from each congruence class mod $m$, while $n(\idl)$ consists of $n$ integers, one from each congruence class mod $n$.   An example of~\Cref{def:filters} is given in~\Cref{fig:filters}. 

\begin{figure}[htb]
\includegraphics{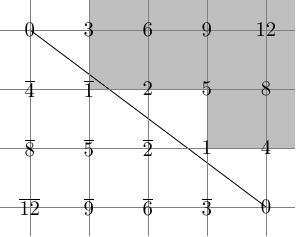} \hfill \includegraphics{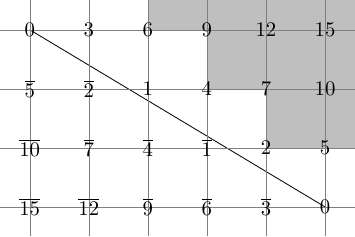}
\caption{Ignoring the shading, a portion of the levels of $ \ZZ \times \ZZ$ for $(m,n)=(3,4)$ and $(3,5)$  (the picture is extended to the rest of the plane by periodicity).  The solid gray line through the points of level $0$ separates the positive and negative levels.  On the left, the shading specifies $\idl \in \FF_3^4$ as those lattice points contained in the shaded region; similarly, the shading on the right specifies $\idl' \in \FF_3^5$.  One checks that $m(\idl)=[-1,1,3]$, $n(\idl)=[-1,1,2,4]$, $m(\idl')=[2,4,6]$, and $n(\idl')=[2,4,5,6,8]$.}
\label{fig:filters}
\end{figure}

\begin{definition}\label{def:equiv}
We say that $\idl,\idl'\in \FF_m^n$ are \defn{equivalent} if $\idl = \idl'+(x,y)$ for some $(x,y) \in \ZZ \times \ZZ,$ and write $\EF_m^n$ for the set of equivalence classes of $\FF_m^n$.
\end{definition}

\begin{definition}
Define a directed graph $\mathfrak{F}_m^n$ on $\EF_m^n$ with a directed edge from $\widetilde{\idl} \in \EF_m^n$ to $\widetilde{\idl}' \in \EF_m^n$ iff there is some $\idl' \in \widetilde{\idl}'$ and some $\idl \in \widetilde{\idl}$ such that $\ell(\idl')$ can be obtained from $\ell(\idl)$ by removing a single level from $\idl$ (pictorially, this means that $\idl'$ can obtained from $\idl$ by removing a corner).  We write $\widetilde{\bal}_m^n$ for the equivalence class containing the $(m,n)$-filters generated by a single level.
\label{def:directed_graph}
\end{definition}

The $(m{\leftrightarrow}n)$-bijection gives an isomorphism between $\mathfrak{F}_m^n$ and $\mathfrak{F}_n^m$.  The graphs $\mathfrak{F}_3^4$ and $\mathfrak{F}_3^5$ are illustrated in~\Cref{fig:balanced,fig:balanced2}.

\begin{figure}[htb]
\centering
\scalebox{0.7}{
\begin{tikzpicture}[scale=3]
\node (A) [fill=white!20,rounded corners=2pt,inner sep=1pt] at (0,0) {\scalebox{0.7}{\includegraphics{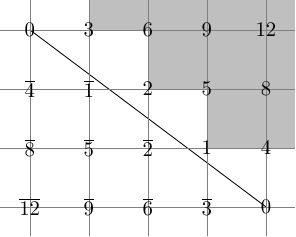}}};
\node (D) [fill=white!20,rounded corners=2pt,inner sep=1pt] at (1.5,2.5) {\scalebox{0.7}{\includegraphics{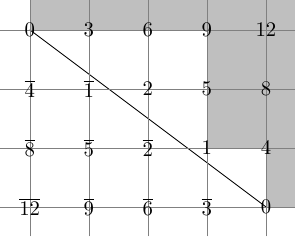}}};
\node (E) [fill=white!20,rounded corners=2pt,inner sep=1pt] at (0,3.5) {\scalebox{0.7}{\includegraphics{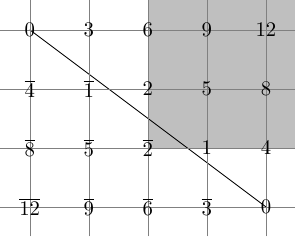}}};
\node (B) [fill=white!20,rounded corners=2pt,inner sep=1pt] at (-1.5,2.5) {\scalebox{0.7}{\includegraphics{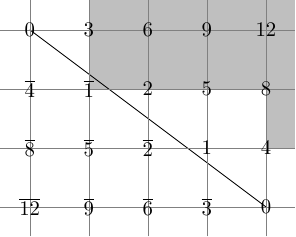}}};
\node (C) [fill=white!20,rounded corners=2pt,inner sep=1pt] at (0,1.5) {\scalebox{0.7}{\includegraphics{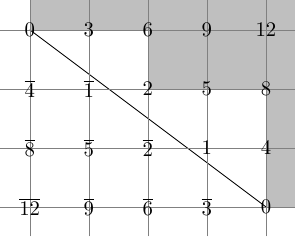}}};
\draw[->] (B) to [bend right]  node[midway, left] {$-1$} (A);
\draw[->] (C) to [bend right] node[midway, left] {$0$} (A);
\draw[->] (A) to [bend right] node[midway, right] {$2$} (C);
\draw[->] (B) to node[midway, above] {$4$} (E);
\draw[->] (D) -- (B) node[midway, above] {$1$};
\draw[->] (C) -- (B) node[midway, below] {$2$};
\draw[->] (D) -- (C) node[midway, below] {$0$};
\draw[->] (A) to [bend right] node[midway, right] {$3$} (D);
\draw[->] (E) -- (D) node[midway, above] {$-2$};
\draw [->] (A) edge[loop below]node{$1$} (A);
\end{tikzpicture}
}
\caption{The directed graph $\mathfrak{F}_3^4 \cong \mathfrak{F}_4^3$, with equivalence classes represented by the balanced $(3,4)$-filters of \protect\Cref{sec:balanced}.  The edge labels record the level of minimal element removed.  Compare with \protect\Cref{fig:perms_balanced}.}
\label{fig:balanced}
\end{figure}

\begin{figure}[htb]
\centering
\scalebox{0.7}{
\begin{tikzpicture}[scale=3.5]
\node (A) [fill=white!20,rounded corners=2pt,inner sep=1pt] at (0,0) {\scalebox{0.7}{\includegraphics{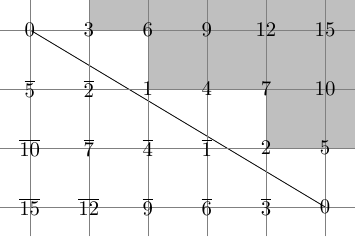}}};
\node (D) [fill=white!20,rounded corners=2pt,inner sep=1pt] at (1.5,2.5) {\scalebox{0.7}{\includegraphics{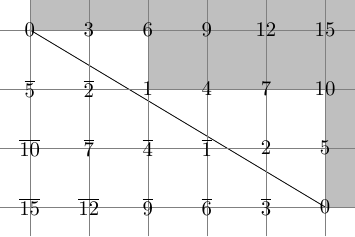}}};
\node (D2) [fill=white!20,rounded corners=2pt,inner sep=1pt] at (1.5,4) {\scalebox{0.7}{\includegraphics{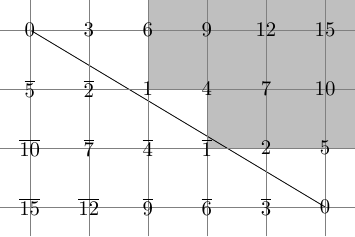}}};
\node (E) [fill=white!20,rounded corners=2pt,inner sep=1pt] at (0,5) {\scalebox{0.7}{\includegraphics{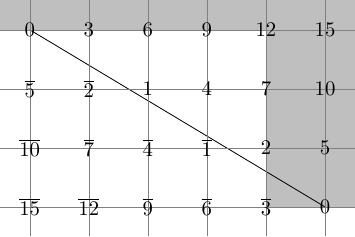}}};
\node (B) [fill=white!20,rounded corners=2pt,inner sep=1pt] at (-1.5,2.5) {\scalebox{0.7}{\includegraphics{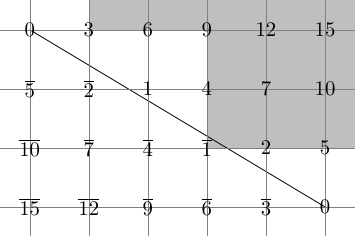}}};
\node (B2) [fill=white!20,rounded corners=2pt,inner sep=1pt] at (-1.5,4) {\scalebox{0.7}{\includegraphics{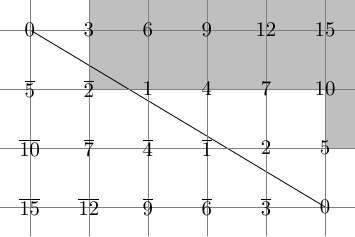}}};
\node (C) [fill=white!20,rounded corners=2pt,inner sep=1pt] at (0,1.5) {\scalebox{0.7}{\includegraphics{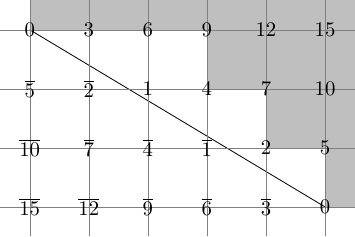}}};
\node (start) [fill=white!20,text=red,rounded corners=2pt,inner sep=1pt] at (-1.5,2) {start cycle here};
\draw[->] (B) to [bend right]  node[midway, left] {$-1$} (A);
\draw[->,red,thick,dotted] (C) to [bend right] node[midway, left] {$0$} (A);
\draw[->,red,thick,dotted] (A) to [bend right] node[midway, right] {$2$} (C);
\draw[->] (B2) to node[midway, above] {$5$} (E);
\draw[->,red,thick,dotted] (B2) to node[midway, right] {$-2$} (C);
\draw[->] (C) to node[midway, left] {$4$} (D2);
\draw[->,red,thick,dotted] (B) to node[midway, left] {$3$} (B2);
\draw[->] (D) -- (B) node[midway, above] {$1$};
\draw[->,red,thick,dotted] (C) -- (B) node[midway, below] {$2$};
\draw[->] (D) -- (C) node[midway, below] {$0$};
\draw[->] (A) to [bend right] node[midway, right] {$3$} (D);
\draw[->] (E) -- (D2) node[midway, above] {$-3$};
\draw[->] (D2) -- (D) node[midway, right] {$-1$};
\draw[->] (D2) -- (B2) node[midway, above] {$1$};
\draw [->] (A) edge[loop below]node{$1$} (A);
\end{tikzpicture}}
\caption{The directed graph $\mathfrak{F}_3^5 \cong \mathfrak{F}_5^3$, with equivalence classes represented by the balanced $(3,5)$-filters of \protect\Cref{sec:balanced}.  The edge labels record the level of the minimal element removed.  The cycle consisting of the dotted red edges corresponds to the parking $(m,n)$-filter tuple considered in \protect\Cref{rem:park_cycles}.  Compare with \protect\Cref{fig:perms_balanced2}.}
\label{fig:balanced2}
\end{figure}

\subsection{Representatives}

In this section, we introduce two natural representatives of the equivalence classes of $(m,n)$-filters:
\begin{itemize}
    \item Dyck $(m,n)$-filters, in bijection with Dyck paths and most useful to relate our constructions to the standard combinatorial objects (\Cref{rem:dyck_pic,rem:park_pic}); and
    \item balanced $(m,n)$-filters, which will be essential for specifying affine permutations (\Cref{thm:balanced_eq_sommers}, \Cref{thm:dominant_bij}, and~\Cref{thm:and_bij})
\end{itemize}

\subsubsection{Dyck filters}
\label{sec:dyck}
We define a first representative of the equivalence classes in $\EF_m^n$.  These representatives are usually defined in the literature as lattice paths staying above or below a diagonal, and we show how our definition recovers this interpretation in~\Cref{rem:dyck_pic}.

\begin{definition}
A \defn{Dyck $(m,n)$-filter} is an $(m,n)$-filter $\del$ such that \[\min_{(i,j) \in \del} \ell(i,j) = 0.\]  We write $\Dyck_m^n$ for the set of all Dyck $(m,n)$-filters.
\label{def:dyck_filter}
\end{definition}

In particular, for $\del \in \Dyck_m^n$, $\min(n(\del))=\min(m(\del))=0.$  Note that the $(m{\leftrightarrow}n)$-bijection restricts to a bijection between $\Dyck_m^n$ and $\Dyck_n^m$.

\begin{remark}
We relate~\Cref{def:dyck_filter} to the set of \defn{$(m,n)$-Dyck paths}---those lattice paths from $(0,0)$ to $(-n,m)$ using north steps $(0,1)$ and west steps $(-1,0)$ and staying above the line $(x,y)\cdot(m,n)=0$.  The boundary of an $(m,n)$-filter of $\ZZ \times \ZZ$ traces out a periodic path in the plane.  This periodicity allows us to restrict to the contiguous $n \times m$ rectangle with corners at level $0$ without losing information, giving $\Dyck_m^n$ the standard geometric interpretation as $(m,n)$-Dyck paths.  This is illustrated in~\Cref{fig:filters,fig:dyck_paths}.
\label{rem:dyck_pic}
\end{remark}

\begin{figure}[htb]
\centering
\includegraphics{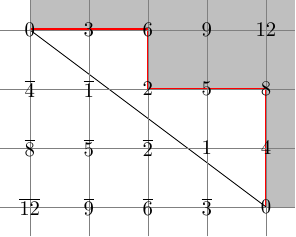} \hfill \includegraphics{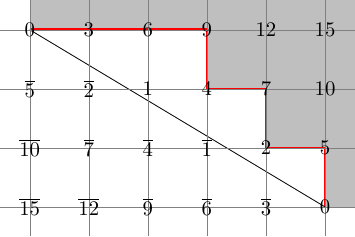}
\caption{The Dyck $(3,4)$- and $(3,5)$-filters corresponding to the filters in~\Cref{fig:filters} (they also happen to be balanced).  The boundary between two consecutive points with level $0$ traces an $(m,n)$-Dyck path (marked in red).  Recording the column lengths of the left path gives the $(3,4)$-Dyck word $[0,0,1,1]$, while the right path corresponds to the $(3,5)$-Dyck word $[0,0,0,1,2]$}
\label{fig:dyck_paths}
\end{figure}

All $(m,n)$-filters whose boundaries trace out the same path---up to translation---are equivalent to the same Dyck $(m,n)$-filter.

\begin{proposition}
Each equivalence class in $\EF_m^n$ contains a unique element of $\Dyck_m^n$.
\end{proposition}

\begin{proof}
To a $(m,n)$-filter $\idl$ we associate the unique equivalent Dyck $(m,n)$-filter $\del$ obtained by translating $\idl$ so that it touches the line $(x,y)\cdot(m,n)=0$, but does not go below.
\end{proof}
\begin{lemma}\label{lem:connected}
There is a directed path in $\mathfrak{F}_m^n$ from the equivalence class $\widetilde{\bal}_m^n$ (containing the $(m,n)$-filters generated by a single level) to any other equivalence class.
\end{lemma}

\begin{proof}
Starting from $\widetilde{\bal}_m^n$, the directed graph $\mathfrak{F}_m^n$ contains a copy of the distributive lattice (whose Hasse diagram is thought of as a directed graph) consisting of the $(m,n)$-Dyck paths ordered by inclusion.  (Note that this is via the identification with Dyck paths which lie below the diagonal, not above it.)
\end{proof}

The following enumeration is a well-known application of the cycle lemma.

\begin{proposition}\label{thm:dyck_paths}  For $m$ and $n$ relatively prime,
\[\left|\EF_m^n\right| = \left|\EF_n^m\right| = \frac{1}{n+m}\binom{n+m}{n}.\]
\end{proposition}

\subsubsection{Balanced Filters}
\label{sec:balanced}

We define a second representative of the equivalence classes in $\EF_m^n$.  These objects appear to have been much less studied, and will allow us to relate $\EF_m^n$ and affine permutations.

\begin{definition}
We call an $(m,n)$-filter $\bal \in \FF_m^n$ satisfying
\[\sum m(\bal):= \sum_{ i \in m(\bal)} i = \binom{m+1}{2} \text{ and } \sum n(\bal):=\sum_{ j \in n(\bal)} j = \binom{n+1}{2}\]
a \defn{balanced $(m,n)$-filter}.  We write $\Bal_m^n$  for the set of balanced $(m,n)$-filters
\label{prop:balanced1}
\end{definition}

It is a simple check that this set is nonempty---it contains the $(m,n)$-filter $\bal_m^n$ generated by the points with level $\ell=\frac{1+m+n-mn}{2}$.

\begin{proposition}
\label{prop:balanced2}
Each equivalence class in $\EF_m^n$ contains a unique element of $\Bal_m^n$.  Furthermore, for $\bal \in \FF_m^n$, \[ \sum m(\bal) = \binom{m+1}{2} \text{ if and only if } \sum n(\bal) = \binom{n+1}{2}.\]
\end{proposition}

\begin{proof} We first show that any equivalence class in $\EF_m^n$ contains an element of $\Bal_m^n$. Let $\idl$ be a balanced filter. Removing a minimal element from $\idl$ to make a new filter $\idl'$ has the effect of adding $m$ to the sum of elements of $m(\idl)$ and the effect of adding $n$ to the sum of the elements of $n(\idl)$. Rebalancing, the $(m,n)$-filter defined by $\ell(\idl'')=\ell(\idl')-1$ therefore is also balanced. By starting with $\bal_m^n$ and applying~\Cref{lem:connected}, we conclude that it is possible to find a balanced filter in each equivalence class of $\EF_m^n$. 

Now, if $\bal \neq \bal'$ are in the same equivalence class, then $\sum m(\bal) \neq \sum m(\bal')$ and $\sum n(\bal) \neq \sum n(\bal')$. This shows than any element of a given equivalence class other than the balanced element we found above, satisfied neither that $\sum m(\bal) = \binom{m+1}{2}$ nor that $\sum n(\bal) = \binom{n+1}{2}$. This completes the proof of the proposition. \end{proof}

 The five balanced $(3,4)$-filters are illustrated in~\Cref{fig:balanced}.

\subsection{Filter Tuples}
\label{sec:tuples}

We define $(m,n)$-filter tuples as certain sequences of $(m,n)$-filters, and we explain in~\Cref{rem:park_pic} how $(m,n)$-filter tuples are in bijection with the usual definition of parking functions as labeled Dyck paths.

\begin{definition}
An \defn{$(m,n)$-filter tuple} $\prk$ is a tuple of $n{+}1$ $(m,n)$-filters \[\prk=\left(\prk^{(0)},\prk^{(1)},\ldots,\prk^{(n)}\right)\] such that:
\begin{itemize}
    \item for $0 \leq i < n$, $m(\prk^{(i+1)})$ is obtained from $m(\prk^{(i)})$ by removing some $p_{i} \in m(\prk^{(i)})$ and inserting $p_i+m$ (pictorially, as in~\Cref{def:directed_graph}, this means that $\prk^{(i+1)}$ is obtained from $\prk^{(i)}$ by removing a corner), and
    \item $m(\prk^{(0)})+n = m(\prk^{(n)})$ (pictorially, this means that $\prk^{(n)}$ is obtained by displacing $\prk^{(0)}$ one step upwards).
\end{itemize}
We write $\mathcal{T}_m^n$ for the set of all $(m,n)$-filter tuples and we say that two $(m,n)$-filters tuples $\prk_1$ and $\prk_2$ are \defn{equivalent} if $\prk_1^{(i)}=\prk_2^{(i)}+(x,y)$ for all $0\leq i \leq n$ and some fixed $(x,y) \in \mathbb{Z}\times\mathbb{Z}$.
\label{def:parking}
\end{definition}

\Cref{def:parking} is illustrated in~\Cref{fig:parking_functions}; the caption is explained in the next few paragraphs.
\begin{figure}[htb]
\[
  \raisebox{-0.5\height}{\scalebox{.6}{\includegraphics{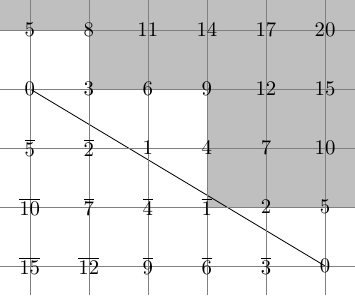}}} \xmapsto{3} \raisebox{-0.5\height}{\scalebox{.6}{\includegraphics{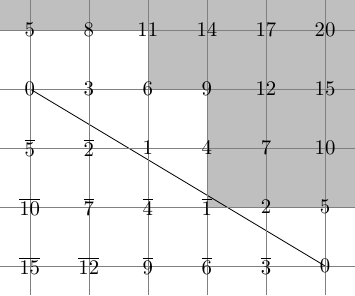}}} \xmapsto{\overline{1}} \raisebox{-0.5\height}{\scalebox{.6}{\includegraphics{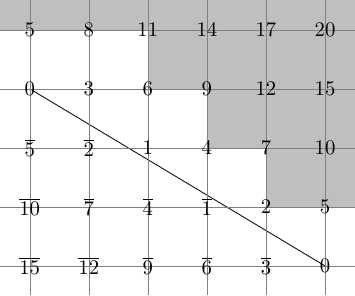}}} \xmapsto{2}\]
\[    \raisebox{-0.5\height}{\scalebox{.6}{\includegraphics{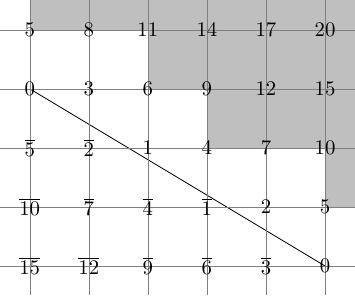}}} \xmapsto{5}\raisebox{-0.5\height}{\scalebox{.6}{\includegraphics{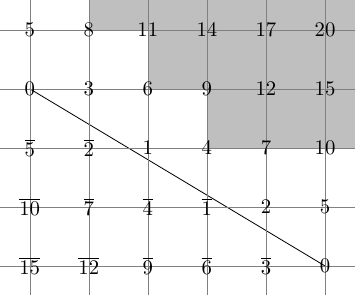}}} \xmapsto{6} \raisebox{-0.5\height}{\scalebox{.6}{\includegraphics{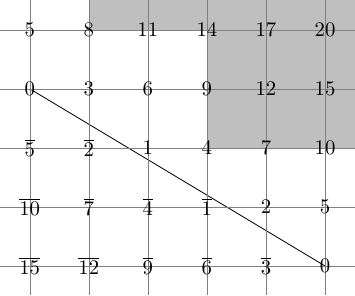}}}\phantom{\xmapsto{2}}\]
\caption{A balanced $(3,5)$-filter tuple $\prk$ with $n(\prk)=[3,-1,2,5,6]$ corresponding to the cycle consisting of the dotted red edges in \protect\Cref{fig:balanced2}.}
\label{fig:parking_functions}
\end{figure}

An $(m,n)$-filter tuple may be equivalently thought of as a cycle of length $n$ in the directed graph $\mathfrak{F}_n^m$ of~\Cref{def:directed_graph} with a choice of initial representative in the first equivalence class, as in~\Cref{rem:park_cycles}.

An $(m,n)$-filter tuple $\prk$ is specified by the sequence of the $n$ levels removed: \begin{equation}n(\prk) = [p_0,p_1,\ldots,p_{n-1}].\label{eq:n_of_park}\end{equation}
\Cref{def:parking} ensures that $n(\prk)$ is a permutation of $n(\prk^{(0)})$, such that levels in the same residue class modulo $m$ appear in increasing order.

\begin{example}
\Cref{fig:balanced2} illustrates a cycle of length $5$ in $\mathfrak{F}_5^3$: start at the vertex labeled by the balanced $(3,5)$-filter $\bal$ with $m(\bal)=[-1,3,4]$ and then follow the red edges.  This cycle corresponds to the $(3,5)$-filter tuple $\prk$ with $n(\prk)=[3,-1,2,5,6]$ in~\Cref{fig:parking_functions}.
\label{rem:park_cycles}
\end{example}

We call $\prk \in \mathcal{T}_m^n$ \defn{parking} if $\prk^{(0)} \in \Dyck_m^n$ and write $\PF_m^n$ for the set of all parking $(m,n)$-filter tuples.  We call $\prk \in \mathcal{T}_m^n$ \defn{balanced} if $\prk^{(0)}\in \Bal_m^n$ and write $\BF_m^n$ for the set of all balanced $(m,n)$-filter tuples.  \Cref{thm:dyck_paths,prop:balanced2} show that any $(m,n)$-filter tuple is equivalent to a unique  parking $(m,n)$-filter tuple and a unique balanced $(m,n)$-filter tuple.

\begin{remark}
We relate~\Cref{def:parking} to the definition of \defn{$(m,n)$-parking paths}---$(m,n)$-Dyck paths whose $n$ horizontal edges are labeled $1,2,\ldots,n$, such that levels in the same row increase from left to right.  Fix $\prk\in \PF_m^n$, so that $\prk^{(0)}$ may be thought of as an $(m,n)$-Dyck path by~\Cref{rem:dyck_pic}. Number each horizontal step in this path by the order in which its left endpoint is removed in $\prk$.  Since $\prk^{(i)}$ is an $(m,n)$-filter, points in the same row must be removed in order---this recovers the condition on levels for parking paths, as illustrated in~\Cref{fig:parking_function} (which corresponds to the parking $(m,n)$-filter tuple of~\Cref{fig:parking_functions}).  Thus, we may represent a parking $(m,n)$-filter as an $(m,n)$-parking path.
\label{rem:park_pic}
\end{remark}

\begin{figure}[htb]
\centering
\includegraphics{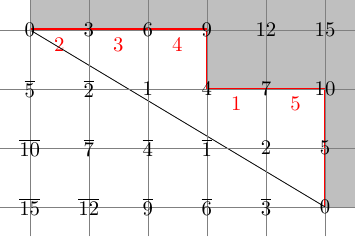}\hfill
\scalebox{0.9}{\includegraphics{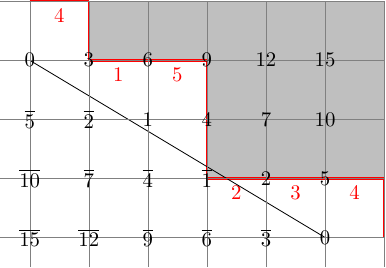}}
\caption{On the left is the parking $(3,5)$-filter tuple corresponding to the balanced $(3,5)$-filter tuple in~\Cref{fig:parking_functions}, encoded in the traditional manner as a Dyck path with labeled horizontal steps (the path is marked in red).  The labels record the order in which the points to their left were removed.  On the right is the corresponding balanced $(3,5)$-filter tuple.}
\label{fig:parking_function}
\end{figure}

The following enumerative result follows from the cycle lemma, and is given geometric meaning in~\Cref{sec:filters_and_sommers}.

\begin{proposition}[{\cite[Corollary 4]{armstrong2016rational},\cite{aval2015interlaced}}]
\label{thm:number_parking_functions}
For $m$, $n$ coprime,
    \[\left|\PF_m^n\right| = m^{n-1} \text{ and } \left|\PF_n^m\right| = n^{m-1}.\]
\end{proposition}

\section{The Zeta Map}
\label{sec:combinatorial_zeta}
After reviewing the state of the art for zeta maps in~\Cref{sec:zeta_history}, we use the combinatorial objects of~\Cref{sec:parking_functions} to define two (different) bijections between parking $(m,n)$-filter tuples and $(m,n)$-parking words (\Cref{def:parka,def:parkb})---the first map is trivially a bijection, but we only conclude that the second map is a bijection as a corollary of~\Cref{thm:park_char} in~\Cref{thm:b_is_bij}.  The composition of these two bijections defines the zeta map for rational parking words (\Cref{def:zetamap}).

In~\Cref{sec:sweep}, we show that our zeta map on rational words recovers Armstrong, Loehr, and Warrington's sweep map on rational Dyck paths using a canonical injection of Dyck paths inside parking paths.

\subsection{Context and History}
\label{sec:zeta_history}
The classical zeta map $\zeta$ is a bijection from $(n{+}1,n)$-Dyck paths to themselves developed by Garsia, Haglund, and Haiman to explain the equidistribution on Dyck paths of $(\mathsf{area},\mathsf{bounce})$ and $(\mathsf{dinv},\mathsf{area})$.  The statistic $\mathsf{bounce}$ is due to Haglund, while $\mathsf{dinv}$ is due to Haiman; we shall not review their definitions here.  This equidistribution expresses the agreement of the two formulas on the righthand side of the following combinatorial expansion of the Hilbert series of the \emph{alternating} subspace of the space of diagonal coinvariants~\cite{garsia2002proof,haiman2002vanishing,carlsson2015proof}:

\begin{align*}
\mathrm{Hilb}\left(\left(\CC[\xx_n,\mathbf{y}_n] / \langle \CC[\xx_n,\mathbf{y}_n]_+^{\Sym_n} \rangle\right)^\epsilon; q,t\right) &=\sum_{\mathsf{d} \in \DyckW_{n{+}1}^n} q^{\area(\mathsf{d})} t^{\bounce(\mathsf{d})}\\ &= \sum_{\mathsf{d} \in \DyckW_{n{+}1}^n} q^{\dinv(\mathsf{d})} t^{\area(\mathsf{d})},
\end{align*}
With the proper conventions\footnote{For consistency with its generalization to parking paths, we are using the inverse of the zeta map from~\cite[Theorem 3.15]{haglund2008q}.}, the map $\zeta$ explains the equidistribution of these statistics, in the sense that $\area(\mathsf{d})=\dinv(\zeta(\mathsf{d}))$ and $\bounce(\mathsf{d}) = \area(\zeta(\mathsf{d}))$.

From the point of view of lattice path combinatorics, the Dyck paths encoding the Hilbert series of the alternating subspace of the space of diagonal coinvariants are much simpler than the parking paths encoding the full Hilbert series of the space of diagonal coinvariants.  Presumably due to this difference in complexity, the definition and study of the zeta map was restricted to Dyck paths at first~\cite{haglund2003conjectured,garsia2002proof}, and its extension by Haglund and Loehr~\cite{haglund2005conjectured}\footnote{Although this bijection is between two slightly different manifestations of parking paths.} and by Loehr and Remmel~\cite{loehr2004conjectured} to parking paths only came later:

\begin{align*}
\mathrm{Hilb}\left(\CC[\xx_n,\mathbf{y}_n] / \langle \CC[\xx_n,\mathbf{y}_n]_+^{\Sym_n} \rangle; q,t\right) &=\sum_{\pp \in \Park_{n{+}1}^n} q^{\area(\pp)} t^{\mathsf{pmaj}(\pp)}\\ &= \sum_{\pp \in \Park_{n{+}1}^n} q^{\dinv(\pp)} t^{\area(\pp)},
\end{align*}
where $\mathsf{pmaj}$ is a generalization of $\mathsf{bounce}$, $\area(\pp)=\dinv(\zeta(\pp))$, and $\mathsf{pmaj}(\pp) = \area(\zeta(\pp))$.  When restricted to Dyck paths, this result generalizes the zeta map on Dyck paths~\cite[Exercise 5.7]{haglund2008q}; see also our~\Cref{prop:sweep_and_zeta}.

As Dyck paths and parking paths were generalized to the Fuss $(kn+1,n)$, Dogolon $(kn-1,n)$, and rational $(m,n)$ cases, extensions of the zeta map were again first defined on Dyck paths, and only later for parking paths.  The definition of zeta on rational parking words turns out to be surprisingly simple, as we show in~\Cref{def:parka,def:parkb,def:zetamap}.  This definition appears in~\cite{gorsky2016affine}---but in a different language that we postpone to~\Cref{sec:sommers_parking_words}.

The table in~\Cref{fig:table_zeta} contains a historical summary of the definitions of zeta, where for brevity we have suppressed some details as to the exact generality of the maps involved---in the column with heading ``Type,'' we use ``Dyck'' or ``Parking'' to refer to the unlabeled or labeled case of lattice paths, respectively.  (We recommend~\cite{armstrong2015sweep} for a thorough survey of the literature on zeta maps defined on lattice paths, at least when the dimensions of the bounding rectangle are coprime.)

\begin{figure}[htb]
\centering
\begin{tabular}{|l|c|c|c|r|}\hline
 Authors & Reference & Type & Generality & Proof of Bijectivity\\ \hline \hline
\begin{tabular}{l} Garsia \\ Haiman \\ Haglund \end{tabular} & \cite{garsia2002proof} & Dyck & $(n{+}1,n)$ &  \cite{garsia2002proof} \\ \hline
\begin{tabular}{l}Loehr\end{tabular} & \cite{loehr2005conjectured} & Dyck & $(kn{+}1,n)$ &  \cite{loehr2005conjectured}\\ \hline
\begin{tabular}{l} Haglund \\  Loehr\end{tabular} & \cite{haglund2005conjectured} & Parking & $(n{+}1,n)$ &  \cite{haglund2005conjectured} \\ \hline
\begin{tabular}{l} Gorsky \\  Mazin \end{tabular} & \begin{tabular}{c}\cite{gorsky2013compactified}\\\cite{gorsky2014compactified}\end{tabular} & Dyck & coprime $(m,n)$ &  \begin{tabular}{r} \cite{gorsky2013compactified} for $(kn{\pm}1,n)$ \\  \cite{thomas2015sweeping} \\ {\bf This paper} \end{tabular}  \\ \hline
\begin{tabular}{l} Amstrong \\ Loehr \\  Warrington \end{tabular} & \cite{armstrong2015sweep} &  Dyck & \begin{tabular}{c} $N$-dim. box\\ (integer labels) \end{tabular} & \cite{thomas2015sweeping} \\ \hline
\begin{tabular}{l} Thomas \\ Williams \end{tabular} & \cite{thomas2015sweeping} & Dyck & \begin{tabular}{c} $N$-dim. box \\ (modular labels) \end{tabular} & \cite{thomas2015sweeping} \\ \hline
\begin{tabular}{l} Gorsky \\ Mazin \\  Vazirani \end{tabular} & \cite{gorsky2016affine} & Parking & coprime $(m,n)$ & \begin{tabular}{r} \cite{gorsky2013compactified} for $(kn{\pm}1,n)$ \\  {\bf This paper} \end{tabular}  \\ \hline
\begin{tabular}{l} Gorsky \\ Mazin \\  Vazirani \end{tabular} & \cite{gorsky2017rational} & Dyck & $(m,n)$ & \cite{gorsky2017rational} \\ \hline
\end{tabular}
\caption{A brief overview of various definitions and work on zeta maps.}
\label{fig:table_zeta}
\end{figure}
\medskip

\subsection{The Zeta Map}
\label{sec:zeta}

We define the zeta map using two bijections $A,B$ from parking $(m,n)$-filter tuples to $(m,n)$-parking words.  The zeta map is then defined to be the map $\zeta := B \circ A^{-1}$.

Following Gorsky, Mazin, and Vazirani, the $q$ and $t$ statistics may be read off these $(m,n)$-parking words~\cite{armstrong2013hyperplane,gorsky2016affine}: \begin{align*}\area(\prk)&:=\frac{(n-1)(m-1)}{2}-\sum_{i=0}^{n-1} A(\prk)_i, \\ \dinv(\prk)&:=\frac{(n-1)(m-1)}{2}-\sum_{i=0}^{n-1} B(\prk)_i.\end{align*}

\subsubsection{Area (A)}
\label{sec:mapa}

Our first map is a simple application of the interpretation in~\Cref{rem:park_pic} of an $(m,n)$-filter tuple as an $(m,n)$-parking path.  This will be useful again in~\Cref{sec:anderson2} in the context of affine permutations.

\begin{definition}
Define $A: \PF_m^n \to \Park_m^n$ to be the $(m,n)$-parking word recording the column lengths (in the order of the edge labels) of the $(m,n)$-parking path associated to $\prk$ by~\Cref{rem:park_pic}.
\label{def:parka}
\end{definition}

It is easy to see that $A$ may be equivalently defined by \[n(\prk)=[p_1,p_2,\ldots,p_n] \xmapsto{A} \left[ap_1, ap_2,\ldots, ap_n \right] \mod m,\] where $a n = -1 \mod m$. 

\begin{example} \label{ex:mapa}
The parking $(3,5)$-filter tuple $\prk \in \PF_3^5$ encoded by the $(3,5)$-parking path in~\Cref{fig:parking_function} is mapped to the $(3,5)$-parking word $A(\prk)=10001$ (there is one gray box in each column containing the horizontal edges with labels 1 and 5, and no gray boxes in the other columns).  We may also compute it using the word $n(\prk)$ from~\Cref{fig:parking_function}:
\[n(\prk) = [4,0,3,6,7] \xmapsto{A} [1\cdot 4,1\cdot 0,1\cdot 3,1\cdot 6,1\cdot 7] \mod 3 = [1,0,0,0,1],\] since $1 \cdot 5 = -1 \mod 3$.
We compute $\area(\prk) = \frac{2\cdot 4}{2}-\left(1+0+0+0+1\right) = 2.$

On the other hand, since $3 \cdot 3 = -1 \mod 5$, we compute the  $(5,3)$-parking word  for the element $\prk \in \PF_5^3$ with $n(\prk)=[0,4,5]$ to be:
\[n(\prk) = [0,4,5] \xmapsto{A} [3\cdot 0,3\cdot 4,3\cdot 5] \mod 5 = [0,2,0].\]
\end{example}

It is immediate from~\Cref{rem:park_pic} that~\Cref{def:parka} is a bijection from $\PF_m^n$ to $\Park_m^n$.

\subsubsection{Dinv ($B$)}
\label{sec:mapb}
Our second map is more subtle, requiring an application of \Cref{thm:park_char} to prove that it is well-defined.
\begin{definition}\label{def:parkb}
Define $B: \PF_m^n \to [m]^n$ to be the word $\w=\w_1\cdots \w_n \in [m]^n$ where $\w_{i+1}=j$ if $p_{i}$ is the $j$th smallest number in $m(\prk^{(i)})$.  That is, $B(\prk)$ is defined by recording the number of letters in $m(\prk^{(i)})$ strictly less than $p_{i}$ for $0\leq i < n$ (we call this number the \defn{position} of $p_i$ in $m(\prk^{(i)})$).
\end{definition}

\begin{example}\label{ex:b1}
As in~\Cref{rem:park_cycles}, we compute $m(\prk^{(i)})$ for each $(3,5)$-filter in~\Cref{fig:parking_functions} to be \[ [-1,\mathbf{3},4] \to [\mathbf{-1},4,6] \to [\mathbf{2},4,6] \to [4,\mathbf{5},6] \to [4,\mathbf{6},8] \to [4,8,9],\] where we haven't rebalanced (but note that this doesn't change relative order, and so won't change the image of $B$).  Recording the position of the elements removed (marked in bold above) gives the $(3,5)$-parking word $B(\prk)=10011.$
As with $\area$, we compute $\dinv(\prk) = \frac{2\cdot 4}{2}-\left(1+0+0+1+1\right) = 1.$
\end{example}

It is not obvious that~\Cref{def:parkb} really does produce $(m,n)$-parking words.

\begin{theorem}
The map $B$ is a bijection from  $\PF_m^n$ to $\Park_m^n$.
\label{thm:b_is_bij}
\end{theorem}

\begin{proof}
Let $\prk\in\PF_m^n$. For $0 \leq i\leq n$, we define a point $\xx^{(i)} \in V_0^m$ by $\mathbf{x}^{(i)} = m(\prk^{(i)}) $, and adding a multiple of $\mathbbm{1}$ so that the sum of the elements in $\xx^{(i)}$ is zero (since every element in $ m(\prk^{(i)})$  changes by the same amount, their relative order is preserved).  So the action of $B(\prk)$ on $\xx^{(0)} \in V_0^m$ (as defined in~\Cref{sec:new_char}) is recorded by the sequences $\xx^{(i)}$.  Finally, $\xx^{(0)} = \xx^{(n)}$ because $\prk$ is a parking $(m,n)$-tuple.  In particular, we have shown that the word $B(\prk)$ has a fixed point. Now~\Cref{thm:park_char} tells us that $B(\prk)$ is a parking word.

Further,~\Cref{thm:park_char} tells us that the fixed point of $B(\prk)$ is unique.  Therefore, from 
$B(\prk)$, we can identify its unique fixed point $\xx^{(0)}$, from which we can reconstruct 
$\xx^{(i)}$ for all $i$, and thus $\prk^{(i)}$ for all $i$.  That is to say, from $B(\prk)$, we can reconstruct $\prk$.  This implies that the map $B$ is an injection from $\PF_m^n$ to $\Park_m^n$.  We have already established that the map $A$ is a bijection between these two sets, so the fact that $B$ is an injection means that it must also be surjective.
\end{proof}

Given an $(m,n)$-filter tuple $\prk$, \emph{the} fixed point for $B(\prk)$ in $V_0^m$ is the word $m(\prk^{(0)})$---up to addition of a multiple of $\mathbbm{1}$.

\begin{example}
Continuing~\Cref{ex:b1} (and recalling~\Cref{rem:park_cycles}), balancing each $(3,5)$-filter $\prk^{(i)}$ gives the sequence of $m(\prk^{(i)})$
\[ [-1,\mathbf{3},4] \to [\mathbf{-2},3,5] \to [\mathbf{0},2,4] \to [1,\mathbf{2},3] \to [0,\mathbf{2},4]\to [-1,3,4].\]  Balancing adds the same amount to each element, and thinking of $m(\prk^{(0)})$ as an element of $V_6^3$, we observe that $m(\prk^{(0)})$ is a fixed point for the action of $B(\prk)=10011$.
\end{example}

\subsubsection{The Zeta Map ($A \mapsto B$)}

The zeta map $\zeta$ sends the first method of associating an $(m,n)$-parking word to a parking $(m,n)$-filter tuple in~\Cref{def:parka} to the second in~\Cref{def:parkb}.
By~\Cref{thm:b_is_bij}, we conclude that $\zeta$ is a bijection.

\begin{definition}
\label{def:zetamap}
The \defn{zeta map} is the bijection from $\Park_m^n$ to itself defined by
\begin{align*}
\zeta: \Park_m^n &\to \Park_m^n \\
\mathsf{p} & \mapsto B \circ A^{-1}(\mathsf{p})
\end{align*}
\end{definition}

Examples are illustrated in~\Cref{fig:zeta34,fig:zeta35}.  The grid in~\Cref{fig:zeta34} gives the expansions of the $q,t$-Catalan and parking polynomials:
\begin{align*}
\sum_{\del \in \Dyck_4^3} q^{\area(\del)} t^{\dinv(\del)}  &= q^3+q^2t+qt+qt^2+t^3,\\
\sum_{\prk \in \PF_4^3} q^{\area(\prk)} t^{\dinv(\prk)} &= q^3+2q^2+q^2t+2q+3qt+qt^2+1+2t+2t^2+t^3.
\end{align*}

\renewcommand{\arraystretch}{1.2}

\begin{figure}[htb]
\centering
\begin{tabular}{ccc|ccc|ccc}
$n(\prk)$ & $A(\prk)$ & $B(\prk)$ & $n(\prk)$ & $A(\prk)$ & $B(\prk)$ & $n(\prk)$ & $A(\prk)$ & $B(\prk)$ \\ \hline \hline
123 \cellcolor{lightgray} & 012 \cellcolor{lightgray} & 000 \cellcolor{lightgray} & $\overline{1}34$ \cellcolor{lightgray} & 001 \cellcolor{lightgray} & 011 \cellcolor{lightgray} & 015 \cellcolor{lightgray} & 011 \cellcolor{lightgray} & 002 \cellcolor{lightgray} \\
132 & 021 & 010 & $\overline{1}43$ & 010 & 021 & 105 & 101 & 102 \\
213 & 102 & 100 & $4\overline{1}3$ & 100 & 201 & 150 & 110 & 120 \\ \cline{4-9}
231 & 120 & 110 & 024 \cellcolor{lightgray} & 020 \cellcolor{lightgray} & 001 \cellcolor{lightgray} & $\overline{2}26$ \cellcolor{lightgray} & 000 \cellcolor{lightgray} & 012 \cellcolor{lightgray} \\
312 & 201 & 200 & 042 & 002 & 020 & & \\
321 & 210 & 210 & 204 & 200 & 101 & &
\end{tabular}
\begin{align*}
\begin{array}{cc}
\begin{array}{c|cccc}
 & 1 & q & q^2 & q^3 \\ \hline
 1 & 0 & 0 & 0 & 1 \\
 t& 0 & 1 & 1 & 0 \\
 t^2& 0 & 1 & 0 & 0 \\
 t^3& 1 & 0 & 0 & 0 \\
\end{array}
&\begin{array}{c|cccc}
 & 1 & q & q^2 & q^3 \\ \hline
 1 & 1 & 2 & 2 & 1 \\
  t& 2 & 3 & 1 & 0 \\
 t^2&2 & 1 & 0 & 0 \\
  t^3&1 & 0 & 0 & 0 \\
\end{array}
\end{array}
\end{align*}
\caption{The zeta map on $\Park_4^3$, along with the $q,t$-Catalan and parking polynomials.  The rows shaded in gray correspond to the canonical embedding of $\Dyck_4^3$ in $\PF_4^3$ from~\Cref{rem:embed}.  The grids represent the $q,t$-Catalan and parking polynomial---the number in the column labeled $q^i$ and $t^j$ is the coefficient of $q^i t^j$ in the corresponding polynomial $\sum_{\del \in \Dyck_4^3} q^{\area(\del)} t^{\dinv(\del)}$ or $\sum_{\prk \in \PF_4^3} q^{\area(\prk)} t^{\dinv(\prk)}$.}
\label{fig:zeta34}
\end{figure}

\begin{figure}[htb]
\centering
\begin{tabular}{ccc|ccc|ccc|ccc}
$n(\prk)$ & $A(\prk)$ & $B(\prk)$ & $n(\prk)$ &  $A(\prk)$ & $B(\prk)$ & $n(\prk)$ &  $A(\prk)$ & $B(\prk)$ & $n(\prk)$ &  $A(\prk)$ & $B(\prk)$ \\ \hline \hline
123 \cellcolor{lightgray}& 031 \cellcolor{lightgray} & 000 \cellcolor{lightgray} & 024 \cellcolor{lightgray} & 012 \cellcolor{lightgray} & 001 \cellcolor{lightgray} & $\overline{2}35$ \cellcolor{lightgray} & 001 \cellcolor{lightgray} & 012 \cellcolor{lightgray} & $015$ \cellcolor{lightgray} & 030 \cellcolor{lightgray} & 002 \cellcolor{lightgray}\\
132 & 013 & 010 & 042 & 021 & 020 & $\overline{2}53$ & 010 & 031 & $051$ & 003 & 030\\
312 & 103 & 200 & 204 & 102 & 101 & $5\overline{2}3$ & 100 & 301 & $105$ & 300 & 102\\ \cline{7-12}
321 & 130 & 210 & 240 & 120 & 120 & $\overline{1}34$ \cellcolor{lightgray} & 020 \cellcolor{lightgray} & 011 \cellcolor{lightgray} & $\overline{1}16$ \cellcolor{lightgray} & 011 \cellcolor{lightgray} & 003 \cellcolor{lightgray}\\
213 & 301 & 100 & 402 & 201 & 300 & $\overline{1}43$ & 002 & 021 & $1\overline{1}6$ & 101 & 103\\
231 & 310 & 110 & 420 & 210 & 310 & $3\overline{1}4$ & 200 & 201 & $16\overline{1}$ & 110 & 130\\
\hline
&&&&&& & && $\overline{3}27$ \cellcolor{lightgray} & 000 \cellcolor{lightgray} & 013       \cellcolor{lightgray}
\end{tabular}
\begin{align*}
\begin{array}{cc}
\begin{array}{c|ccccc}
& 1 & q & q^2 & q^3 & q^4 \\ \hline
1 & 0 & 0 & 0 & 0 & 1 \\
t & 0 & 0 & 1 & 1 & 0 \\
t^2& 0 & 1 & 1 & 0 & 0 \\
t^3& 0 & 1 & 0 & 0 & 0 \\
t^4& 1 & 0 & 0 & 0 & 0 \\
\end{array} &
\begin{array}{c|ccccc}
& 1 & q & q^2 & q^3 & q^4 \\ \hline
1& 0 & 1 & 2 & 2 & 1 \\
t& 1 & 4 & 3 & 1 & 0 \\
t^2& 2 & 3 & 1 & 0 & 0 \\
t^3& 2 & 1 & 0 & 0 & 0 \\
t^4& 1 & 0 & 0 & 0 & 0 \\
\end{array}
\end{array}
\end{align*}
\caption{The zeta map on $\Park_5^3$.   The rows shaded in gray correspond to the canonical embedding of $\Dyck_5^3$ in $\PF_5^3$ from~\Cref{rem:embed}.  The grids represent the $q,t$-Catalan and parking polynomials, as in~\Cref{fig:zeta34}.}
\label{fig:zeta35}
\end{figure}


\subsection{The Sweep Map}
\label{sec:sweep}
In this section, we relate the zeta map on $(m,n)$-parking words to the sweep map on $(m,n)$-Dyck paths.

Having fixed $m$ and $n$ coprime, define the \defn{level} of a step of a lattice path in $\ZZ \times \ZZ$ to be the level of its north/west endpoint.  In~\cite{armstrong2015sweep}, Armstrong, Loehr, and Warrington defined the \defn{sweep map} on $(m,n)$-Dyck paths by \emph{sorting} the steps of a given path by their levels, that is to say, we reorder the steps of the path by increasing order of level.\footnote{This is a special case of the general definition of the sweep map, which is on general lattice paths in an $N$-dimensional box.}  See~\Cref{fig:sweep} for an example.  One can visualize this procedure geometrically as a \emph{sweep} of the line ${\mathcal H}_{\mathbf{a},k}:=\{\mathbf{x}:\mathbf{x} \cdot (m,n)=k\}$ up from  $k=0$ to $k=\infty$, as illustrated in~\Cref{fig:ratpaths} for $(m,n)=(4,7)$.

\begin{figure}[htb]
\centering
\raisebox{-0.5\height}{\scalebox{.6}{\includegraphics{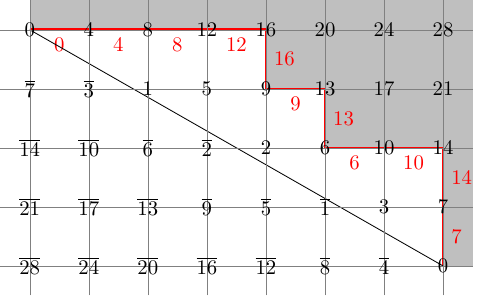}}} $\xmapsto{\swg}$ \raisebox{-0.5\height}{\scalebox{.6}{\includegraphics{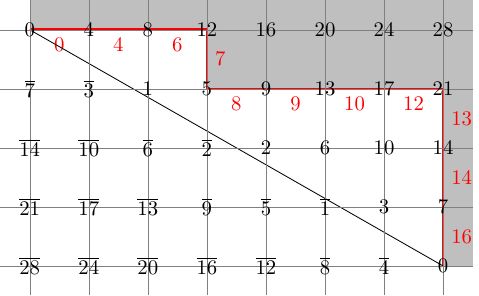}}}
\caption{The $(4,7)$-filters corresponding to a path $\del$ (left) and the corresponding path $\swg(\del)$ (right).  The horizontal steps of $\del$ are labeled by the level to their west, while the vertical steps are labeled by the level to their north; we have preserved these labels on the steps of $\swg(\del)$.  To form the path $\swg(\del)$, the steps of the path $\del$ are rearranged according to the order in which they are encountered by a line of slope $-4/7$ sweeping up from below.  Compare with~\Cref{fig:ratpaths}.}
\label{fig:sweep}
\end{figure}
\begin{figure}[htb]
\begin{alignat*}{4}
\framebox{\begin{tabular}{c}\raisebox{-0.5\height}{\reflectbox{\includegraphics[width=.12\linewidth]{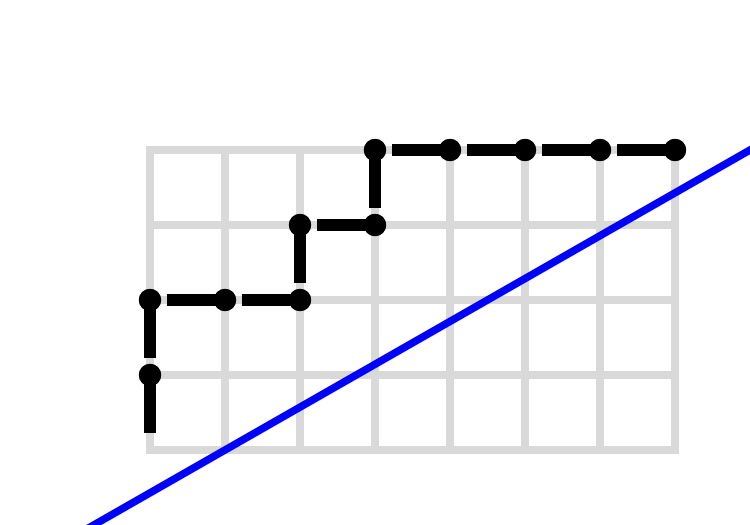}}\reflectbox{\includegraphics[width=.12\linewidth]{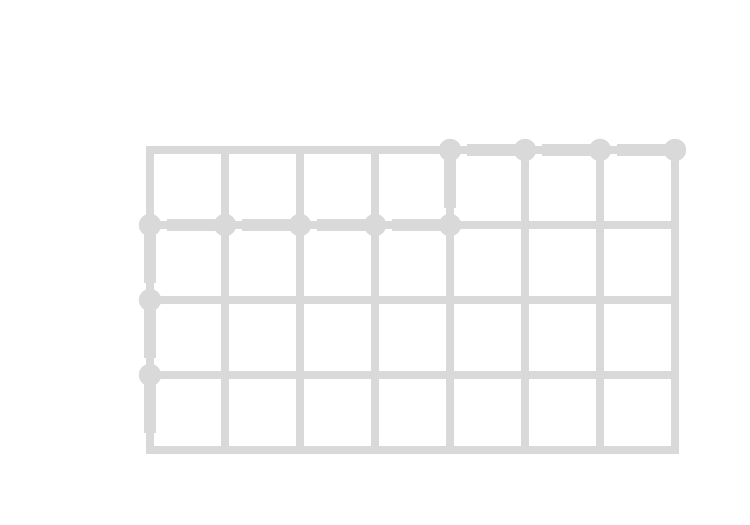}}} \\ \textcolor{red}{0} \ 6 \ 7 \ 9 \end{tabular}} &&\to
\framebox{\begin{tabular}{c}\raisebox{-0.5\height}{\reflectbox{\includegraphics[width=.12\linewidth]{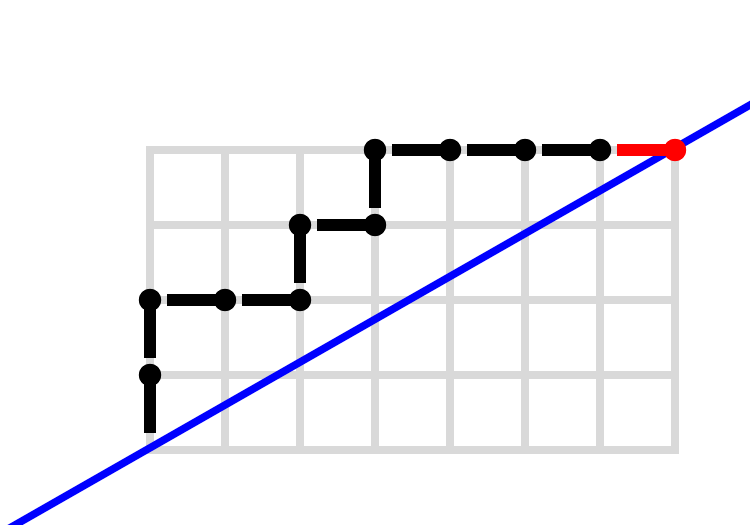}}\reflectbox{\includegraphics[width=.12\linewidth]{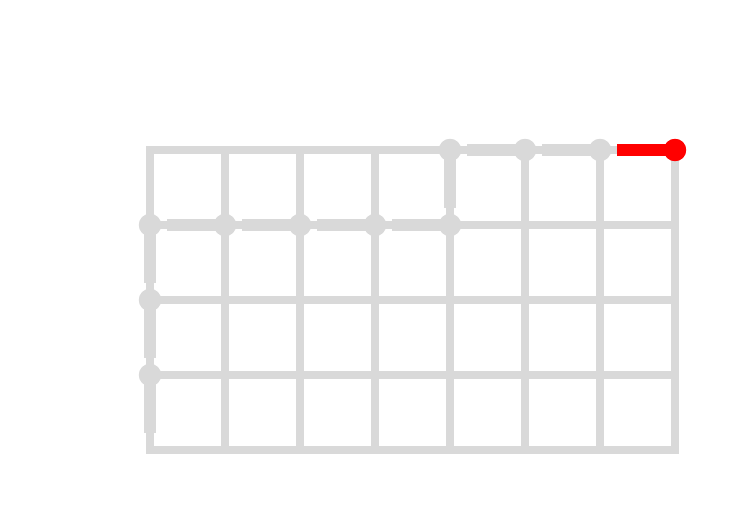}}}\\ \textcolor{red}{4} \ 6 \ 7 \ 9\end{tabular}} &&\to
\framebox{\begin{tabular}{c}\raisebox{-0.5\height}{\reflectbox{\includegraphics[width=.12\linewidth]{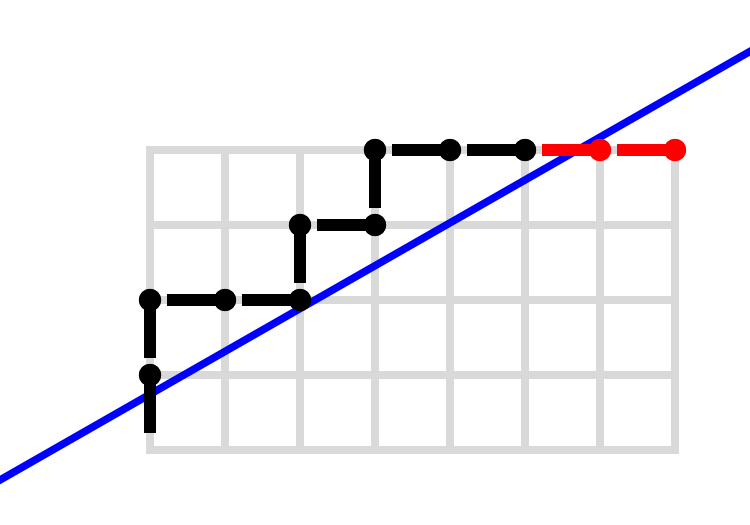}}\reflectbox{\includegraphics[width=.12\linewidth]{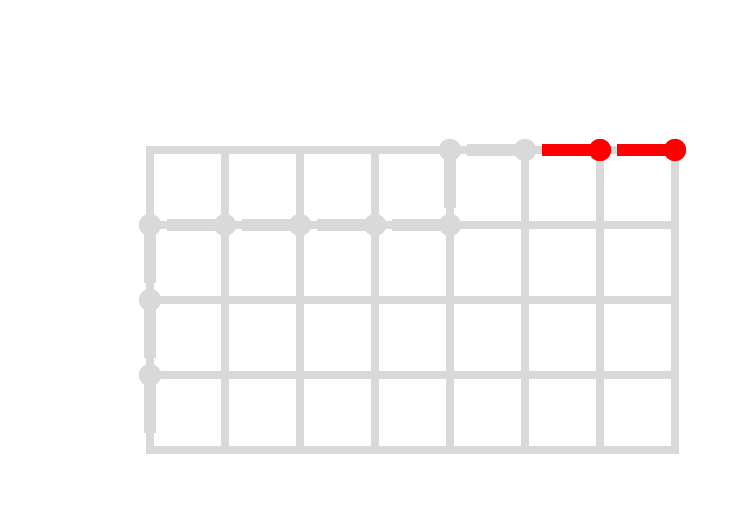}}}\\ \textcolor{red}{6} \ 7 \ 8 \ 9\end{tabular}} &&\to \\
\framebox{\begin{tabular}{c}\raisebox{-0.5\height}{\reflectbox{\includegraphics[width=.12\linewidth]{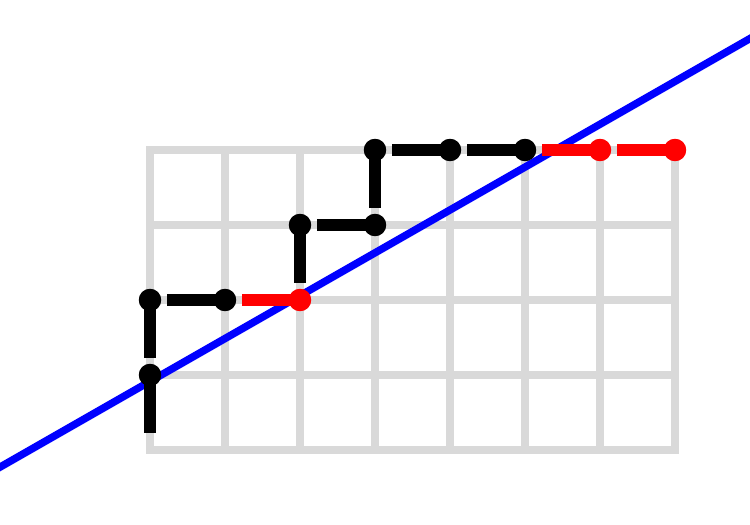}}\reflectbox{\includegraphics[width=.12\linewidth]{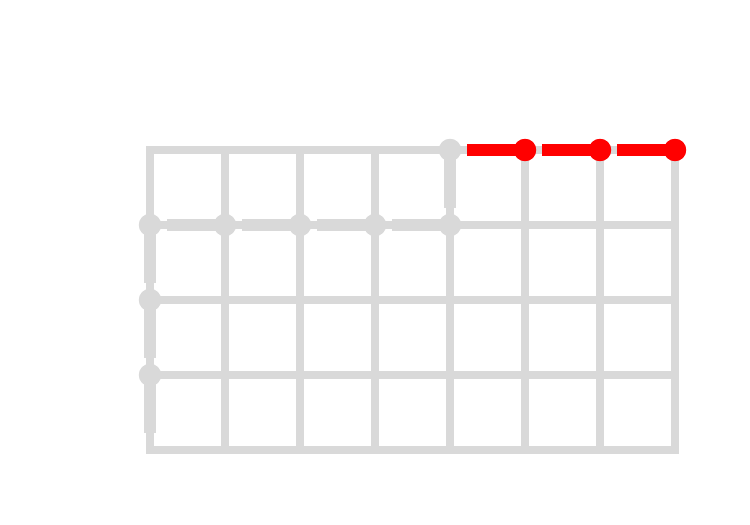}}}\\ \textcolor{red}{\bf 7} \ 8 \ 9 \ 10\end{tabular}} &&\to
\framebox{\begin{tabular}{c}\raisebox{-0.5\height}{\reflectbox{\includegraphics[width=.12\linewidth]{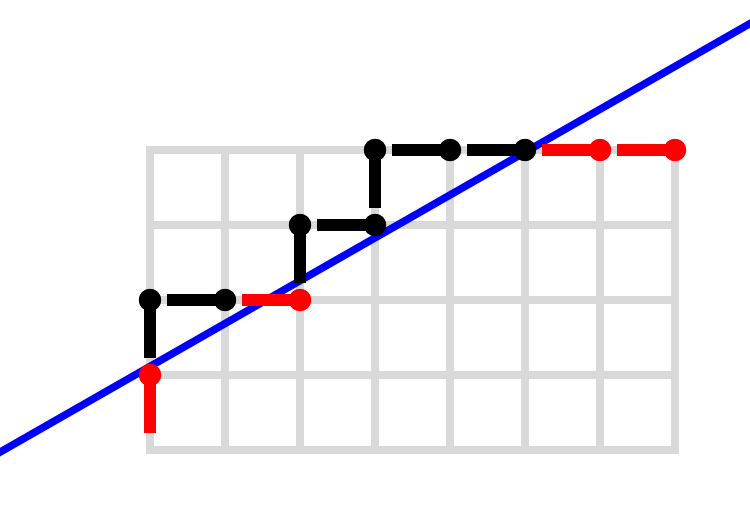}}\reflectbox{\includegraphics[width=.12\linewidth]{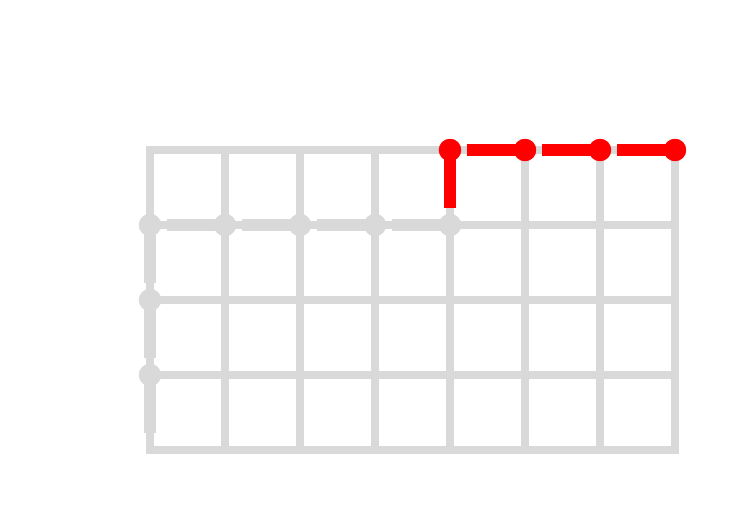}}} \\ {\bf 7} \ \textcolor{red}{8} \ 9 \ 10\end{tabular}} &&\to
\framebox{\begin{tabular}{c}\raisebox{-0.5\height}{\reflectbox{\includegraphics[width=.12\linewidth]{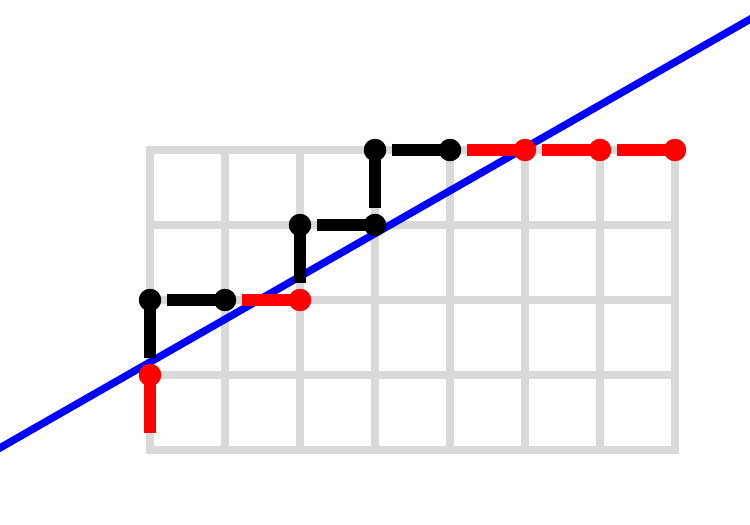}}\reflectbox{\includegraphics[width=.12\linewidth]{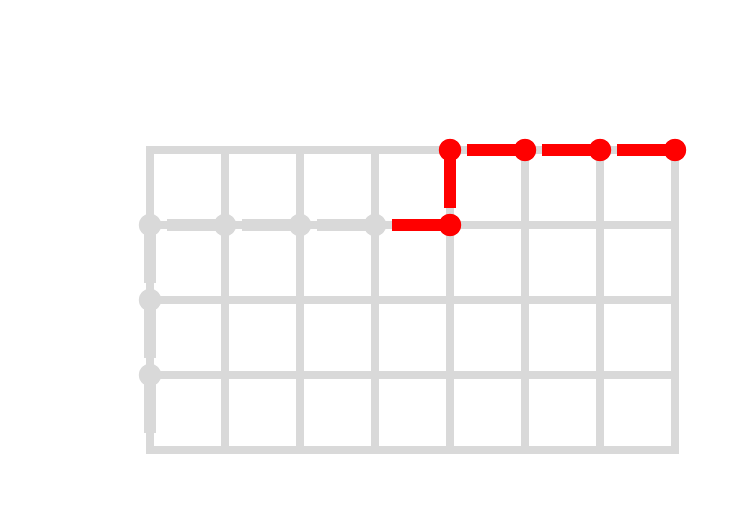}}}\\ {\bf 7} \ \textcolor{red}{9} \ 10 \ 12\end{tabular}} &&\to \\
\framebox{\begin{tabular}{c}\raisebox{-0.5\height}{\reflectbox{\includegraphics[width=.12\linewidth]{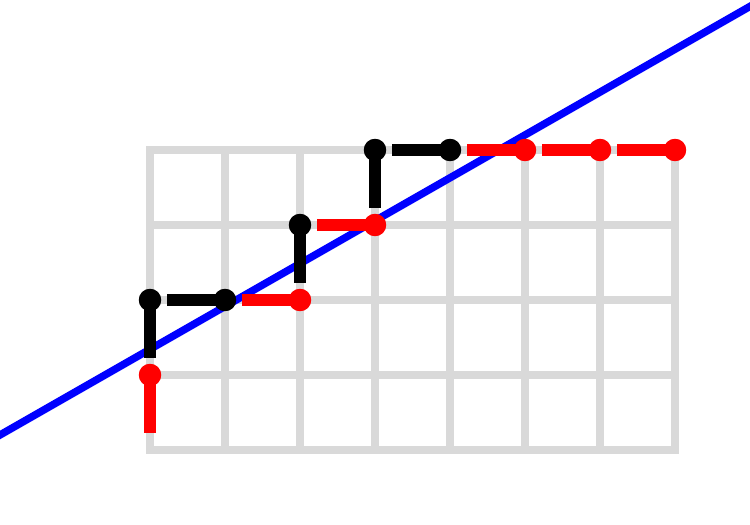}}\reflectbox{\includegraphics[width=.12\linewidth]{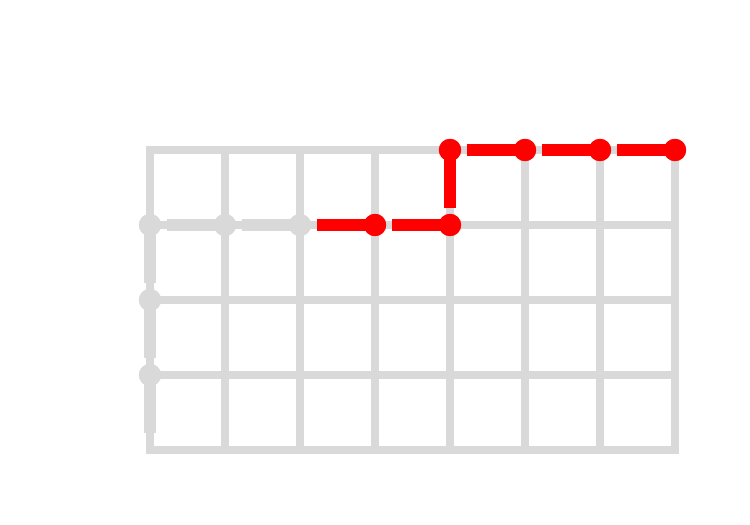}}}\\ {\bf 7} \ \textcolor{red}{10} \ 12 \ {13}\end{tabular}} &&\to
\framebox{\begin{tabular}{c}\raisebox{-0.5\height}{\reflectbox{\includegraphics[width=.12\linewidth]{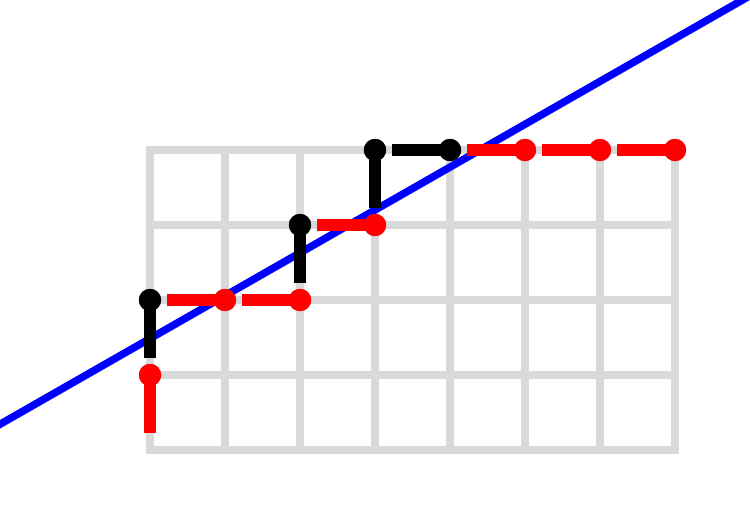}}\reflectbox{\includegraphics[width=.12\linewidth]{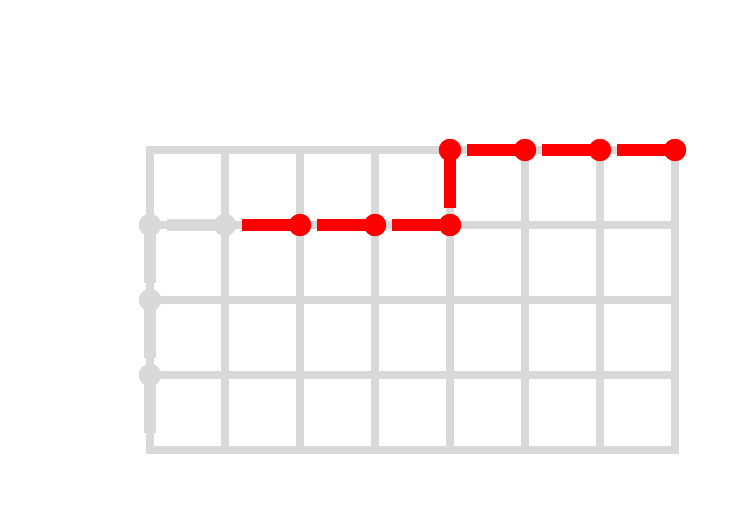}}}\\ {\bf 7} \ \textcolor{red}{12} \ 13 \ {14}\end{tabular}} &&\to
\framebox{\begin{tabular}{c}\raisebox{-0.5\height}{\reflectbox{\includegraphics[width=.12\linewidth]{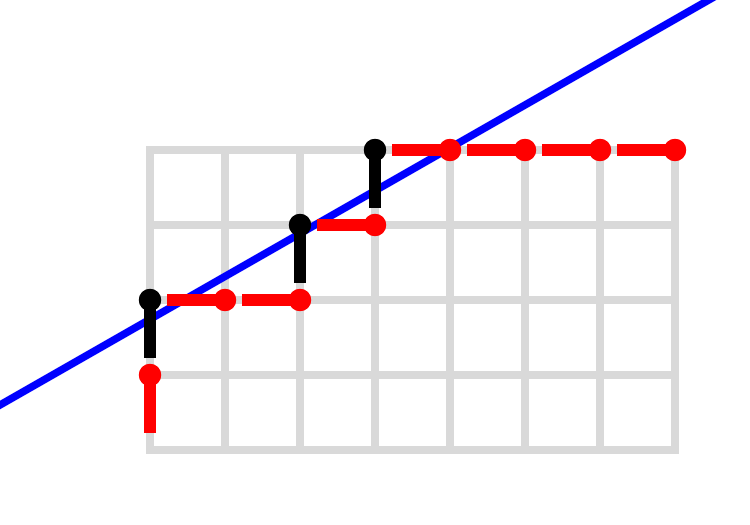}}\reflectbox{\includegraphics[width=.12\linewidth]{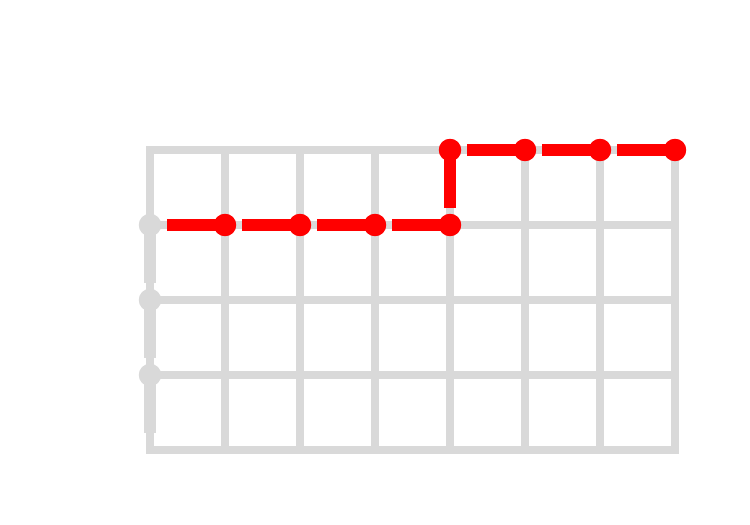}}}\\ {\bf 7} \ \textcolor{red}{\bf 13} \ 14 \ {16}\end{tabular}} &&\to \\
\framebox{\begin{tabular}{c}\raisebox{-0.5\height}{\reflectbox{\includegraphics[width=.12\linewidth]{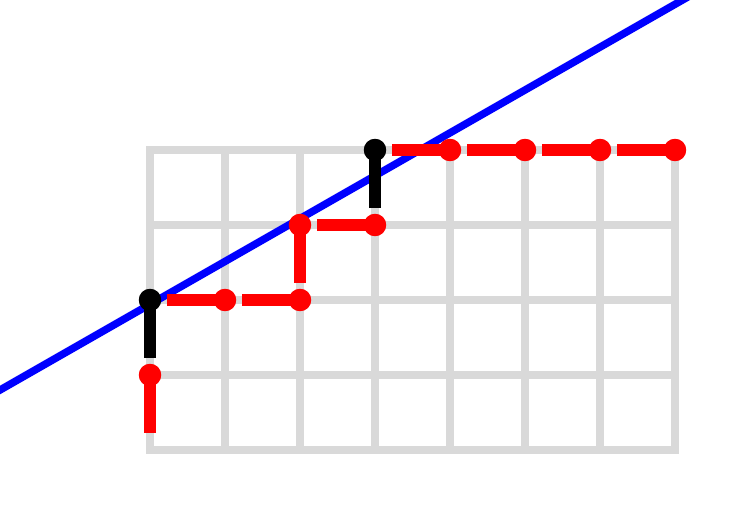}}\reflectbox{\includegraphics[width=.12\linewidth]{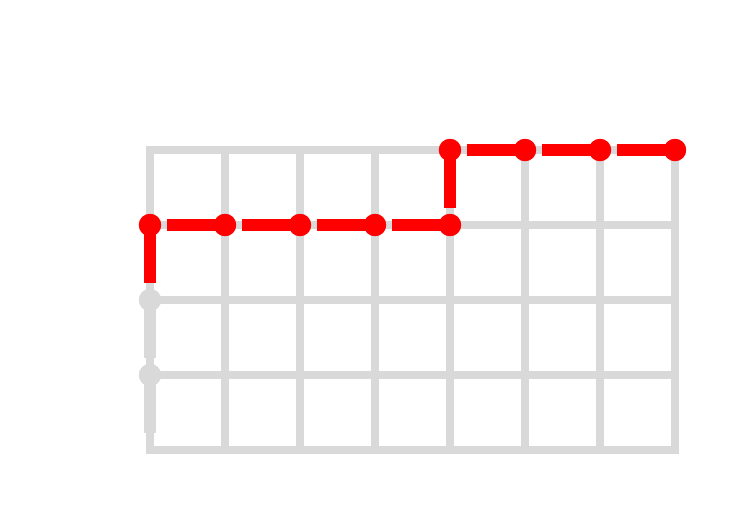}}}\\ {\bf 7} \ {\bf 13} \ \textcolor{red}{\bf 14} \ 16\end{tabular}} &&\to
\framebox{\begin{tabular}{c}\raisebox{-0.5\height}{\reflectbox{\includegraphics[width=.12\linewidth]{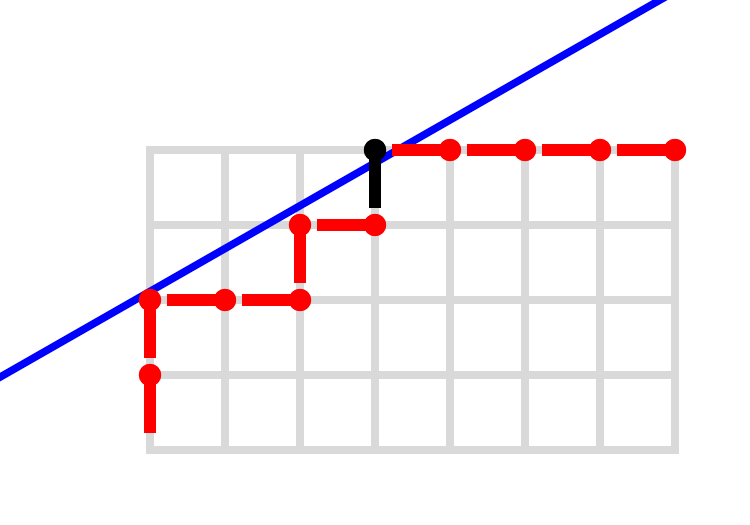}}\reflectbox{\includegraphics[width=.12\linewidth]{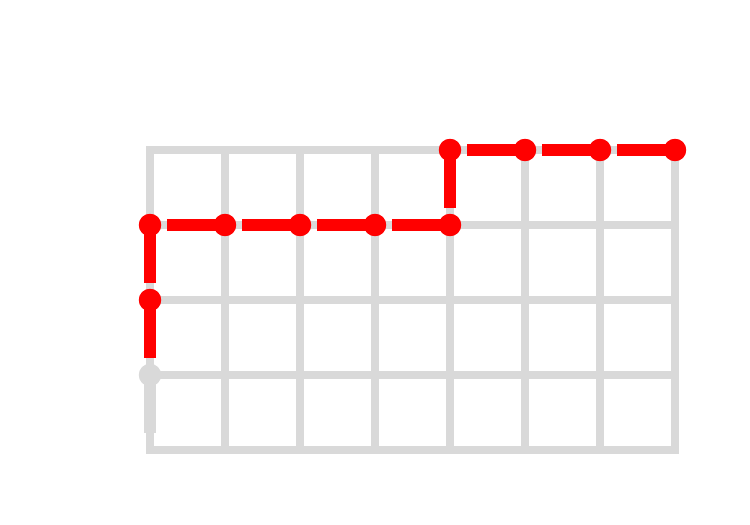}}}\\ {\bf 7} \ {\bf 13} \ {\bf 14} \ \textcolor{red}{\bf 16}\end{tabular}} &&\to
\framebox{\begin{tabular}{c}\raisebox{-0.5\height}{\reflectbox{\includegraphics[width=.12\linewidth]{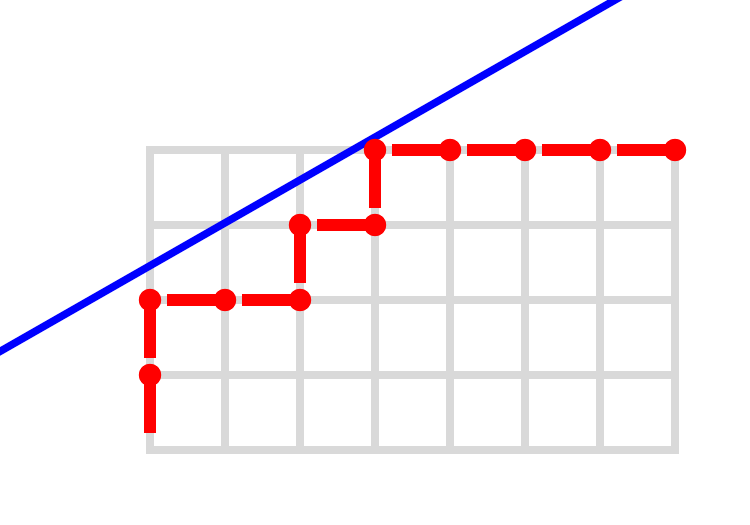}}\reflectbox{\includegraphics[width=.12\linewidth]{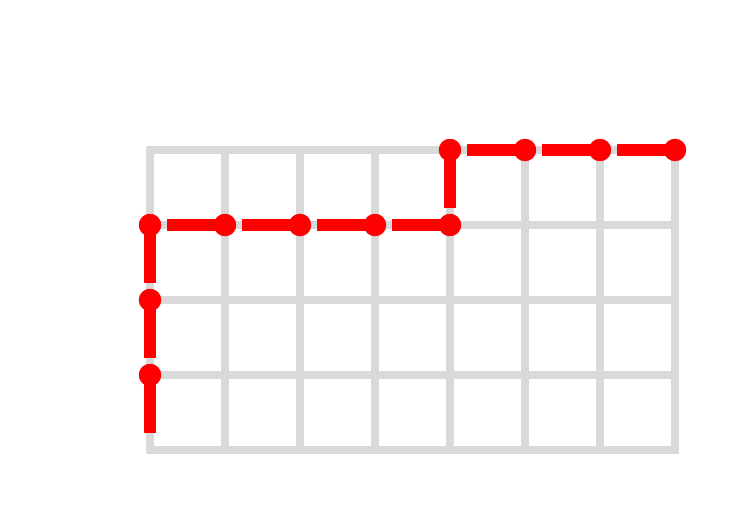}}}\\ {\bf 7} \ {\bf 13} \ {\bf 14} \ {\bf 16} \end{tabular}} &&\makebox[\widthof{$\to$}][c]{$.$}
\end{alignat*}
\caption{An illustration of the geometric interpretation of $\swg$ for $(m,n)=(4,7)$ and the path $\del$ of~\Cref{fig:sweep}.  Each box in the figure corresponds to a step of the sweeping procedure.  Each box contains the path $\del$ with the steps already swept marked in red (top left) and the steps of the new path $\swg(\del)$ already built (top right).  The $4$-tuples record the levels visible from the west in $\del$ if the red (swept) steps are rendered invisible---the level to be swept next is colored red.  Sweeping either increases the level by $4$ if it corresponds to the level to the west of a swept horizontal step, or freezes the level (indicated by bold styling) if it corresponds to the level to the north of a swept vertical step.  See~\Cref{rem:sweep_levels}.}
\label{fig:ratpaths}
\end{figure}

It is not hard to argue that the sweep map sends an $(m,n)$-Dyck path to another $(m,n)$-Dyck path~\cite[Theorem 6.7]{thomas2015sweeping}, but invertibility is considerably more difficult.  The sweep map and its various generalizations were first shown to be bijective by Thomas and Williams in~\cite{thomas2015sweeping}.

\begin{remark}
There is a canonical injection \begin{align*}\Dyck_m^n &\xhookrightarrow{} \PF_m^n \\ \del &\mapsto \prk_\del, \end{align*} where $\prk_\del$ is the unique element of $\PF_m^n$ such that $n(\del) = n(\prk_\del)$. (That is to say, $\prk_\del$ is the parking tuple from which corners of $\del$ are removed in increasing order of label.)
We call $\prk_\del$ a \defn{Dyck $(m,n)$-filter tuple}.   By replacing $\del$ by $\prk_\del$, we may consider $\Dyck_m^n$ as a subset of $\PF_m^n$.

We can rephrase this injection using the interpretation of $\del$ as an $(m,n)$-Dyck path and elements of $\PF_m^n$ as $(m,n)$-parking paths.  Thinking of $\del$ as an $(m,n)$-Dyck path from $(0,0)$ to $(-n,m)$ (as in~\Cref{fig:dyck_paths}), we label each horizontal edge by the position of the level of its left endpoint.  This associates a canonical $(m,n)$-parking path to $\del$, which corresponds to a parking $(m,n)$-filter tuple $\prk_\del$ by~\Cref{rem:park_pic}.   Note that the lattice path used to compute $\swg(\del)$ is the same as the (unlabeled) lattice path associated to $\prk_\del$ in~\Cref{rem:park_pic}, whose column heights are counted by $A(\prk_\del)$.
\label{rem:embed}
\end{remark}

The injection of~\Cref{rem:embed} allows us to relate the zeta and sweep maps as follows.

\begin{proposition}
\label{prop:sweep_and_zeta}
For $\del$ an $(m,n)$-Dyck path, $B(\prk_\del)$ is an increasing word that records the column heights of $\swg(\del)$.
\end{proposition}

\begin{proof}
 We check that $B(\prk_\del)$  encodes $\swg(\del)$: by construction of $\prk_\del$, the number $p_i$ being removed when passing from $\prk_\del^{(i)}$ to $\prk_\del^{(i+1)}$ is the minimal level among those levels in $m(\prk_\del^{(i)})$ which are the levels of horizontal edges. 
 Meanwhile, the levels in $m(\prk_\del^{(i+1)})$ with value less than that minimal horizontal edge level keep track of the vertical steps in the construction of $\swg(\del)$.  The number of such vertical edge levels which are present as we pass from $\prk_\del^{(i)}$ to $\prk_\del^{(i+1)}$ tells us, on the one hand, the height of the $i$-th column in $\swg(\del)$ and on the other hand the number of letters in $m(\prk_\del^{(i)})$ strictly less than $p_i$ (which is what $B$ records).  This is illustrated in~\Cref{fig:ratpaths}.
\end{proof}

By~\Cref{thm:b_is_bij}, since $B: \PF_m^n \to \Park_m^n$ is a bijection, we obtain a new proof that the sweep map on $(m,n)$-Dyck paths is invertible.

\begin{theorem}
For $m,n$ coprime, the sweep map on $(m,n)$-Dyck paths is invertible.
\end{theorem}

\begin{remark}
\label{rem:sweep_levels}
In~\cite[Section 5.2]{armstrong2015sweep}, Armstrong, Loehr, and Warrington remark that the sweep map can be inverted if the levels of each of the steps on the path specified by $\swg(\del)$ can be determined.  The last two authors gave an algorithm to determine these levels in~\cite{thomas2015sweeping}.

This strategy of determining levels can be related to the fixed point of a parking word as follows.  \Cref{prop:sweep_and_zeta} shows that the fixed point of the (increasing) parking word $B(\prk_\del)$ encodes the levels of the vertical steps of $\swg(\del)$.  For example, the left path $\del$ in~\Cref{fig:ratpaths} corresponds to the $(m,n)$-filter tuple $\prk_\del$ specified on the left of~\Cref{fig:sweep}.   Then $m(\prk_\del^{(n)})= [7,13,14,16]$ records the levels that should be assigned (from top left to bottom right) to the vertical steps of $\swg(\del)$, as illustrated on the right of~\Cref{fig:sweep}.  The remaining levels---corresponding to the horizontal steps (again from top left to bottom right)---are determined from by the word $n(\prk_\del^{(0)}) = [0,4,6,8,9,10,12]$. 
\end{remark}

\section{The Affine Symmetric Group}
\label{sec:affine_symmetric_and_regions}

In~\Cref{sec:words_and_actions,sec:new_char}, we gave a new interpretion of $(m,n)$-parking words as transformations of $V_0^m$---that is, they were words 
acting with fixed points on points with $m$ coordinates.  In this section, we recall the interpretation of $(m,n)$-parking words as points in $\RR^n_{n(n+1)/2}$---that is, as certain points with $n$ coordinates.

The coincidence between the number of regions in the type $\widetilde{\Sym}_n$ Shi arrangement (\Cref{sec:sommers}) and the number of $(n{+}1,n)$-parking words has led to many purely combinatorial investigations~\cite{stanley1996hyperplane,stanley1998hyperplane,athanasiadis1999simple,armstrong2012shi,leven2014bijections}.   Although many different authors have found many different bijections between Shi regions and parking words, this direction of research culminates in work of Gorsky, Mazin, and Vazirani~\cite{gorsky2016affine}, who expand upon and generalize Armstrong's work in~\cite{armstrong2013hyperplane} from the Fuss to the rational level of generality.  In this section, we prove several of their conjectures.

We first review the basic combinatorics of $\widetilde{\Sym}_n$ in~\Cref{sec:affine_symmetric}.  We state the simple relationship between parking $(m,n)$-filter tuples and the affine symmetric group in~\Cref{thm:balanced_eq_sommers,thm:and_bij} and \Cref{thm:dominant_bij}.  This relationship allows us to define two maps from a generalization of Shi regions (alcoves in the Sommers region) to parking words, which are a restatement of~\Cref{def:parka,def:parkb}.

\subsection{The Affine Symmetric Group}
\label{sec:affine_symmetric}
The affine symmetric group $\widetilde{\Sym}_n$ is the group of bijections
$\aw: \ZZ \to \ZZ$ such that
\begin{align*} \aw(i+n)&=\aw(i)+n \text{ and} \\
\sum_{i=1}^n \aw(i) &= \binom{n+1}{2}.\end{align*}
We often represent elements of $\widetilde{\Sym}_n$ in (short) one-line notation \[\aw = \left[\aw(1),\aw(2),\ldots,\aw(n)\right].\]  A \defn{dominant permutation} is an affine permutation $\aw$ whose one-line notation increases, so that $\aw(1)<\aw(2)<\cdots<\aw(n)$.  An \defn{inversion} of $\aw$ is a pair $(i,j)$ with $1 \leq i \leq n$ and $i<j$ such that $\aw(i)>\aw(j)$.  We refer the reader to~\cite{lusztig1983some,gorsky2016affine} for more details.

The one-line notation of affine permutations bijectively corresponds to the alcoves in the affine $\widetilde{\mathfrak{S}}_{n}$ hyperplane arrangement, introduced in~\Cref{sec:hyperplane_and_functions}.\footnote{But note that we are now working with $\widetilde{\mathfrak{S}}_{n}$ not $\widetilde{\mathfrak{S}}_{m}$.}

\begin{theorem}[{\cite[Lemma 2.9]{gorsky2016affine}}]
Each alcove of $\RR^n_{n(n+1)/2} \setminus \mathcal{H}$ contains a unique point $(x_1,\ldots,x_n)$ that is the one-line notation of an element of $\widetilde{\Sym}_n$.  Conversely, each element of $\widetilde{\Sym}_n$ occurs as such a point.
\label{thm:alcove_points}
\end{theorem}

The alcove labeled by the identity permutation $[1,2,\ldots,n]$ is called the \defn{fundamental alcove} $\A_0$.   An inversion $(i,j)$ of $w\in\widetilde{\Sym}_n$ corresponds to the hyperplane $\mathcal{H}_{i,j'}^k$ that separates the alcove containing the one-line notation for $w$ from $\A_0$, where $j'=\begin{cases} j \mod n & \text{if } j \neq 0 \mod n \\ n & \text{otherwise} \end{cases}$ and $k=\frac{1}{n}(j-j')$.  The bijection of~\Cref{thm:alcove_points} between $\widetilde{\Sym}_n$ and the alcoves of $\RR^n_{n(n+1)/2} \setminus \mathcal{H}$ is illustrated for $n=3$ in~\Cref{fig:perms}.  On the other hand,~\Cref{fig:invperms} depicts the the labeling of an alcove by the \emph{inverse} of the corresponding permutation.

\begin{figure}[htb]
\includegraphics[width=.75\textwidth]{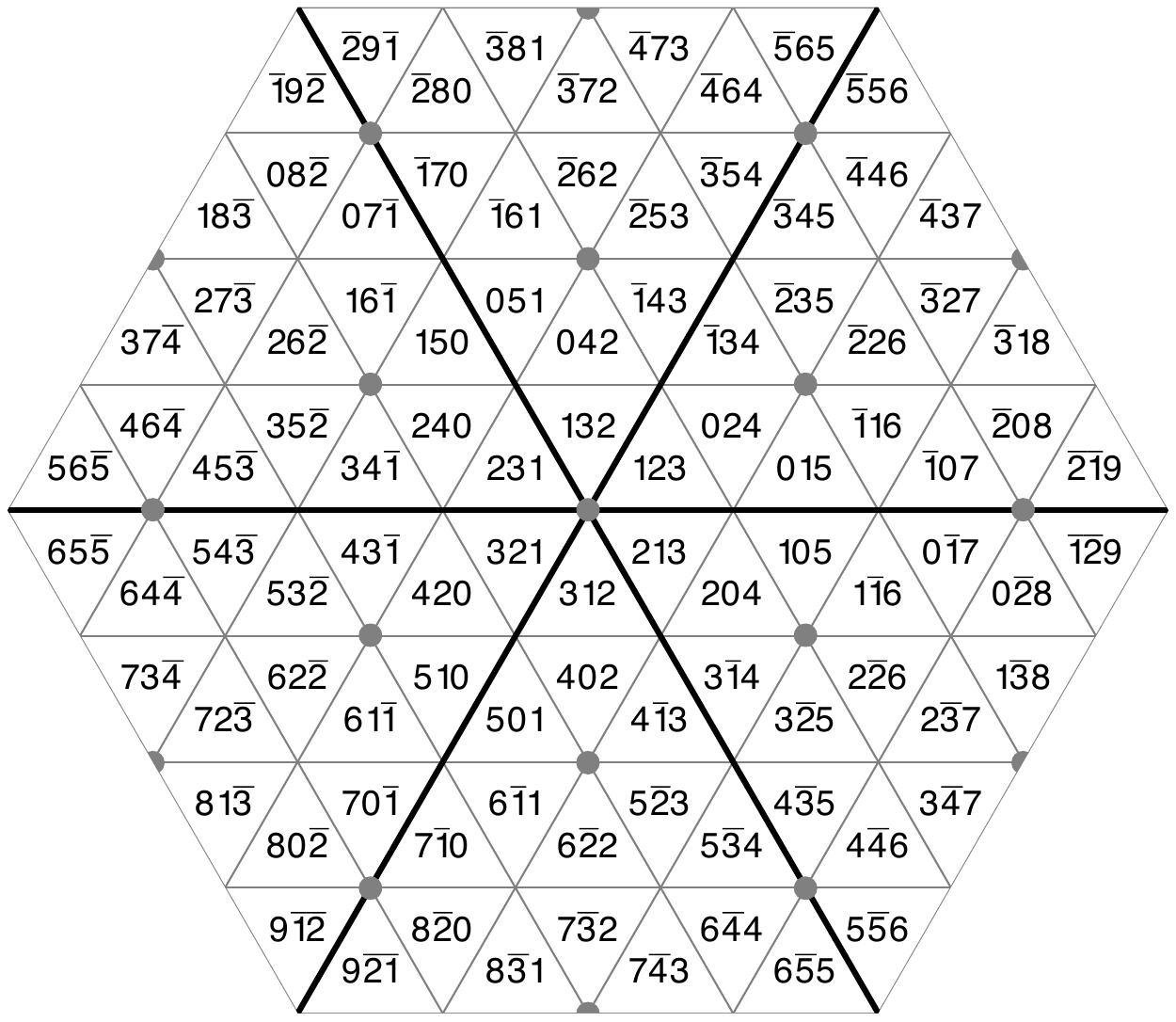}
\caption{The labeling of alcoves in $\RR^3_{6} \setminus \mathcal{H}$ by $\widetilde{\Sym}_3$.  The three solid black lines are the hyperplanes $\mathcal{H}_{1,2}^0, \mathcal{H}_{1,3}^0,$ and $\mathcal{H}_{2,3}^0$.}
\label{fig:perms}
\end{figure}

\begin{figure}[htb]
\includegraphics[width=.75\textwidth]{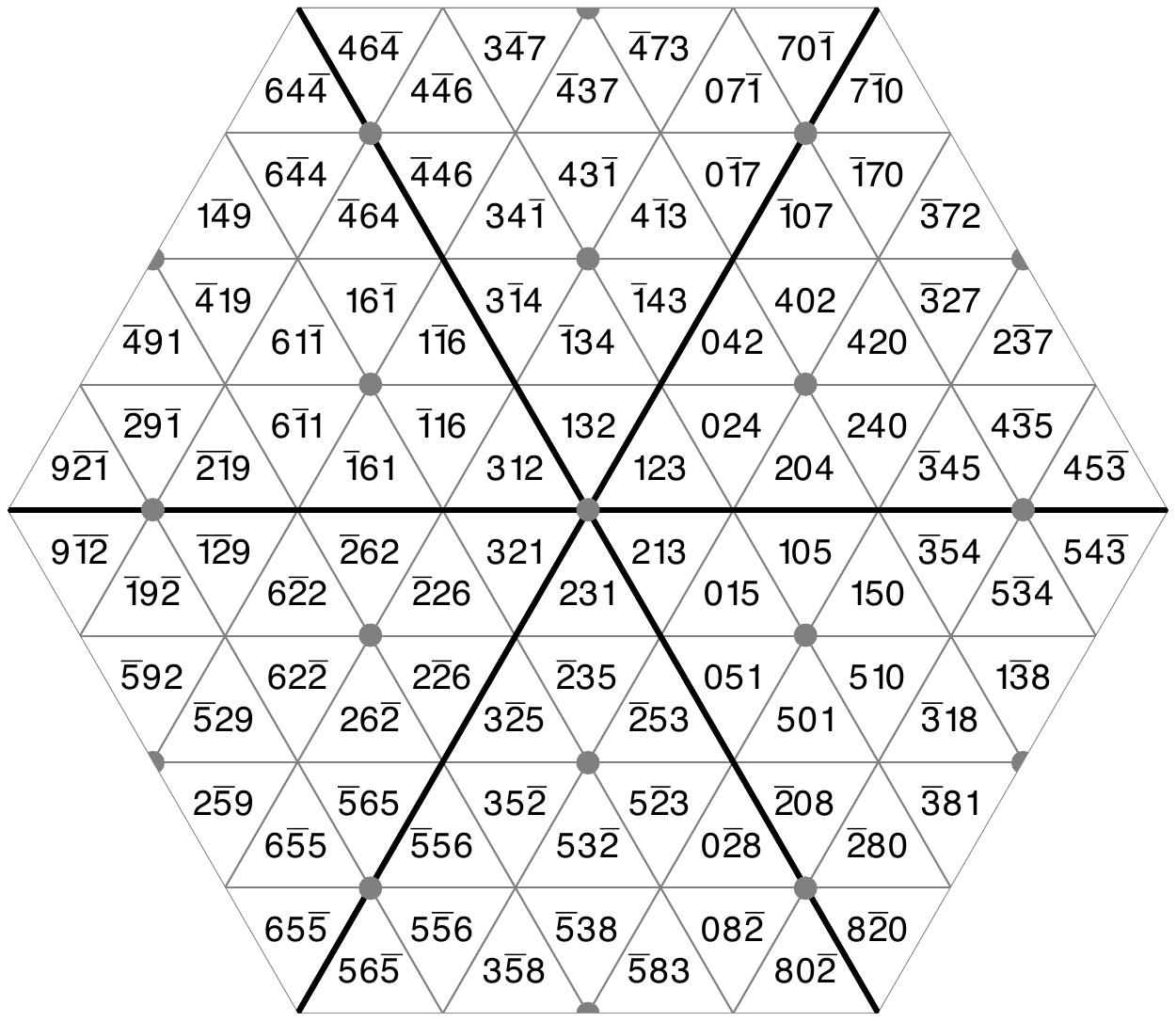}
\caption{The labeling of alcoves in $\RR^3_6 \setminus \mathcal{H}$ by inverse permutations.}
\label{fig:invperms}
\end{figure}

\subsection{The Sommers Region}
\label{sec:sommers}

\begin{definition}\label{def:sommers}
For $m$ coprime to $n$, the \defn{Sommers region}  $\So_m^n\subset\mathbb{R}^{n}_{n(n+1)/2}$ is the region bounded by the $n$ affine hyperplanes in $\widetilde{\Sym}_n$ of height $m$.
\end{definition}

The regions $\So_4^3$ and $\So_5^3$ are illustrated in~\Cref{fig:sommers}.  We have chosen to denote the Sommers region as $\So^n_m$ so that the exponent $n$ matches the exponent in the ambient space $\mathbb R^n_{n(n+1)/2}$---some references, such as~\cite{gorsky2016affine}, make the choice of opposite convention so that the subscript matches the subscript of $\widetilde{\Sym}_n$.    Note that when $m$ is not coprime to $n$, the hyperplanes of height $m$ do not bound a finite region.

By abuse of notation, using~\Cref{thm:alcove_points} we write $\aw \in \So_m^n$ if $\aw$ is an affine permutation labeling an alcove inside $\So_m^n$.  We can detect such affine permutations with the following simple proposition.

\begin{proposition}[{\cite[Definition 2.14]{gorsky2016affine}}]
An affine permutation $\aw^{-1} \in \widetilde{\Sym}_n$ labels an alcove in the region $\So_m^n$ iff $\aw(i)-\aw(j)\neq m$ for all $i<j$.
\label{prop:affine_perm_in_so}
\end{proposition}

\subsubsection{History of the Sommers Region}
\label{sec:history_sommers}
The Sommers region originated in Shi's study of Kazhdan-Lusztig cells of affine Weyl groups~\cite{shi2006kazhdan}, as we now outline.  The collection of affine hyperplanes \[\bigcup_{1 \leq i < j \leq n} \left(\mathcal{H}_{i,j}^0 \cup  \mathcal{H}_{j,i}^1 \right)\] is called the \defn{Shi arrangement}, and these hyperplanes cut out connected regions called \defn{Shi regions}.  Each Kazhdan-Lusztig cell is a union of Shi regions.  Following a suggestion of Carter, Shi gave an elegant geometric proof that there are $(n+1)^{n-1}$ Shi regions by showing that the inverses of the permutations labeling the minimal alcoves in the Shi regions coalesce into what has become known as the Sommers region $\So_{n+1}^n$~\cite{shi1987sign,sommers2003b}.\footnote{Eric Sommers was surprised to learn that the region has recently been named after him.}

 There is a Fuss analogue of the Shi arrangement, defined as the hyperplanes \[\bigcup_{\substack{1 \leq i < j \leq n \\ -k \leq s \leq k-1}} \mathcal{H}_{i,j}^s.\]
 This arrangement has $(kn+1)^{n-1}$ connected regions---again, the inverses of the minimal alcoves coalesce into the Sommers region $\So_{kn+1}^n$.
 
 The \defn{fundamental alcove} $\mathcal{A}_0$ in $\RR_{n(n+1)/2}^n$ is the simplex bounded by the affine simple hyperplanes.  It turns out that $\So_m^n$ is congruent to the $m$-fold dilation of the fundamental alcove $m\mathcal{A}_0$---this may be realized by multiplication by the element~\cite[Lemma 2.16]{gorsky2016affine},\cite[Theorem 4.2]{thiel2017strange}
\begin{equation}\aw_m^n := [\ell,\ell+m,\ldots,\ell+(n-1)m] \in \widetilde{\Sym}_n, \text{ where } \ell=\frac{1+m+n-mn}{2}.\label{eq:wmn}\end{equation} 
 
 Variations on subarrangements of affine Weyl hyperplane arrangements has led to interesting and surprisingly difficult combinatorics~\cite{stanley1996hyperplane,stanley1998hyperplane,athanasiadis1998free,postnikov2000deformations,armstrong2012shi,leven2014bijections,thomas2014cyclic}, but outside of $m=kn+1$ there are no hyperplane arrangements whose regions have minimal alcoves given by the inverses of the elements in $\So_m^n$~\cite[Example 9.2]{gorsky2016affine}.  Suggestive results exist for $m=kn-1$ using Zaslavsky's theorem enumerating bounded regions of a hyperplane arrangement (or Ehrhart duality)~\cite{fishel2010bijection,fishel2013counting}, and some work has been done when $m$ and $n$ are not coprime~\cite{gorsky2017rational}.

\begin{figure}[htb]
\includegraphics[width=.45\textwidth]{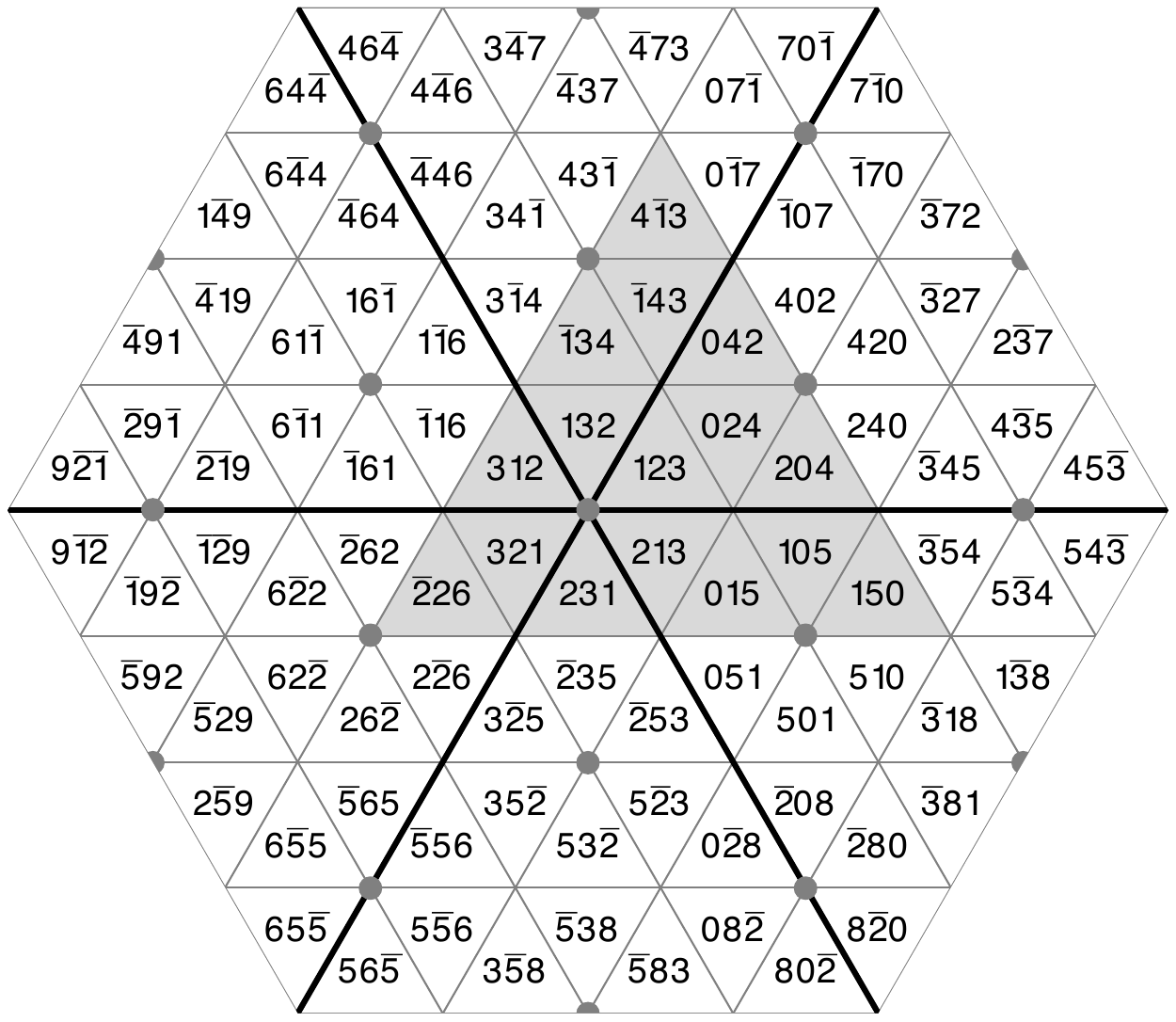}\hfill\includegraphics[width=.45\textwidth]{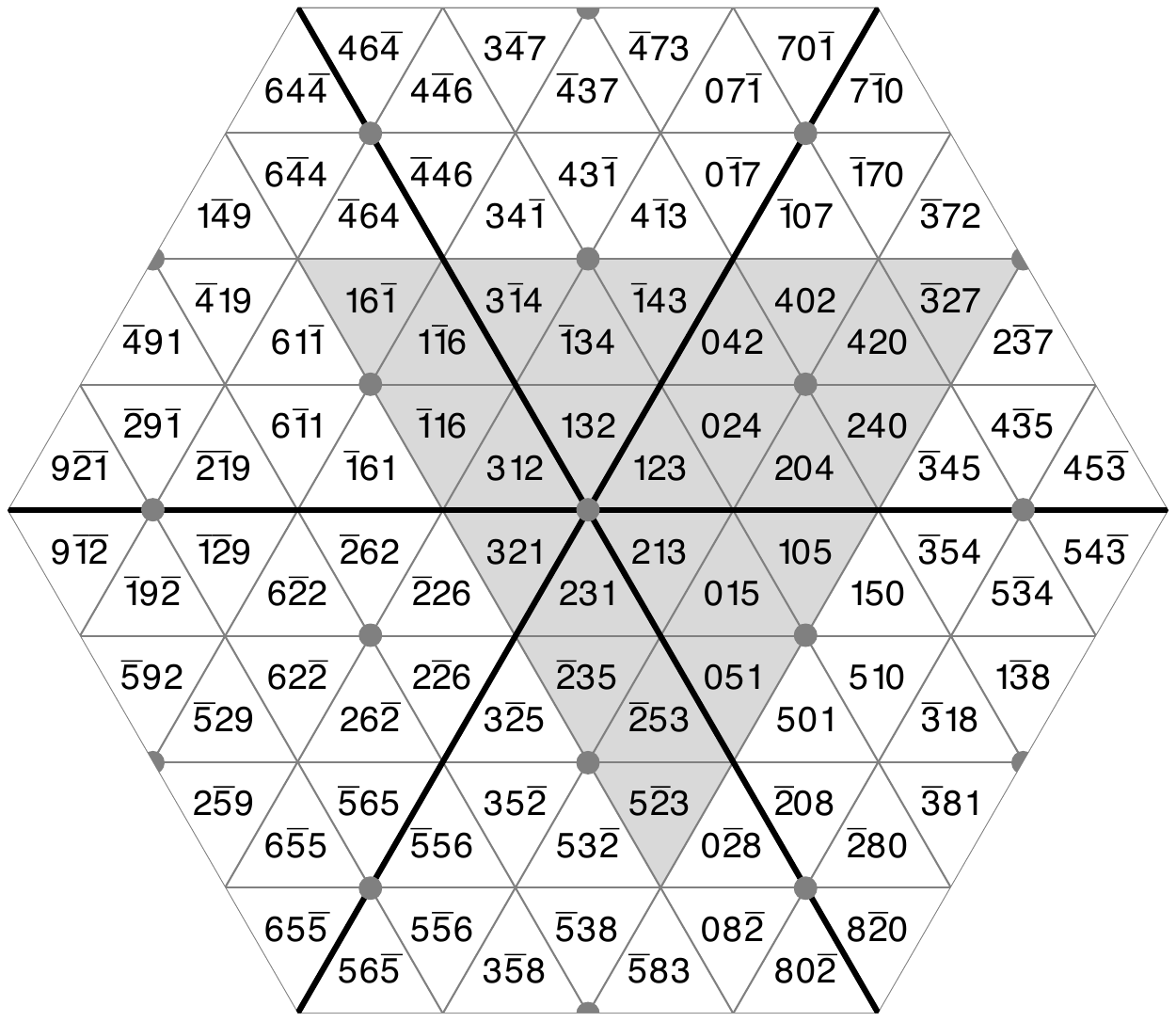}
\caption{The Sommers regions $\So_4^3$ and $\So_5^3$, with alcoves labeled by inverse permutations.}
\label{fig:sommers}
\end{figure}

\subsubsection{Filters and the Sommers Region}
\label{sec:filters_and_sommers}

To connect $(m,n)$-filters and affine permutations, we define the analogue of the directed graph $\mathfrak{F}_m^n$ in~\Cref{def:directed_graph}.

Fix $\aw\in \widetilde{\Sym}_n$ with $\aw^{-1} \in \So_m^n$.   An \defn{$m$-minimal element} of $\aw$ is an element of $\{\aw(i) : i \in \mathbb{N}\}$ that is minimal in its residue class modulo $m$.  We say that an $m$-minimal element of $\aw$ is \defn{removable} if it is in the short one-line notation of $\aw$---that is, if it is $\aw(i)$ for some $1 \leq i \leq n$. 

\begin{definition}
Define a directed graph $\mathfrak{P}_m^n$ with vertex set \[\left\{\aw \text{ dominant} : \aw^{-1} \in \So_m^n \right\}\] and a directed edge between $\aw$ and $\aw'$ iff the short one-line notation of $\aw'$ can be obtained from the short one-line notation of $\aw$ by adding $n$ to a removable $m$-minimal element of $\aw$, subtracting one from every element, and then resorting.
\label{def:directed_graph2}
\end{definition}

\begin{lemma}
Acting as described in~\Cref{def:directed_graph2} on a removable $m$-minimal element of a dominant $\aw$ with $\aw^{-1}\in \So_m^n$ produces another dominant element whose inverse is in $\So_m^n$.
\end{lemma}

\begin{proof}
Suppose that $\aw(i)$ is a removable $m$-minimal element, and let 
$\aw'$ be produced as above starting from that element.  Clearly
$\aw'$ is dominant. We now apply the condition of~\Cref{prop:affine_perm_in_so} to $\aw'$.  The only way a problem could arise would be if there were some $j>n$ with $\aw(j)=\aw(i)+n-m$.  But if
$j-n<i$, the fact that $\aw(j-n)$ is congruent modulo $m$ to $\aw(i)$ would violate the $m$-minimality of $i$, while $j-n>i$ would violate the condition of~\Cref{prop:affine_perm_in_so} for $\aw$.
\end{proof}

\begin{figure}[htb]
\centering
\begin{tikzpicture}[scale=1.5]
\node (A) [fill=white!20,rounded corners=2pt,inner sep=1pt] at (0,0) {$[1,2,3,4]$};
\node (D) [fill=white!20,rounded corners=2pt,inner sep=1pt] at (1.5,2.5) {$[0,1,3,6]$};
\node (E) [fill=white!20,rounded corners=2pt,inner sep=1pt] at (0,3.5) {$[-2,1,4,7]$};
\node (B) [fill=white!20,rounded corners=2pt,inner sep=1pt] at (-1.5,2.5) {$[-1,2,4,5]$};
\node (C) [fill=white!20,rounded corners=2pt,inner sep=1pt] at (0,1.5) {$[0,2,3,5]$};
\draw[->] (B) to [bend right]  node[midway, left] {$0$} (A);
\draw[->] (C) to [bend right] node[midway, left] {$0$} (A);
\draw[->] (A) to [bend right] node[midway, right] {$1$} (C);
\draw[->] (B) to node[midway, above] {$2$} (E);
\draw[->] (D) -- (B) node[midway, above] {$1$};
\draw[->] (C) -- (B) node[midway, below] {$1$};
\draw[->] (D) -- (C) node[midway, below] {$0$};
\draw[->] (A) to [bend right] node[midway, right] {$2$} (D);
\draw[->] (E) -- (D) node[midway, above] {$0$};
\draw [->] (A) edge[loop below]node{$0$} (A);
\end{tikzpicture}\hfill
\begin{tikzpicture}[scale=1.5]
\node (A) [fill=white!20,rounded corners=2pt,inner sep=1pt] at (0,0) {$[1,2,3]$};
\node (D) [fill=white!20,rounded corners=2pt,inner sep=1pt] at (1.5,2.5) {$[0,1,5]$};
\node (E) [fill=white!20,rounded corners=2pt,inner sep=1pt] at (0,3.5) {$[-2,2,6]$};
\node (B) [fill=white!20,rounded corners=2pt,inner sep=1pt] at (-1.5,2.5) {$[-1,3,4]$};
\node (C) [fill=white!20,rounded corners=2pt,inner sep=1pt] at (0,1.5) {$[0,2,4]$};
\draw[->] (B) to [bend right]  node[midway, left] {$0$} (A);
\draw[->] (C) to [bend right] node[midway, left] {$0$} (A);
\draw[->] (A) to [bend right] node[midway, right] {$1$} (C);
\draw[->] (B) to node[midway, above] {$2$} (E);
\draw[->] (D) -- (B) node[midway, above] {$1$};
\draw[->] (C) -- (B) node[midway, below] {$1$};
\draw[->] (D) -- (C) node[midway, below] {$0$};
\draw[->] (A) to [bend right] node[midway, right] {$2$} (D);
\draw[->] (E) -- (D) node[midway, above] {$0$};
\draw [->] (A) edge[loop below]node{$0$} (A);
\end{tikzpicture}
\caption{The five dominant permutations whose inverses lie in $\So_3^4$ and $\So_4^3$, arranged in the directed graphs $\mathfrak{P}_3^4 \cong \mathfrak{P}_4^3$.   The edge labels record the \emph{position} of the removable $m$-minimal element chosen.  Each parking word in $\Park_{4}^3$ occurs as a unique directed cycle of length $3$ in $\mathfrak{P}_3^4$, while each parking word in $\Park_{3}^4$ occurs as a unique directed cycle of length $4$ in $\mathfrak{P}_4^3$.  Compare with~\Cref{fig:balanced}.}
\label{fig:perms_balanced}
\end{figure}

\begin{figure}[htb]
\centering
\begin{tikzpicture}[scale=1.5]
\node (A) [fill=white!20,rounded corners=2pt,inner sep=1pt] at (0,0) {$[1,2,3,4,5]$};
\node (D) [fill=white!20,rounded corners=2pt,inner sep=1pt] at (1.5,2.5) {$[0,1,3,4,7]$};
\node (D2) [fill=white!20,rounded corners=2pt,inner sep=1pt] at (1.5,4) {$[-1,1,2,5,8]$};
\node (E) [fill=white!20,rounded corners=2pt,inner sep=1pt] at (0,5) {$[-3,0,3,6,9]$};
\node (B) [fill=white!20,rounded corners=2pt,inner sep=1pt] at (-1.5,2.5) {$[-1,2,3,5,6]$};
\node (B2) [fill=white!20,rounded corners=2pt,inner sep=1pt] at (-1.5,4) {$[-2,1,4,5,7]$};
\node (C) [fill=white!20,rounded corners=2pt,inner sep=1pt] at (0,1.5) {$[0,2,3,4,6]$};
\draw[->] (B) to [bend right]  node[midway, left] {$0$} (A);
\draw[->] (C) to [bend right] node[midway, left] {$0$} (A);
\draw[->] (A) to [bend right] node[midway, right] {$1$} (C);
\draw[->] (B2) to node[midway, above] {$3$} (E);
\draw[->] (B2) to node[midway, right] {$0$} (C);
\draw[->] (C) to node[midway, left] {$3$} (D2);
\draw[->] (B) to node[midway, left] {$2$} (B2);
\draw[->] (D) -- (B) node[midway, above] {$1$};
\draw[->] (C) -- (B) node[midway, below] {$1$};
\draw[->] (D) -- (C) node[midway, below] {$0$};
\draw[->] (A) to [bend right] node[midway, right] {$2$} (D);
\draw[->] (E) -- (D2) node[midway, above] {$0$};
\draw[->] (D2) -- (D) node[midway, right] {$0$};
\draw[->] (D2) -- (B2) node[midway, above] {$1$};
\draw [->] (A) edge[loop below]node{$0$} (A);
\end{tikzpicture}\hfill
\begin{tikzpicture}[scale=1.5]
\node (A) [fill=white!20,rounded corners=2pt,inner sep=1pt] at (0,0) {$[1,2,3]$};
\node (D) [fill=white!20,rounded corners=2pt,inner sep=1pt] at (1.5,2.5) {$[0,1,5]$};
\node (D2) [fill=white!20,rounded corners=2pt,inner sep=1pt] at (1.5,4) {$[-1,1,6]$};
\node (E) [fill=white!20,rounded corners=2pt,inner sep=1pt] at (0,5) {$[-3,2,7]$};
\node (B) [fill=white!20,rounded corners=2pt,inner sep=1pt] at (-1.5,2.5) {$[-1,3,4]$};
\node (B2) [fill=white!20,rounded corners=2pt,inner sep=1pt] at (-1.5,4) {$[-2,3,5]$};
\node (C) [fill=white!20,rounded corners=2pt,inner sep=1pt] at (0,1.5) {$[0,2,4]$};
\draw[->] (B) to [bend right]  node[midway, left] {$0$} (A);
\draw[->] (C) to [bend right] node[midway, left] {$0$} (A);
\draw[->] (A) to [bend right] node[midway, right] {$1$} (C);
\draw[->] (B2) to node[midway, above] {$2$} (E);
\draw[->] (B2) to node[midway, right] {$0$} (C);
\draw[->] (C) to node[midway, left] {$2$} (D2);
\draw[->] (B) to node[midway, left] {$1$} (B2);
\draw[->] (D) -- (B) node[midway, above] {$1$};
\draw[->] (C) -- (B) node[midway, below] {$1$};
\draw[->] (D) -- (C) node[midway, below] {$0$};
\draw[->] (A) to [bend right] node[midway, right] {$2$} (D);
\draw[->] (E) -- (D2) node[midway, above] {$0$};
\draw[->] (D2) -- (D) node[midway, right] {$0$};
\draw[->] (D2) -- (B2) node[midway, above] {$1$};
\draw [->] (A) edge[loop below]node{$0$} (A);
\end{tikzpicture}
\caption{The seven dominant permutations whose inverses lie in $\So_3^5$ and $\So_5^3$, arranged in the directed graphs $\mathfrak{P}_3^5$ (left) and  $\mathfrak{P}_5^3$ (right).  The edge labels record the \emph{position} of the minimal element chosen.   Note that although the graphs are isomorphic as unlabeled direct graphs, the edge labels differ.  Each parking word in $\Park_{5}^3$ occurs as a unique directed cycle of length $3$ in $\mathfrak{P}_3^5$, while each parking word in $\Park_{3}^5$ occurs as a unique directed cycle of length $5$ in $\mathfrak{P}_5^3$.  Compare with~\Cref{fig:balanced2}.}
\label{fig:perms_balanced2}
\end{figure}

We now relate $(m,n)$-filters and the Sommers region, using the balanced representatives of $(m,n)$-filters.  We first use $(m,n)$-filters to understand \emph{dominant} affine permutations whose inverses lie in the Sommers region.

\begin{theorem}
   A dominant affine permutation $\aw \in \widetilde{\Sym}_n$ satisfies $\aw^{-1} \in \So_m^n$ if and only if \begin{align*}[\aw(1),\aw(2),\ldots,\aw(n)]&=n(\bal_\aw)\end{align*}
    for some balanced $(m,n)$-filter $\bal_\aw \in \Bal_m^n$.
    \label{thm:balanced_eq_sommers}
\end{theorem}

\begin{proof}
Note that the one-line notation of the element $\aw_m^n$ defined in \Cref{eq:wmn} is $n(\bal_m^n)$, where $\bal_m^n$ is the balanced $(m,n)$-filter generated by the points with level $\ell$ (see~\Cref{prop:balanced1}).  If we have $\aw = n(\bal)$ for some balanced $(m,n)$-filter, then the corresponding notion of minimal elements coincide, and acting on a minimal element of $\aw$ mirrors removing the corresponding minimal element of $\bal$.  The result now follows from~\Cref{def:directed_graph,def:directed_graph2}.
\end{proof}

Of course,~\Cref{thm:balanced_eq_sommers} applies equally well with the roles of $m$ and $n$ switched, and so we obtain an $(m{\leftrightarrow}n)$-bijection and a version of~\Cref{thm:dyck_paths} for dominant affine permutations whose inverses lie in the Sommers region.

\begin{proposition}
For $m$ and $n$ coprime, there is a bijection \begin{align*}\Big\{\aw \text{ dominant}: \aw^{-1} \in \So_m^n\Big\} &\leftrightarrow \Big\{\aw \text{ dominant}: \aw^{-1} \in \So_n^m\Big\}.\end{align*}  Furthermore, both sets have cardinality \[\frac{1}{n+m}\binom{n+m}{n}.\]
\label{thm:dominant_bij}
\end{proposition}

\begin{proof}
The enumeration follows from~\Cref{thm:balanced_eq_sommers}, and the bijection is induced by the map $n(\bal) \leftrightarrow m(\bal)$.
\end{proof}

\begin{example}
For example, looking at the balanced $(3,5)$-filter on the righthand side of~\Cref{fig:parking_function}, and disregarding the labels on the horizontal steps, the sorted list of the left-most level in each row gives $m(\bal)=[-1,3,4]$, while the sorted list of the bottom level in each column  gives $n(\bal)=[-1,2,3,5,6]$.
\end{example}

\begin{remark}
\Cref{thm:dominant_bij} is well-known in the language of simultaneous $(m,n)$-cores using the bijection between $n$-cores (respectively $m$-cores) and the coroot lattices of $\widetilde{\Sym}_n$ (respectively $\widetilde{\Sym}_m$).  This bijection of~\Cref{thm:dominant_bij} takes an element in $\widetilde{\Sym}_n$ associated to a particular simultaneous $(m,n)$-core and produces the corresponding element in $\widetilde{\Sym}_m$ associated to the same $(m,n)$-core.  We refer the reader to~\cite{anderson,misramiwa} and~\cite[Section 4]{armstrong2014results} for more details on cores and simultanous cores.
\end{remark}

\begin{remark}\label{rem:directly}
We can compute the bijection of~\Cref{thm:dominant_bij} directly on the one-line notation of an affine permutation $\aw$ by recording the $m$-minimal elements of $\aw$.  The sequence $\aw(1), \aw(2),\dots$ is obtained by recording the lowest entry of each column of $\bal_\aw$, in order, then the second-lowest entry of each column, and continuing in this way.  The first time an entry in a given congruence class is recorded is when we come to the leftmost entry of the corresponding row (i.e., an element of $m(\bal_\aw)$). Thus, the $3$-minimal elements of $\aw=[-1,2,3,5,6]$ are $[-1,3,4]$: \[\begin{array}{c|ccccc|cc} i & 1 & 2 & 3 & 4 & 5 & 6 & \ldots \\ \aw(i) & {\bf -1} & 2 & {\bf 3} & 5 & 6 & {\bf 4} & \ldots \\ \aw(i) \text{ mod } 3 & {\bf 2} & 2 & {\bf 0} & 2 & 0 &{\bf 1} & \ldots \end{array}.\]
Similarly, the $5$-minimal elements of $\aw=[-1,3,4]$ are $[-1,2,3,5,6]$:
\[\begin{array}{c|ccc|ccc|cc} i & 1 & 2 & 3 & 4 & 5 & 6 & 7  & \ldots \\ \aw(i) & {\bf -1} &{\bf 3} & 4 & {\bf 2}& {\bf 6} & 7 & {\bf 5} & \ldots \\ \aw(i) \text{ mod } 5 & {\bf 4} & {\bf 3} & 4 & {\bf 2} & {\bf 1} & 2 &{\bf 0}  & \ldots \end{array}.\]
\end{remark}

In fact,~\Cref{thm:balanced_eq_sommers} can be extended to the whole Sommers region if we pass from balanced $(m,n)$-filters to balanced $(m,n)$-filter tuples.

\begin{theorem}
   An affine permutation $\aw \in \widetilde{\Sym}_n$ satisfies $\aw^{-1} \in \So_m^n$ if and only if \begin{align*}[\aw(1),\aw(2),\ldots,\aw(n)]&=n(\prk_\aw)\end{align*} for some balanced $(m,n)$-filter tuple $\prk_\aw \in \BF_m^n$.
   \label{thm:and_bij}
\end{theorem}

\begin{proof}
Choose $\prk\in\BF_m^n$.  Now $n(\prk)$ is the short one-line notation of an affine permutation $w$ since $n(\prk)$ is a permutation of 
$n(\prk^{(0)})$ and $\prk^{(0)}$ is balanced.  We can think of the sequence $w(1),w(2),\dots$ as being obtained by recording the levels removed from $n(\prk^{(0)})$ by repeatedly removing boxes in the order specified by $\prk$.  (In this way, $w(1)$ through $w(n)$ are the levels removed on the first pass, $w(n+1)$,\dots, $w(2n)$ are the levels removed on the second pass, and so on.)  Since levels that differ by $m$ lie in the same row, the smaller is necessarily removed before the larger, guaranteeing that the condition of~\Cref{prop:affine_perm_in_so} is satisfied, so $w^{-1}\in \So_m^n$.

Now $\So_m^n$ is an $m$-fold dilation of the fundamental alcove in $\mathbb{R}^{n-1}$, and so contains $m^{n-1}$ affine permutations.  Since $\prk\mapsto n(\prk)$ is a bijection and $|\BF_m^n|=m^{n-1}$ by~\Cref{thm:number_parking_functions}, we conclude the result.
\end{proof}

\begin{remark}
\label{rem:cycles_affine}
Since $\mathfrak{P}_n^m \cong \mathfrak{F}_n^m$ as unlabeled directed graphs, by~\Cref{def:parking} we can interpret affine elements $\aw$ with $\aw^{-1} \in \So_m^n$ as cycles of $n$ vertices in the directed graph $\mathfrak{P}_n^m$ (we recall that vertices of $\mathfrak{P}_n^m$ are short one-line notation of permutations in $\So_n^m$), with a choice of initial vertex.  The short one-line notation of $\aw$ is given by reading the $m$-minimal element chosen for the edge (undoing the rebalancing that occurs at each step).

For example, reproducing~\Cref{rem:park_cycles} with $m=3$ and $n=5$ (see also~\Cref{fig:parking_functions}), the 5-cycle in $\mathfrak{P}_5^3$ with removable $3$-minimal elements in bold
\begin{alignat*}{6}
& [-1,{\bf 3},4] &&\to [{\bf -2},3,5] &&\to [{\bf 0},2,4] &&\to [1,{\bf 2},3] &&\to [0,{\bf 2},4] &&\to [-1,3,4] \text{ (rebalanced)}\\
&+0 &&\phantom{\to} +1 &&\phantom{\to} +2 &&\phantom{\to} +3 &&\phantom{\to} +4 &&\phantom{\to} +5\\
&[-1,{\bf 3},4] &&\to [{\bf -1},4,6] &&\to [{\bf 2},4,6] &&\to [4,{\bf 5},6] &&\to [4,{\bf 6},8] &&\to [4,8,9] \text{ (not rebalanced)}
\end{alignat*}
produces the short one-line notation of the affine Weyl group element \[\aw = [3,-1,2,5,6] \in \ASym_5.\]
\end{remark}

\subsection{Parking Words from the Sommers Region}
\label{sec:sommers_parking_words}

Using~\Cref{thm:and_bij}, we can easily restate the maps $A$ and $B$ from~\Cref{sec:mapa,sec:mapb}---originally defined on parking $(m,n)$-tuple filters---in the language of affine permutations.  These maps originally appeared in this form in~\cite{gorsky2016affine}.

\begin{remark}
There are \emph{many} statistics one can define on Dyck paths and parking functions (in their various combinatorial manifestations).  In~\cite{armstrong2013hyperplane} for $(m,n)=(n{+}1,n)$, Armstrong introduced statistics on the affine symmetric group that corresponded to what Haglund and Loehr called $\area'$ and $\bounce$ in~\cite{haglund2005conjectured}.  Armstrong suggested that his statistics would recover work in the $(kn+1,n)$ case, previously considered by Loehr and Remmel in~\cite{loehr2004conjectured}.  By using the relationship between Shi arrangements and Sommers regions, Gorsky, Mazin, and Vazirani generalized Armstrong's constructions to general coprime $(m,n)$---and called the statistics $\dinv$ and $\area$ (see~\Cref{sec:zeta_history}).  Finally, we note that the paper~\cite{armstrong2016rational} also defines statistics for general coprime $(m,n)$, but doesn't define a zeta map on $(m,n)$-parking paths or words.
\end{remark}

\subsubsection{The Map $\Astar$: the Anderson Labeling}
\label{sec:anderson2}
Translating~\Cref{def:parka} using the bijection of~\Cref{thm:and_bij} gives the following definition (compare with~\cite[Section 3.1]{gorsky2016affine}).

\begin{definition}
Let $\aw \in \widetilde{\Sym}_n$ be given with $\aw^{-1} \in \So_m^n$. Then $\Astar(\aw)$ is defined by
\[\aw \xmapsto{\Astar} \left[a \cdot \aw'(1), a \cdot \aw'(2),\ldots, a \cdot \aw'(n) \right] \mod m,\] where $\aw'(i)=\aw(i)-\min\{\aw(1),\aw(2),\ldots,\aw(n)\}$ and $a n = -1 \mod m$.
\end{definition}

\begin{example}
As in~\Cref{ex:mapa}, for $\aw = [3,-1,2,5,6]$ with $\aw^{-1}=[0,3,1,7,4] \in \So_3^5$ and $\aw' = [3+1,-1+1,2+1,5+1,6+1] =[4,0,3,6,7]$, since $1 \cdot 5 = -1 \mod 3$ we compute $\Astar(\aw)$ as
\[\aw = [3,-1,2,5,6] \xmapsto{\Astar} [1\cdot 4,1\cdot 0,1\cdot 3,1\cdot 6,1\cdot 7] \mod 3 = [1,0,0,0,1].\]  Using $3 \cdot 3 = -1 \mod 5$, we also find $\Astar(\aw)$ for $\aw=[-1,3,4]$ with $\aw^{-1}=[0,4,2]\in \So_5^3$ and $\aw'=[0,4,5]$:
\[\aw = [-1,3,4] \xmapsto{\Astar} [3\cdot 0,3\cdot 4,3\cdot 5] \mod 5 = [0,2,0].\]
\end{example}

If the short one-line notations of $\aw_1$ and $\aw_2$ are permutations of each other, then so are $\Astar(\aw_1)$ and $\Astar(\aw_2)$, so that elements in the same coset of $\ASym_n/\Sym_n$ are assigned to the same $(m,n)$-parking word by $\Astar$, up to a permutation.  It follows from~\Cref{sec:mapa} and~\Cref{thm:and_bij} that $\Astar$ is a bijection; this is illustrated for $(m,n)=(4,3)$ and $(5,3)$ in~\Cref{fig:anderson}.

\begin{theorem}
For $m$ and $n$ relatively prime, the map
\begin{align*}A^*: \So_m^n &\to \Park_m^n\\ \aw &\mapsto A(\aw^{-1})\end{align*} is a bijection.
\end{theorem}

There is a more geometric way to recover the parking word $\Astar(\aw)$, which we quickly sketch.  There is a natural bijection between dominant affine permutations in $\widetilde{\Sym}_n$ and the coroot lattice $\check{Q}:=\left\{\mathbf{x} \in \mathbb{Z}^n : \sum_{i=1}^n x_i = 0\right\}$: \[\aw \in \widetilde{\Sym}_n \mapsto \aw^{-1}(\mathbf{0}),\] where $\mathbf{0}=(0,0,\ldots,0) \in \check{Q}$.  This extends to a bijection between affine permutations and $\Sym_n \ltimes \check{Q}$.  The restriction of this bijection to the permutations whose inverses lie in the Sommers region $\So_m^n$ gives a set of representatives for $\check{Q}/m\check{Q}$, which are in bijection with $(m,n)$-parking words using natural coordinates and the cycle lemma.  We refer the reader to~\cite{haiman1994conjectures,gorsky2016affine,thiel2016anderson} for more details relating to this construction.

\begin{figure}[htb]
\includegraphics[width=.45\textwidth]{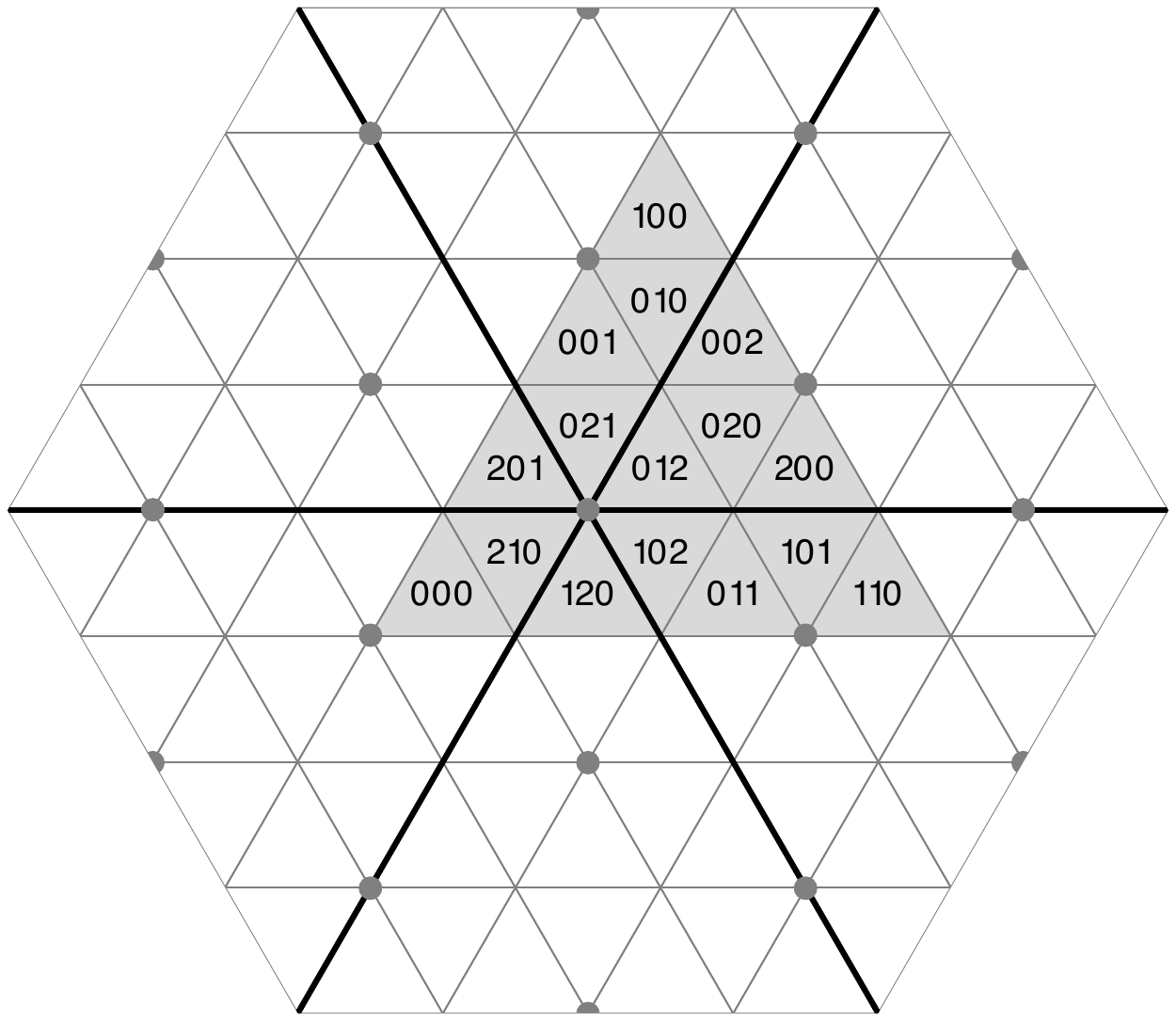}\includegraphics[width=.45\textwidth]{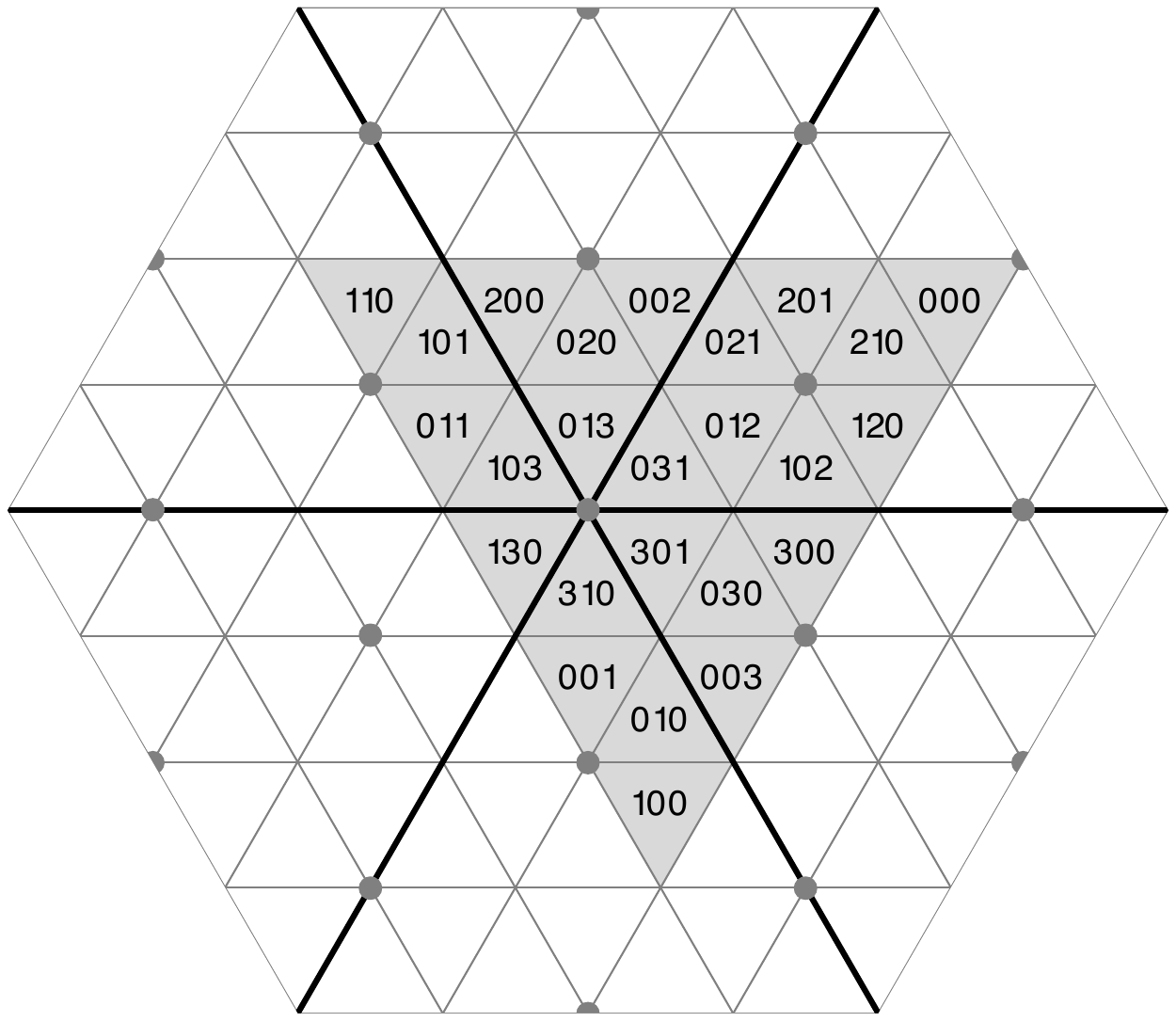}
\caption{The Sommers regions $\So_4^3$ and $\So_5^3$, with alcoves labeled by parking words under the Anderson bijection $\Astar$.}
\label{fig:anderson}
\end{figure}

\subsubsection{The Map $\Bstar$: the Pak-Stanley Labeling}
\label{sec:stanley_pak}
It is natural to ask for a \emph{bijective} proof for the number of Shi regions---for example, via a bijection betwen Shi regions and $(n{+}1,n)$-parking words.  Pak and Stanley found such a labeling of the Shi regions~\cite[Theorem 5.1]{stanley1996hyperplane}, which Stanley later extended to the Fuss level of generality~\cite{stanley1998hyperplane}.  Using the correspondence between the minimal alcoves of the Shi arrangement and the Sommers region, the Pak-Stanley labeling was finally extended to the rational level in~\cite{gorsky2016affine} as an affine analogue of the \defn{code} of a permutation.

\begin{definition}\label{def:sp}
For $\aw \in \widetilde{\Sym}_n$ with $\aw^{-1}\in \So_m^n$, $\Bstar(\aw)$ is defined by
\begin{align*}
	\aw &\xmapsto{B} \pp_1 \ldots \pp_n,
 \end{align*}
where for $1 \leq i \leq n$, \[\pp_i = \left|\left\{ j : j>i \text{ and } 0<\aw(i) - \aw(j)<m \right\}\right|.\]
\end{definition}

Using the correspondence between inversions and hyperplanes, $\pp_i$ counts the number of hyperplanes of the form $\mathcal{H}_{i,j}^k$ of height less than $m$ separating the alcove corresponding to $\aw^{-1}$ from the fundamental alcove.  The Pak-Stanley labeling of the Sommers region is illustrated in the cases $(m,n)=(4,3)$ and $(5,3)$ in~\Cref{fig:stanley_pak}.

\begin{figure}[htb]
\includegraphics[width=.45\textwidth]{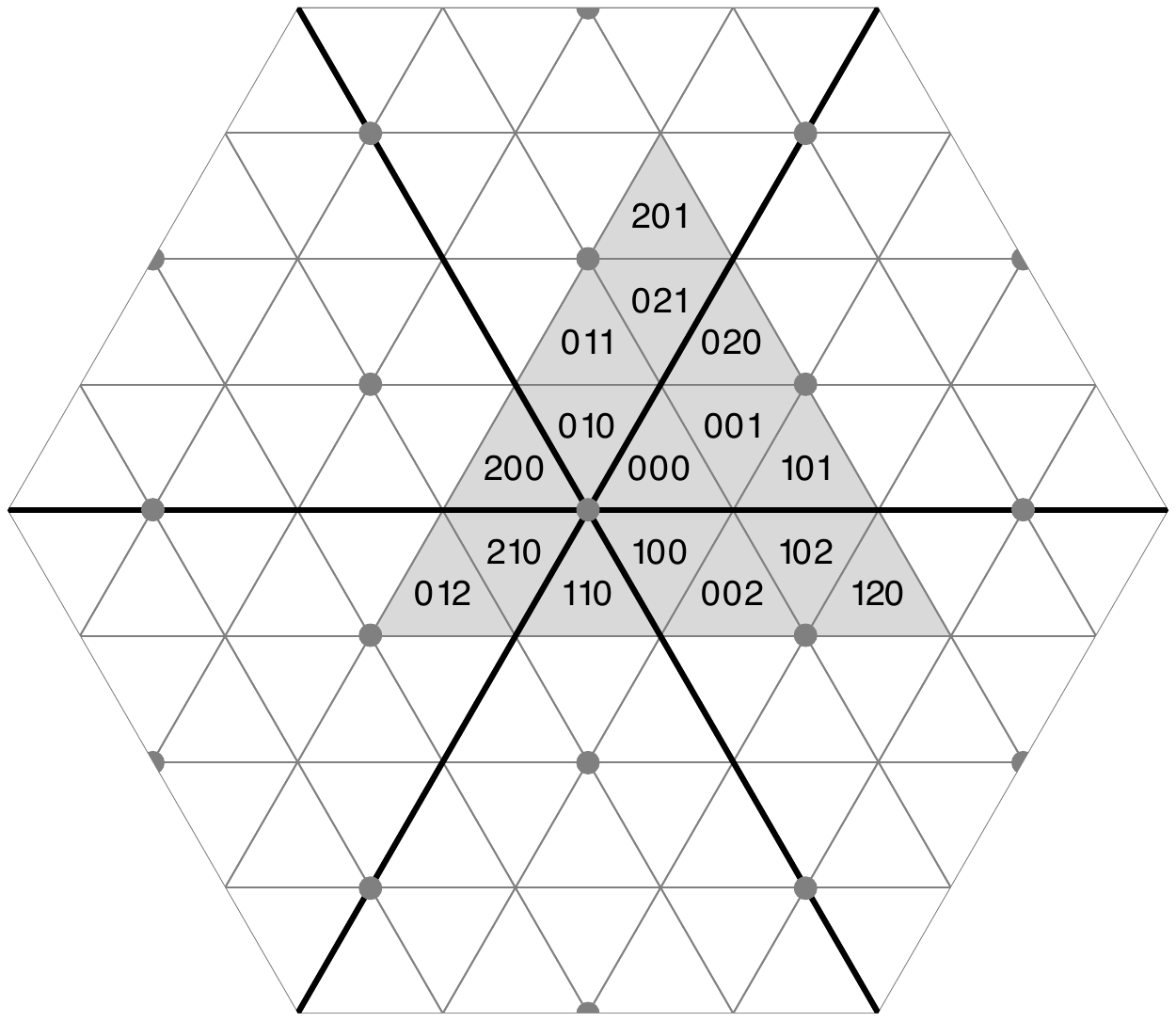}\includegraphics[width=.45\textwidth]{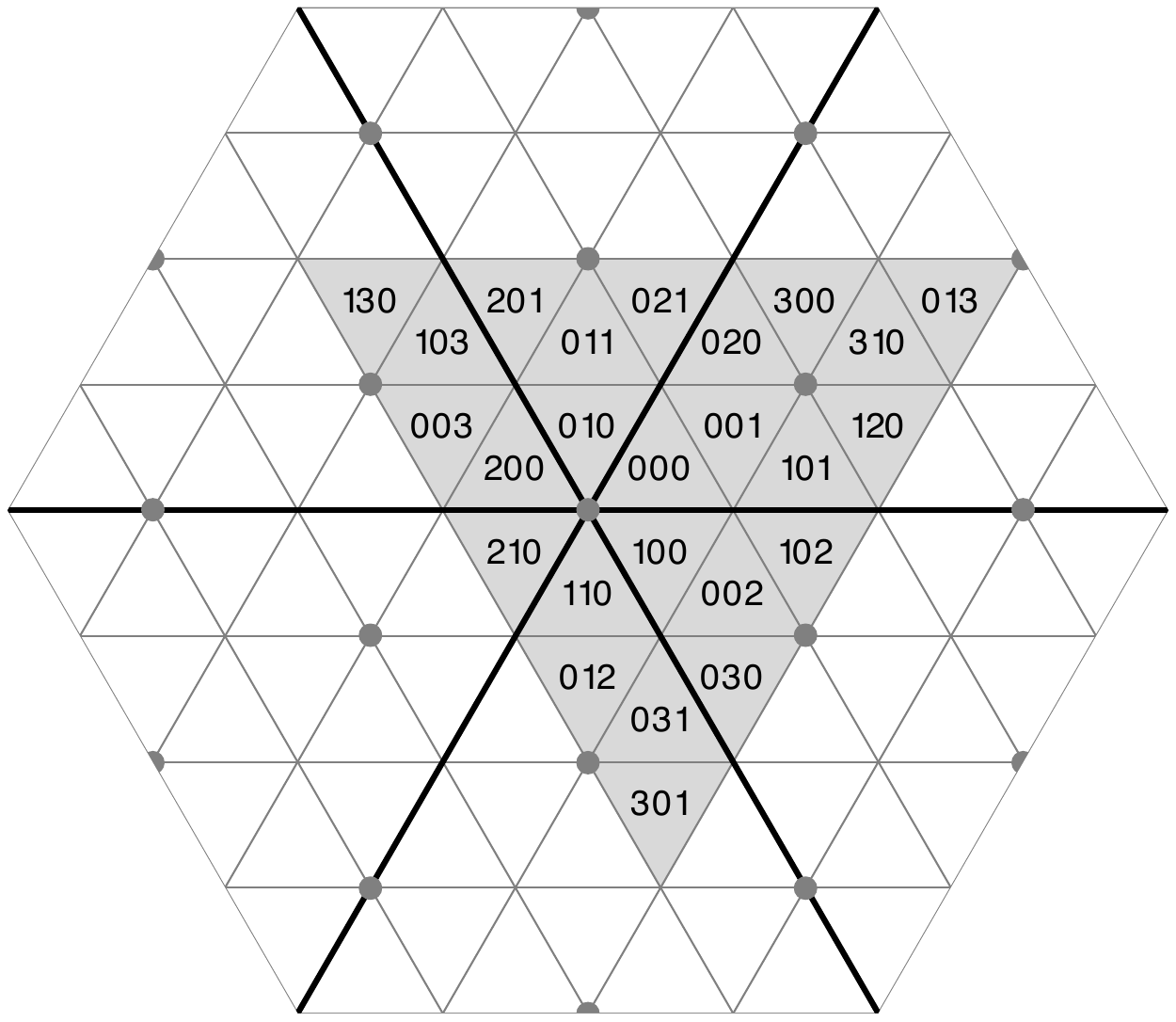}
\caption{The Sommers regions $\So_4^3$ and $\So_5^3$, with alcoves labeled by parking functions under the Pak-Stanley bijection $\Bstar$.}
\label{fig:stanley_pak}
\end{figure}

We will now show that $\Bstar(\aw)$ in~\Cref{def:sp} is equivalent to $B(\prk_\aw)$ in~\Cref{def:parkb} under the bijection in~\Cref{thm:and_bij}.

\begin{theorem}\label{thm:b_eq}
For any $\aw \in \widetilde{\Sym}_n$ with $\aw^{-1} \in \So_m^n$, we have that $\Bstar(\aw) = B(\prk_\aw),$ where $\prk_\aw$ is the $(m,n)$-filter tuple with $[w(1),w(2),\ldots,w(n)]=n(\prk_\aw)$.
\end{theorem}

\begin{proof}
\Cref{rem:cycles_affine} gives a bijection between $n$-cycles in $\mathfrak{P}_n^m$ and affine permutations $\aw \in \ASym_n$ with $\aw^{-1} \in \So^n_m$.  Since $\mathfrak{P}_n^m \cong \mathfrak{F}_n^m$, we can see $B(\prk_\aw)$ directly on the $n$-cycle. 
Fix $1 \leq i \leq m$.  At most one element from each residue class modulo $n$ in the one-line notation of $\aw$ can contribute to 
$\pp_i$.  The number of residue classes which contribute (which equals $\pp_i$) is also the position of the number removed when calculating $B(\prk_\aw)$.
\end{proof}

\begin{remark}
Continuing~\Cref{rem:cycles_affine}, we interpret parking words $\pp \in \Park_m^n$ (of length $n$) as cycles with $n$ vertices in the directed graph $\mathfrak{P}_n^m$.  The word is obtained by recording the position of the element chosen for the edge.  For example, for $(m,n)=(3,5)$ and $\aw = [3,-1,2,5,6]$ with $\aw^{-1}=[0,3,1,7,4] \in \So_3^5$, recording the position of the element removed computes the parking word $B(\prk_\aw)$ from the $5$-cycle encoding the corresponding parking $(3,5)$-filter tuple $\prk_\aw$:

\begin{alignat*}{7}
 & [-1,{\bf 3},4] &&\to [{\bf -2},3,5] &&\to [{\bf 0},2,4] &&\to [1,{\bf 2},3] &&\to [0,{\bf 2},4] &&\to [-1,3,4]\\
&+0 &&\phantom{\to} +1 &&\phantom{\to} +2 &&\phantom{\to} +3 &&\phantom{\to} +4 &&\phantom{\to} +5\\
 & [-1,{\bf 3},4] &&\to [{\bf -1},4,6] &&\to [{\bf 2},4,6] &&\to [4,{\bf 5},6] &&\to [4,{\bf 6},8] &&\to [4,8,9] \\
B(\prk_\aw) :&  \phantom{+}1 && \phantom{\to+}0 && \phantom{\to+}0 && \phantom{\to+}1 && \phantom{\to+}1 &&
\end{alignat*}
On the other hand, we can compute $B(\aw)=\pp_1\ldots\pp_5$ by extending the short one-line notation of $\aw$.  \Cref{thm:b_eq} tells us that the results of these two calculations agree.

\[\xymatrix@C=1em@R=1.5em{ i \ar@{-}[]+<2em,+1em>;[dd]+<2em,-1em> &  1 & 2 & 3 & 4 & 5 \ar@{--}[]+<1em,+1em>;[dd]+<1em,-1em> & 6 & 7 & 8 & 9 & 10 \\  \aw(i) & {\bf 3}\ar@/_1pc/[rr] & {\bf -1} & 2 & 5 \ar@/_2pc/[rrr] & 6 \ar@/_1pc/[rr] & 8 & {\bf 4} & 7 & 10 & 11 \\ \pp_i & 1 & 0 & 0 & 1 & 1 & & & &}.\]

The letters in the one-line notation of $w$ that occur in $\prk_\aw^{(0)}$ are written in bold, and we have marked the inversions that count towards $\Bstar(\aw)$ using arrows.  Note that the inversion $(i,j)=(1,2)$ doesn't count towards $\Bstar(\aw)$ because $\aw(1)-\aw(2)=3-(-1)\geq 3$.
\label{ex:b2}
\end{remark}

Now \Cref{thm:b_eq,thm:b_is_bij} imply that $\Bstar$ is a bijection from affine permutations whose inverse lies in $\So_m^n$ to $(m,n)$-parking words.  This resolves~\cite[Conjecture 1.4]{gorsky2016affine}.

\begin{theorem}[{\cite[Conjecture 1.4]{gorsky2016affine}}]
    For $m$ and $n$ relatively prime, the map
\begin{align*}B^*: \So_m^n &\to \Park_m^n\\ \aw &\mapsto B(\aw^{-1})\end{align*} is a bijection.
\label{thm:rational_code}
\end{theorem}

\begin{remark}
 In~\cite[Section 7.1]{gorsky2016affine}, Gorsky, Mazin, and Vazirani provide a conjectural algorithm to invert $\Bstar$.  Their Conjecture 7.9 (which essentially says that their algorithm succeeds) follows now from our~\Cref{thm:b_is_bij} and the convergence proved in~\Cref{lem:rp_fixed} and ~\Cref{lem:rp_unique}.
 \end{remark}

\section*{Acknowledgements}

We thank Drew Armstrong for the clarity of his exposition, Marko Thiel and Robin Sulzgruber for their input on the history of zeta maps, and Adriano Garsia, Eugene Gorsky, Nick Loehr, Mikhail Mazin, Igor Pak, Monica Vazirani, and Greg Warrington for many inspiring conversations.  We thank an anonymous referee for their many helpful and detailed comments which improved the paper.

H.T. was partially supported by the Canada Research Chair grant CRC-2014-00042 and NSERC Discovery Grant RGPIN-2016-04872.  This research was supported in part by the National Science Foundation under Grant No. NSF PHY17-48958.  N.W was partially supported by Simons Foundation grant 585380.  We also gratefully acknowledge SageDays{@}ICERM and the Centre de Recherches Math\'emathiques.




\end{document}